\documentclass{article}

\usepackage{caption}
\usepackage{xspace}
\usepackage{enumitem}

\usepackage{tikz}
\usetikzlibrary{arrows, decorations.pathmorphing, positioning, backgrounds, fit, petri}
\usetikzlibrary{patterns}
\usepackage{xcolor}

\usepackage[section]{placeins}

\usepackage[latin1]{inputenc}
\usepackage{amsmath}
\usepackage{amsthm,amssymb,amsfonts}
\usepackage{graphicx}
\usepackage{verbatim}
\usepackage{color,cite}
\usepackage{latexsym}
\usepackage{lscape}
\usepackage[all,cmtip]{xy}
\usepackage{hyperref}
\theoremstyle{plain}
\newtheorem{theorem}{Theorem}[section]
\newtheorem{lemma}[theorem]{Lemma}

\newtheorem{proposition}[theorem]{Proposition}
\newtheorem{corollary}[theorem]{Corollary}
\newtheorem*{conjecture}{Conjecture}
\newtheorem{definition}[theorem]{Definition}

\newtheorem{remark}[theorem]{Remark}

\newtheorem*{conv}{Convention}

\newtheorem*{thmnonum}{Theorem}

\newcommand{\nc}{\newcommand}
\newcommand{\rnc}{\renewcommand}

\newcommand{\splin}{$\text{SL}_2(\mathbb{R})$}

\nc\bA{\mathbb{A}}
\nc\bB{\mathbb{B}}
\nc\bC{\mathbb{C}}
\nc\bD{\mathbb{D}}
\nc\bE{\mathbb{E}}
\nc\bF{\mathbb{F}}
\nc\bG{\mathbb{G}}
\nc\bH{\mathbb{H}}
\nc\bI{\mathbb{I}}
\nc{\bJ}{\mathbb{J}}
\nc\bK{\mathbb{K}}
\nc\bL{\mathbb{L}}
\nc\bM{\mathbb{M}}
\nc\bN{\mathbb{N}}
\nc\bO{\mathbb{O}}
\nc\bP{\mathbb{P}}
\nc\bQ{\mathbb{Q}}
\nc\bR{\mathbb{R}}
\nc\bS{\mathbb{S}}
\nc\bT{\mathbb{T}}
\nc\bU{\mathbb{U}}
\nc\bV{\mathbb{V}}
\nc\bW{\mathbb{W}}
\nc\bY{\mathbb{Y}}
\nc\bX{\mathbb{X}}
\nc\bZ{\mathbb{Z}}
\nc\cA{\mathcal{A}}
\nc\cB{\mathcal{B}}
\nc\cC{\mathcal{C}}
\nc\cD{\mathcal{D}}
\nc\cE{\mathcal{E}}
\nc\cF{\mathcal{F}}
\nc\cG{\mathcal{G}}
\nc\cH{\mathcal{H}}
\nc\cI{\mathcal{I}}
\nc{\cJ}{\mathcal{J}}
\nc\cK{\mathcal{K}}
\nc\cM{\mathcal{M}}
\nc\cN{\mathcal{N}}
\nc\cO{\mathcal{O}}
\nc\cP{\mathcal{P}}
\nc\cQ{\mathcal{Q}}
\nc\cS{\mathcal{S}}
\nc\cT{\mathcal{T}}
\nc\cU{\mathcal{U}}
\nc\cV{\mathcal{V}}
\nc\cW{\mathcal{W}}
\nc\cY{\mathcal{Y}}
\nc\cX{\mathcal{X}}
\nc\cZ{\mathcal{Z}}

\nc{\dmo}{\DeclareMathOperator}
\dmo{\Tw}{{\rm Twist}}
\dmo{\CP}{\rm Pres}
\rnc{\Re}{\operatorname{Re}}
\rnc{\Im}{\operatorname{Im}}
\rnc{\span}{\operatorname{span}}
\dmo{\rank}{rank}
\dmo{\End}{End}
\dmo{\Jac}{Jac}
\dmo{\Id}{Id}
\dmo{\lcm}{lcm}
\dmo{\Area}{Area}

\nc{\Tm}{Teichm\"uller\xspace}
\nc{\odd}{\cH^{\text{odd}}(4)}
\nc{\hyp}{\cH^{\text{hyp}}(4)}
\nc{\prym}{\widetilde{\mathcal{Q}}(3,-1^3)}
\nc{\G}{\mathrm{GL}^+(2,\bR)}
\nc{\prymodd}{\widetilde{\cQ}(4,-1^4)}
\nc{\coversodd}{\widetilde{\cH}_{(2,2)}^{\rm odd}(2)}
\nc{\covershyp}{\widetilde{\cH}_{(2,2)}^{\rm hyp}(2)}
\nc{\Odd}{\cH^{\rm odd}(2,2)}
\nc{\Hyp}{\cH^{\rm hyp}(2,2)}
\nc{\wth}{{\rm w}}
\nc{\eps}{\varepsilon}
\nc{\Twi}{\zeta}
\nc{\SL}{\mathrm{SL}}
\nc{\s}{\sigma}
\nc{\twist}{\mathrm{t}}

\title{Rank Two Affine Submanifolds in $\cH(2,2)$ and $\cH(3,1)$}
\author{David Aulicino\thanks{This material is based upon work supported by the ERC Starting Grant ``Quasiperiodic'' of Artur Avila, and later by the National Science Foundation under Award No. DMS - 1204414.} $\,$ and Duc-Manh Nguyen}
\date{}

\begin{document}

\maketitle 

\begin{abstract}
We classify all rank two affine manifolds in strata in genus three with two zeros.  This confirms a conjecture of Maryam Mirzakhani in these cases.  Several technical results are proven for all strata in genus three, with the hope that they may shed light on a complete classification of rank two manifolds in genus three.
\end{abstract}

\tableofcontents

\section{Introduction}

The action of $\G$ on the moduli space of translation surfaces has been studied extensively from both dynamical and geometric perspectives.  However, until recently, the nature of $\G$-orbit closures was very mysterious.  The recent breakthrough of \cite{EskinMirzakhaniInvariantMeas, EskinMirzakhaniMohammadiOrbitClosures} proved that all $\G$-orbit closures are affine invariant submanifolds, thereby demonstrating that these objects possess a very nice structure.  This result tremendously changed the nature of the field and facilitated several more breakthroughs on their properties \cite{AvilaEskinMollerForniBundle, Filip1, Filip2, WrightFieldofDef, WrightCylDef}.

In particular, \cite{WrightCylDef} introduced the (cylinder) rank of an affine manifold, defined as half the dimension of its tangent space projected to absolute cohomology.  Furthermore, \cite{WrightFieldofDef} introduced the field of affine definition of an affine manifold, which corresponds to to the smallest (real) field in which the coefficients of the local systems of linear equations defining the affine manifold lie.  If this field is $\bQ$, then the affine manifold is said to be \emph{arithmetic}.

It was well-known that the translation surfaces in an arithmetic rank one manifold are always given by branched coverings of tori (a formal proof of this can be found in \cite{AulicinoCompDegKZAIS}), so that the rank one manifold itself is a covering of a stratum of marked tori.  Maryam Mirzakhani conjectured that a similar result should hold for higher rank affine manifolds.

\begin{conjecture}[Mirzakhani]
Let $\cM$ be an affine manifold in a stratum of translation surfaces.  If $\text{rank}(\cM) \geq 2$ and $k(\cM) = \mathbb{Q}$, then $\cM$ is either the entire stratum or a covering of a stratum of Abelian or quadratic differentials.
\end{conjecture}

A rank two affine invariant submanifold in genus three is always arithmetic by \cite[Thm. 1.5]{WrightFieldofDef}.  

This conjecture was already known to be true in genus two due to the work of \cite{McMullenGenus2}.  The first results in higher genus were due to \cite{NguyenWright} and \cite{AulicinoNguyenWright}, and they affirmed this conjecture in genus three for the two connected components of the stratum $\cH(4)$: $\hyp$ and $\odd$, respectively.  We use these results as a starting point toward the goal of classifying all higher rank affine manifolds in genus three.  In this paper, we classify all rank two affine manifolds in strata in genus three with two zeros, and our results affirm Mirzakhani's conjecture.  Before stating our results, we recall the classification of connected components of strata in genus three with at most two zeros.

\begin{thmnonum}[\cite{KontsevichZorichConnComps}, Theorem 2]
\label{KZConnCompsThmHnm}
The connected components of strata in genus three with at most two zeros are given in the following diagram, where a line indicates that one connected component lies in the boundary of another.
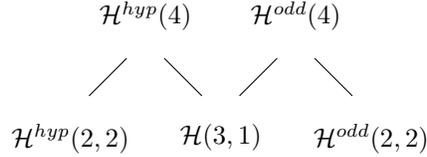
\begin{figure}[htb]
 \centering
\begin{tikzpicture}[scale=0.50]
\draw(-2,2) node[above] {$\mathcal{H}^{hyp}(4)$};
\draw(2,2) node[above] {$\mathcal{H}^{odd}(4)$};
\draw(-4,0) node[below] {$\mathcal{H}^{hyp}(2,2)$};
\draw(0,0) node[below] {$\mathcal{H}(3,1)$};
\draw(4,0) node[below] {$\mathcal{H}^{odd}(2,2)$};
\draw (-3.5,.5)--(-2.5,1.5);
\draw (-.5,.5)--(-1.5,1.5);
\draw (.5,.5)--(1.5,1.5);
\draw (3.5,.5)--(2.5,1.5);
\end{tikzpicture}
\end{figure}
\end{thmnonum}

In this paper, we prove

\begin{theorem}
\label{MainThmRank2Hnm}
Let $\cM$ be a rank two affine manifold in genus three.
\begin{itemize}
\item If $\cM \subset \Hyp$, then $\cM = \widetilde{\mathcal{Q}}(1^2,-1^2)=\covershyp$.
\item If $\cM \subset \Odd$, then $\cM = \prymodd$ or $\cM = \coversodd \subset \prymodd$.
\end{itemize}
Here $\covershyp$ and $\coversodd$ are the loci of unramified double covers of Abelian differentials in $\cH(2)$ in the components $\Hyp$ and $\Odd$, respectively.  They are both $4$-complex dimensional. The loci $\widetilde{\mathcal{Q}}(1^2,-1^2)$ and $\prymodd$ consist respectively of the standard orienting double covers of quadratic differentials in the strata $\mathcal{Q}(1^2,-1^2)$ and $\mathcal{Q}(4,-1^4)$.  They have complex dimensions four and five, respectively. It turns out that $\widetilde{\cQ}(1^2,-1^2)$ and $\covershyp$ coincide.
\medskip

\noindent There are no rank two affine manifolds in $\cH(3,1)$.
\end{theorem}

\begin{remark}
In the above theorem, if $p$ denotes the canonical projection from relative to absolute cohomology, then the kernel of $p$ is $0$-dimensional for all of the $4$-dimensional loci above, and it is $1$-dimensional for all of the $5$-dimensional loci.  The dimension of the kernel of $p$ is commonly known as the ``dimension of REL.''
\end{remark}

Theorem \ref{MainThmRank2Hnm} summarizes Theorems \ref{H31Rank2Class}, \ref{thm:H22hypClass}, Lemma \ref{lm:H2:in:Hyp22}, and Theorem \ref{thm:H22:odd:4cyl}.

\begin{remark}\hfill
\begin{itemize}
\item[$\bullet$] It follows from \cite{LanneauComponents} that $\widetilde{\cQ}(1^2,-1^2)$ and $\prymodd$ are connected. Theorem~\ref{MainThmRank2Hnm} implies in particular that  $\coversodd$ is connected.

\item[$\bullet$] Let $S$ be a (compact) Riemann surface of genus two. An unramified double cover of $S$ corresponds to a (normal) subgroup of index two of $\pi_1(S)$, which can be viewed as the kernel of a non-zero morphism from $\pi_1(S)$ to $\bZ/2\bZ$. Let $x_0$ be a Weierstrass point of $S$, and  ${\rm Mod}(S,x_0)$ denote the group of isotopy classes of homeomorphisms of $S$ fixing $x_0$. The existence of two components of unramified double covers of surfaces in $\cH(2)$ is in accordance with the fact that the action of ${\rm Mod}(S,x_0)$ on $H^1(S,\bZ/2\bZ)\setminus \{0\}$ has two orbits (see~\cite[Lem. 3.3]{NguyenTopH2}).
\end{itemize}
\end{remark}

We recall the main results of \cite{NguyenWright, AulicinoNguyenWright}.

\begin{theorem}\cite{NguyenWright, AulicinoNguyenWright}
\label{NWANWThm}
If $\cM$ is a rank two affine manifold in $\cH(4)$, then $\cM = \prym \subset \odd$. Here $\prym$ is the set of standard double covers of quadratic differentials in $\cQ(3, -1^3)$.
\end{theorem}

Following the diagram of the strata of Abelian differentials given in the theorem by Kontsevich-Zorich, we summarize all of the known classifications of rank two affine manifolds in genus three in Figure \ref{Rank2HnmFig}.

\begin{figure}[htb]
 \centering
\begin{tikzpicture}[scale=0.50]
\draw(-2,2) node[above] {-};
\draw(2,2) node[above] {$\tilde{\mathcal{Q}}(3,-1^3)$};
\draw(-5,0) node[below] {$\covershyp = \tilde{\mathcal{Q}}(1^2,-1^2)$};
\draw(0,0) node[below] {-};
\draw(5,0) node[below] {$\tilde{\mathcal{Q}}(4,-1^4) \supset \coversodd$};
\draw (-3.5,.5)--(-2.5,1.5);
\draw (-.5,.5)--(-1.5,1.5);
\draw (.5,.5)--(1.5,1.5);
\draw (3.5,.5)--(2.5,1.5);
\end{tikzpicture}
 \caption{Rank Two Affine Manifolds in Genus Three}
\label{Rank2HnmFig}
\end{figure}
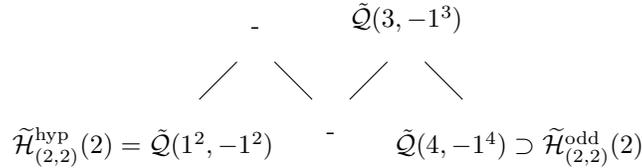

Though the primary goal of this paper is to prove a theorem about strata of translation surfaces in genus three with exactly two cone points, various results in this paper are established in various degrees of generality with the hope that they will lead to a classification of all higher rank affine manifolds in genus three.  There are several key techniques that are new to this paper that have not been used in the previous classifications of affine manifolds mentioned above.  We highlight the two general techniques concerning affine invariant submanifolds, and then we highlight the strategy to approach the overwhelming combinatorics of the problem.

\

\noindent \textbf{Affine Manifold Techniques:} First, both \cite{NguyenWright} and \cite{AulicinoNguyenWright} heavily relied on the use of \cite{SmillieWeissMinSets} to find a horizontally periodic translation surface in every orbit closure.  However, this was not sufficient for our purposes in this paper because when we find a translation surface with a ``partially periodic'' foliation (in the sense that there are horizontal cylinders, but they do not fill the surface) and some desirable properties, then we would like to preserve these properties while moving to a translation surface that is horizontally periodic.  Unfortunately, \cite{SmillieWeissMinSets} does not directly facilitate this because it requires flowing by the horocycle flow for an unspecified amount of time resulting in a loss of certain crucial information.

To remedy this problem, we use the fact that the affine invariant submanifold $\cM$ is arithmetic, which implies that square-tiled surfaces are dense in $\cM$.  This allows us to perturb in an arbitrarily small neighborhood in order to find a square-tiled surface with the desired properties (cf. Proposition \ref{prop:SimpCylPres}). This provides an additional benefit because square-tiled surfaces are Veech surfaces.  In particular, this implies that every direction that admits a saddle connection is periodic.  This tool will be used throughout the paper, cf. Lemma \ref{3PlusCyls}, Proposition \ref{prop:3III:imply:4cyls}, Proposition \ref{prop:C3I:loop:in:c1:c2}.\footnote{This list is not exhaustive.}

The second general technique is to use degenerations of translation surfaces.  Degenerating a translation surface in $\cM$ by collapsing an equivalence class of cylinders so that exactly two zeros collide, but no closed curve pinches, results in a translation surface that lies in an affine manifold of equal rank in a lower dimensional stratum, cf. Proposition \ref{prop:RankkImpRankkBd}.  This allows for a significant departure from the techniques of \cite{NguyenWright} and \cite{AulicinoNguyenWright}, which did not at all rely on any classification result in lower genus.  On the other hand, the present work is built upon the results in $\cH(4)$ by degenerating to $\cH(4)$ and drawing conclusions from this degeneration, cf. Theorem \ref{H31Rank2Class}, Lemmas \ref{Hhyp22UniqueDiagRed}, \ref{lm:Odd:4IOA} etc.

\

\noindent \textbf{Combinatorial Strategy:} The general philosophy behind this work falls perfectly in line with that of \cite{NguyenWright} and \cite{AulicinoNguyenWright}.  In both of those works, all of the cylinder diagrams were enumerated and studied in depth until the final result emerged.  In each of those works, the combinatorics were manageable in the sense that there were very few cylinder diagrams, they could be enumerated by hand, and the desired conclusions could be derived with reasonable effort.  On the other hand, a brute force approach to the classification problem in genus three by studying individual cylinder diagrams is hopeless.  There are over nine pages of cylinder diagrams in genus three with two or more cylinders!  While a general argument for the $2$-cylinder diagrams is relatively easy, the situation is far more complicated for $3$-cylinder diagrams.

After separating the cylinder diagrams by the number of cylinders, we make an additional separation of $n$-cylinder diagrams by the topology of the degenerate surfaces that are seen by pinching the core curves of every cylinder in the $n$-cylinder diagram, cf. Lemmas \ref{3CylDeg} and \ref{lm:4CylDeg}.

For $3$-cylinder diagrams, we have three cases (see Lemma~\ref{3CylDeg}).  Two of the cases are relatively easy to deal with in order to obtain surfaces with more horizontal cylinders.  The most complicated case of $3$-cylinder diagrams, cf. Case 3.I), is when the three core curves of the cylinders span a Lagrangian subspace of homology. Most of the $3$-cylinder diagrams fall into this case.  Of the numerous cylinder diagrams, we single out two properties that categorize all except one of the $3$-cylinder diagrams in this case.  Then we prove that from these properties, we can always produce a $4$-cylinder diagram, and this can also be accomplished in the exceptional case, cf. Section \ref{ExceptionalCaseSect}.

Finally, we enumerate all of the $4$-cylinder diagrams, of which there are not many, and prove the desired result.  We remark that the combinatorics in this paper can be done entirely without the use of a computer program such as Sage.  It suffices to study the dual graph of a periodic translation surface to enumerate all of the desired cylinder diagrams, cf. Section~\ref{sec:intro:DG} and Section~\ref{sec:list:4cyl:diag}.

\bigskip

\noindent {\bf Acknowledgments:} The authors warmly thank Alex Wright for his interest and for sharing with them numerous enlightening ideas and stimulating discussions, from which some of the lemmas in this paper were drawn.   They are also grateful to him for his helpful comments on an earlier version of this paper.  The authors are grateful to Mathematisches Forschungsinstitut Oberwolfach and the Universit\'e de Bordeaux whose hospitality has enabled the collaboration that led to this work.  Both authors thank Erwan Lanneau for helpful discussions, and Vincent Delecroix for providing the list of cylinder diagrams, which inspired the strategy of the paper. They thank the referee for helpful advice that improved the exposition of this paper.


\section{Preliminaries and General Results}

\subsection{Preliminaries}

\noindent \textbf{Strata of Translation Surfaces}: A \emph{translation surface} $M = (X, \omega)$ is a Riemann surface $X$ of genus $g \geq 2$ carrying an Abelian differential $\omega$.  The Abelian differential imposes a flat geometry on the Riemann surface up to a finite number of points $\Sigma \subset X$ corresponding to zeros of $\omega$.  Let each zero in $\Sigma$ have order $k_i$, so that $\kappa = (k_1, \ldots, k_n)$ is a partition of $2g-2$, the total order of the zeros of an Abelian differential.  Let $\cH(\kappa)$ denote a \emph{stratum of translation surfaces} of genus $g$ translation surfaces carrying Abelian differentials with zeros of orders specified by $\kappa$.

\

\noindent \textbf{$\G$ Action}: There is an action by $\G$ on strata of Abelian differentials given by embedding a translation surface in the plane as a collection of polygons and acting on the plane by $\G$.

\

\noindent \textbf{Period Coordinates}: Strata of translation surfaces admit a natural local coordinate system given as follows: Let $\{\gamma_i\} \subset H_1(X, \Sigma, \mathbb{Z})$ be a basis for relative homology.  Consider the map
$$\Phi: (X, \omega) \mapsto \left( \int_{\gamma_1}\omega, \ldots, \int_{\gamma_d} \omega \right).$$
This map is called the {\em period mapping}, it identifies a neighborhood of $(X, \omega)$ in its stratum with an open subset of $H^1(X, \Sigma, \mathbb{C})$. Consequently, $\Phi$ defines a coordinate system of $\cH(\kappa)$, the coordinate changes of this system are given by matrices in $\mathrm{GL}(d,\bZ)$, where $d=\dim H^1(X,\Sigma,\bC)$. In particular, $\cH(\kappa)$ admits a complex affine structure.

\

\noindent \textbf{Affine Manifolds} (see \cite{EskinMirzakhaniInvariantMeas}): An \emph{affine invariant submanifold} of $\cH(\kappa)$ is an immersed submanifold $\cM$ that is invariant under the action by $\G$, and admits a locally linear structure in period coordinates.  Throughout this paper, we abbreviate it to \emph{affine manifold}.  By \cite{EskinMirzakhaniInvariantMeas, EskinMirzakhaniMohammadiOrbitClosures}, every  $\G$-orbit closure is an affine manifold.\footnote{In fact, $\G$~orbit closures are not manifolds, but after passing to a finite cover, which we do throughout this paper, they become immersed submanifolds.  In particular, they may have finitely many self-intersections.  See \cite[Def. 1.1]{EskinMirzakhaniMohammadiOrbitClosures} for details.}  Moreover, if $\cM$ is an affine submanifold, and $\cM_{(1)}$ is the subset of $\cM$ consisting of surfaces of area one, then $\cM_{(1)}$ carries a finite ergodic measure for the action of \splin which is induced by the  Lebesgue measure on $\cM$.  
Since the work of \cite{EskinMirzakhaniInvariantMeas, EskinMirzakhaniMohammadiOrbitClosures}, affine manifolds have been shown to have a lot of additional structure.  It was shown in \cite{Filip2} that $\cM$ is a quasi-projective algebraic subvariety of $\cH(\kappa)$.

Let $T_M(\cM) \subset H^1(X, \Sigma, \bC)$ denote the tangent space to $\cM$ at $M$.  The tangent space behaves well under the tensor product with $\bC$.  Precisely speaking, there is a decomposition of the tangent space such that
$$T_M^{\bR}(\cM) \otimes \bC \subset H^1(X, \Sigma, \bR) \otimes \bC,$$
so that $T_M^{\bR}(\cM) \subset H^1(X, \Sigma, \bR)$.

\

\noindent \textbf{Field of Affine Definition}: Let $k(\cM)$ be the smallest field in which the coefficients of the linear equations defining $\cM$ lie.  It was proven in \cite{WrightFieldofDef}, that $k(\cM) \subset \mathbb{R}$ is a number field of degree at most $g$ over the rationals.  We call an affine manifold $\cM$ \emph{arithmetic} if $k(\cM) = \bQ$.

\

\noindent \textbf{Rank of an Affine Manifold}: Let $p: H^1(X, \Sigma,\bR) \rightarrow H^1(X,\bR)$ be the canonical projection from relative to absolute cohomology.  It was proven in \cite{AvilaEskinMollerForniBundle} that $p(T_M^{\mathbb{R}}(\cM))$ is symplectic, so in particular it is even dimensional.  In \cite{WrightCylDef}, the \emph{cylinder rank}, or \emph{rank} for short, of an affine manifold was defined to be

$$\text{rank}(\cM) = \frac{1}{2}\dim_{\mathbb{R}}p(T_M^{\mathbb{R}}(\cM)).$$

\

\noindent \textbf{Flat Structure}:  A trajectory beginning and ending at a cone point of $M$ is called a \emph{saddle connection}. A periodic (closed) trajectory on a translation surface $M$ that does not pass through the cone points of $M$ determines a cylinder $C$ of parallel periodic trajectories on $M$.  The \emph{height} $h$ of $C$ is the length of the line segment orthogonal to the core curve of the cylinder.  The \emph{width} $w$ is defined to be the circumference of $C$, and the modulus of $C$ is the quantity $h/w$.  By definition $C$ has two boundary components each of which consists of one or more saddle connections and freely homotopic to the core curves. We will say that $C$ is a {\em simple cylinder} if each of its boundary component consists of a single saddle connection.

Another important geometric parameter of a cylinder is its {\em twist}, which is somewhat more delicate to define.  Let $C$ be a horizontal cylinder.  To define the twist of $C$ we pick a saddle connection $a$ in the bottom boundary of $C$ and a saddle connection $b$ in the top boundary of $C$. Consider a saddle connection $s$ in $C$ joining the left endpoint of $a$ and the left endpoint of $b$.  Let $x+\imath y$ be the period of $s$. Note that we must have $y=h$. The twist of $C$ is then defined to be $x \mod w\bZ$.  Remark that there is a unique saddle connection $s_0$ in $C$ having the same endpoints as $s$ such that the real part $x_0$ of the period of $s_0$ satisfies $x_0 \in [0,w)$. By a slight abuse of notations, we sometimes say that the twist of $C$ is $x_0$.

A translation surface is said to be \emph{horizontally periodic} if every trajectory on $M$ in a fixed horizontal direction is closed. It is a theorem of \cite{SmillieWeissMinSets} that every affine manifold contains a horizontally periodic translation surface.  A \emph{cylinder diagram} is a representation of a translation surface $M$ expressing $M$ as a union of cylinders with marked saddle connections on the boundaries to indicate identifications.  Cylinder diagrams do not record the individual lengths of the cylinders or the saddle connections in their boundaries, so the equivalence is much weaker than it is for translation surfaces, i.e. two different translation surfaces can have the same cylinder diagram.  An $n$-cylinder diagram is a cylinder diagram with $n$ cylinders.

\

\noindent \textbf{$\cM$-parallelism and free cylinders}: Two cylinders $C_1$ and $C_2$ (resp. two saddle connections $\sigma_1$ and $\sigma_2$) are said to be \emph{$\cM$-parallel} if their core curves are parallel (resp. $\sigma_1$ and $\sigma_2$ are parallel), and they remain parallel under all local deformations in $\cM$. This notion is introduced in \cite{WrightCylDef}. The property of being $\cM$-parallel is an equivalence relation, so it makes sense to discuss an equivalence class of $\cM$-parallel cylinders (resp. saddle connections). The following two definitions can be found in \cite{WrightCylDef}.

\begin{definition}\label{def:free:cyl}
 A cylinder $C$ on $M$ is said to be {\em free} if it is $\cM$-parallel to no other cylinders on $M$. Equivalently, $C$ is free if its equivalence class is $\{C\}$.
\end{definition}

\begin{definition}
Let $a_s = \left( \begin{smallmatrix} 1 & 0 \\ 0 & e^s \end{smallmatrix} \right)$, and $u_t = \left( \begin{smallmatrix} 1 & t \\ 0 & 1 \end{smallmatrix} \right)$.  Let $\cC$ be an equivalence class of $\cM$-parallel cylinders.  Define the \emph{cylinder twist} $u_t^{\cC}$ (resp. \emph{cylinder stretch} $a_s^{\cC}$) to be the action of $u_t$ (resp. $a_s$) restricted to $\cC$ and leaving the complement of $\cC$ fixed.
\end{definition}

\begin{theorem}[\cite{WrightCylDef}, Thm. 5.1]
\label{thm:Wright:Cyl:Def}
Let $\cM$ be an affine manifold.  If $\cC$ is an equivalence class of $\cM$-parallel horizontal cylinders on $M \in \cM$, then for all $s,t \in \mathbb{R}$, $a_s^{\cC}(u_t^{\cC}(M)) \in \cM$.
\end{theorem}

\noindent \textbf{Twist and Preserving Space}: Let $M$ be a horizontally periodic translation surface in an affine manifold $\cM$.  The \emph{twist space} of $\cM$ at $M$, denoted $\Tw(M,\cM)$ is the largest subspace of $T_M^{\bR}\cM$ whose elements evaluate to zero on all horizontal saddle connections of $M$.  The \emph{cylinder preserving space} $\CP(M, \cM)$ is the largest subspace of $T_M^{\bR}\cM$ whose elements evaluate to zero on all core curves of horizontal cylinders of $M$.  Clearly, $\Tw(M, \cM) \subseteq \CP(M, \cM)$.

\begin{lemma}[\cite{WrightCylDef}, Lem. 8.6]
\label{lm:WrightTwistPresLem}
Let $M$ be a horizontally periodic translation surface in an affine manifold $\cM$.  If $\Tw(M, \cM) \not = \CP(M, \cM)$, then there exists a horizontally periodic translation surface in $\cM$ in a small neighborhood of $M$ with more horizontal cylinders than $M$.
\end{lemma}

Though it is not explicitly stated in the way it is phrased below, the following theorem follows from the proof of \cite[Thm. 1.10]{WrightCylDef} as well as \cite[Cor. 8.12]{WrightCylDef}.

\begin{theorem}[\cite{WrightCylDef}]
\label{thm:RankkImpkCyl}
Given an affine manifold $\cM$ of rank $k$, there exists a horizontally periodic translation surface $M \in \cM$ such that $\Tw(M, \cM) = \CP(M, \cM)$, in which case the cylinders on $M$ are split into at least $k$ equivalence classes of $\cM$-parallel cylinders, and the horizontal core curves of the cylinders on $M$ span a subspace of $T_M(\cM)^*$ of dimension $k$.

No set of core curves of parallel cylinders on a translation surface $M \in \cM$ may span a subspace of $T_M(\cM)^*$ of dimension greater than $k$.
\end{theorem}

\noindent \textbf{Cylinder Proportions}:  Let $\cC$ be an equivalence class of $\cM$-parallel cylinders on a translation surface $M \in \cM$.  Let $X$ be a cylinder on $M$ (whose core curve is not necessarily parallel to those of the cylinders in $\cC$).  Define the \emph{cylinder proportion} of $X$ in $\cC$ to be
$$P(X, \cC) = \frac{\text{Area}(X \cap (\cup_{C \in \cC} C))}{\text{Area}(X)}.$$

\begin{proposition}[\cite{NguyenWright}]
\label{CylinderPropProp}
Let $X$ and $Y$ be $\cM$-parallel cylinders on a translation surface $M \in \cM$.  Let $\cC$ be an equivalence class of $\cM$-parallel cylinders on $M$.  Then $P(X, \cC) = P(Y, \cC)$.
\end{proposition}

\subsection{Dual Graphs}\label{sec:intro:DG}

To each cylinder decomposition we can associate an undirected graph as follows: the set of vertices  of the graph is in one-to-one correspondence  with the set of cylinders, each edge of the graph corresponds to a horizontal saddle connection, and the vertices that are connected by an edge correspond to cylinders that contain the corresponding saddle connection in their boundary. In particular, if there is a saddle connection that is contained in both the top and bottom boundary of a cylinder, then there is a loop at the corresponding vertex in the dual graph.

\medskip

We can now list some elementary properties of dual graphs
\begin{itemize}
\item[(1)] Dual graphs are connected.

\item[(2)] The number of edges of a dual graph depends only on the stratum, this is because the maximal number of saddle connections in a fixed direction is completely determined by the total angles at the singularities.  For $\cH(4)$ this number is five, and for $\cH(m,n), \, m+n=4$, this number is six.

\item[(3)] The valency of each vertex is at least two.

\item[(4)] In general, one cannot retrieve the cylinder diagram from the dual graph alone. To do this one needs to specify

\begin{itemize}
\item[$\bullet$] an orientation for each edge, the origin of an oriented edge  represents the cylinder that contains the corresponding saddle connection in the bottom border, and the head represents the cylinder that contains the corresponding saddle connection in the top border, and

\item[$\bullet$] a cyclic ordering on the set of outgoing (resp. incoming) rays at each vertex.
\end{itemize}
In what follows, we will call a dual graph with an orientation for each edge a {\em directed dual graph}, and  a dual graph with orientation for each edge, and cyclic orderings for the rays at each vertex a {\em complete dual graph}. It is worth noticing that graphs with cyclic order of edges at every vertex are also called {\em fat graphs} or {\em ribbon graphs} in the literature (see {\em e.g.} \cite[$\S$ 2]{Kontsevich92}).

\item[(5)] One can put a positive weight on each edge of a directed dual graph such that the following condition is satisfied: at each vertex, the sum of the weights of the incoming rays equals the sum of weights of the outgoing rays.

\item[(6)] The following configurations are forbidden in a dual graph.

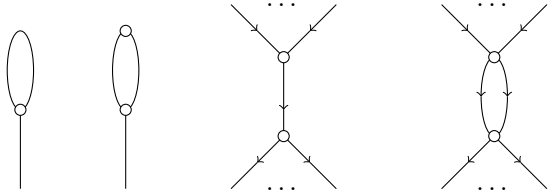
\begin{figure}[htb]
\centering
\begin{minipage}[t]{0.6\linewidth}
\centering
\begin{tikzpicture}[scale=0.7]
\draw (0,1.5) arc (-90:270: 0.25 and 0.75);
\draw (2,1.5) arc (-90:270: 0.25 and 0.75);

\draw (0,1.5) -- (0,0) (2,1.5) -- (2,0);

 \foreach \x in {(4,3.5), (5,1), (8,3.5), (9,1)} \draw[thin, ->]  \x -- +(0.5,-0.5);
 \foreach \x in {(4.5,3), (5.5,0.5), (8.5,3), (9.5,0.5)} \draw[thin] \x -- +(0.5,-0.5);

 \foreach \x in {(5,1), (6,3.5), (9,1), (10,3.5)} \draw[thin, ->]  \x -- +(-0.5,-0.5);
 \foreach \x in {(4,0), (5,2.5), (8,0), (9,2.5)} \draw[thin] \x -- +(0.5,0.5);

 \draw[thin, ->] (5,2.5) -- (5,1.5); \draw[thin] (5,1.5) -- (5,1);
 \draw[thin, <-] (9.25,1.75) arc (0:90: 0.25 and 0.75); \draw[thin] (9,1) arc (-90:0: 0.25 and 0.75);
 \draw[thin, ->] (9,2.5) arc (90:180:0.25 and 0.75); \draw[thin] (8.75,1.75) arc (180:270: 0.25 and 0.75);

 \draw (5,3.5) node{$\dots$} (5,0)  node {$\dots$} (9,3.5) node {$\dots$} (9,0) node {$\dots$};

 \foreach \x in {(0,1.5), (2,1.5), (2,3), (5,1), (5,2.5), (9,1), (9,2.5)} \filldraw[fill=white] \x circle (3pt);

\end{tikzpicture}
\end{minipage}
\caption{Forbidden configurations in a dual graph.}
\label{fig:forbid:config}
\end{figure}
\end{itemize}

The first two configurations (on the left) cannot occur in a dual graph since it is impossible to weight the (oriented) edges so that every vertex satisfies the condition that the sum of weights on the outgoing rays equals the sum of weights on the incoming rays. The last two configurations (on the right) correspond to two cylinders $C_1,C_2$ such that  the bottom side of $C_1$  equals the top side of $C_2$. If the bottom side of $C_1$ consists of one or two saddle connections, then this would imply that $C_1$ and $C_2$ are actually included in a larger cylinder.

\subsection{General Results on Arithmetic Affine Submanifolds}

The following lemma is well known to experts, and the interested reader can find a formal proof in \cite{AulicinoCompDegKZAIS}.

\begin{lemma}
\label{SqTiledDense}
Let $\cM$ be an affine manifold with rational affine field of definition.  Then $\cM$ contains a dense set of square-tiled surfaces.
\end{lemma}

\begin{definition}
A cylinder is called \emph{simple} if each boundary consists of exactly one saddle connection.
\end{definition}

\begin{proposition}
\label{prop:SimpCylPres}
Let $\cM$ be an affine manifold with rational affine field of definition.  Let $M$ be a surface in  $\cM$, $\{\sigma_i | i \in I\}$ be the set of  horizontal saddle connections, and $\{C_j | j \in J\}$  be the set of horizontal cylinders of $M$. Then there exists a square-tiled surface $M' \in \cM$ close to $M$  such that all of the saddle connections $\{\sigma_i | i \in I\}$  persist in $M'$ and  are also horizontal saddle connections in $M'$.

It follows that  $\{C_j | j \in J\}$ are also horizontal cylinders in $M'$. In particular, if $C_j$ is a simple cylinder in $M'$, it is also a simple cylinder in $M'$.
\end{proposition}

\begin{proof}
For all $\varepsilon > 0$, consider the set  $\tilde{U}_{\varepsilon} \subset \cH(\kappa)$ of radius $\varepsilon$ about $M$, where distance is measured in period coordinates.  We can choose $\varepsilon$ small enough so that all of the saddle connections $\{\sigma_i | i \in I\}$ persist on all surfaces $M' \in \tilde{U}_\varepsilon$. Set $U_\eps:=\tilde{U}_\eps \cap \cM$. We identify  $\tilde{U}_\eps$ with an open subset of $H^1(M,\Sigma,\bR+\imath\bR)\simeq \bR^{2d}$, and $U_\eps$ with an open subset of the linear subspace $T:=T_M\cM$. By assumption, $T$ is defined by linear equations with rational coefficients.

Let $z_i=x_i+\imath y_i$ be the period of $\sigma_i$. Since $\sigma_i$ is horizontal, we have $y_i = 0, \, \forall i\in I$. Define the linear subspace
$$
T':=(\cap_{i\in I}\{y_i=0\})\cap T \subset H^1(M,\Sigma,\bR+\imath\bR).
$$
Consider the set $V_\varepsilon:=\tilde{U}_\varepsilon\cap T'$. By definition, $V_\varepsilon$ is an open subset of the linear subspace $T'$, which is defined only by (linear) equations with rational coefficients. It may happen that $T'=\{0\}$, but by assumption, $M \in V_\varepsilon$, thus $\dim_\bR T' >0$, and  $V_\varepsilon \neq \varnothing$.  Therefore, one can find a point in $V_\varepsilon$ with rational coordinates. This point represents a square-tiled surface $M'$ for which $\{\sigma_i | i \in I\}$ are horizontal saddle connections.

Since the core curves of $\{C_j | j \in J \}$ are freely homotopic to a concatenation of some saddle connections in $\{\sigma_i | i \in I\}$, it follows that for all $j \in J$, $C_j$ persists in $M'$ and is also horizontal. In particular, if $C_j$ is a simple cylinder in $M$, then it is also a simple cylinder in $M'$.
\end{proof}


The following lemma is a special case of \cite[Lem. 6.1]{AulicinoNguyenWright}, we give the proof of this special case here for the convenience of the reader.

\begin{lemma}\label{lm:f:sim:cyl}
Let $C_1$ and $C_2$ be horizontal cylinders on a translation surface $M$.  Assume that $C_1$ is a simple cylinder that is only adjacent to $C_2$. If $C_1$ is free, then so is $C_2$.
\end{lemma}

\begin{proof}
After appropriately twisting $C_1$ and the equivalence class of cylinders $\cM$-parallel to $C_2$, there is a vertical cylinder $D$ intersecting only $C_1$ and $C_2$, such that $C_1 \subset \overline{D}$. We claim that $D$ is free. Indeed, let $D'$ be another cylinder $\cM$-parallel to $D$. It follows from Proposition \ref{CylinderPropProp} that
$$
P(D',C_1)=P(D,C_1) >0,
$$
\noindent which implies that $D'\cap C_1\neq \varnothing$. But by definition, $C_1$ is entirely contained in $D$, thus we have a contradiction, which means that $D$ is alone in its equivalence class. Applying again Proposition~\ref{CylinderPropProp} to the intersection of $D$ with  the equivalence class of $C_2$, we see that  any cylinder $\cM$-parallel to $C_2$ must intersect $D$. But $D$ intersects only two horizontal cylinders, namely $C_1$ and $C_2$. Since $C_1$ is not $\cM$-parallel to $C_2$, we conclude that $C_2$ is also free.
\end{proof}

\begin{lemma}\label{lm:n:f:sim:cyl}
Let $\cM$ be an affine manifold defined over $\bQ$ with rank at least two.  Let $M \in \cM$ admit a horizontal cylinder decomposition $C_1, \ldots, C_k$.  Assume that the core curves of $C_1,\dots,C_k$ span a subspace of dimension at least two in $(T_M^{\bR}\cM)^\ast$. If $C_1$ is a  simple cylinder and $C_1$ is only adjacent to $C_2$, then $C_1$ and $C_2$ cannot belong to the same equivalence class.
\end{lemma}
\begin{proof}
By contradiction, suppose that $C_1$ and $C_2$ are $\cM$-parallel.  After twisting $C_2$ by applying the horocycle flow to $M$, there is a pair of vertical saddle connections in $C_2$ that cut out a rectangle $R$ whose horizontal sides form the boundary of $C_1$. Since $\cM$ is defined over $\bQ$, we can suppose that all of the periods of $M$ are rational by Proposition~\ref{prop:SimpCylPres}. In particular, we can suppose that the ratio of the twist of $C_1$ and its circumference is rational. It follows that there exists a vertical cylinder $D$  such that $\overline{D}=\overline{R}\cup \overline{C_1}$. We see that there is a single cylinder $D$ that fills $C_1$ because $\overline{R}\cup\overline{C_1}$ can be regarded as a slit torus as in \cite{AulicinoNguyenWright}. Since the slit is vertical in this case, there is only one vertical cylinder if the vertical flow is periodic.

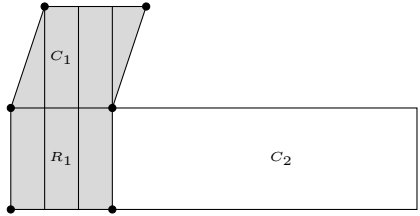
\begin{figure}[htb]
\centering
\begin{tikzpicture}[scale=0.45]
\fill[gray!30] (0,0) -- (3,0) -- (3,3) -- (4,6) -- (1,6) -- (0,3) -- cycle;

\draw[thin] (0,0) -- (0,3) -- (1,6) --  (4,6) -- (3,3) -- (3,0);
\draw[thin] (0,0) -- (12,0) -- (12,3) -- (0,3);
\draw[thin] (1,6) -- (1,0) (2,6) -- (2,0) (3,6) -- (3,3);
\foreach \x in {(0,0), (3,0), (3,3), (4,6), (1,6), (0,3)} \filldraw[fill=black] \x circle (3pt);

\draw (1.5,4.5) node {\tiny $C_1$} (1.5,1.5) node {\tiny $R_1$} (8,1.5) node {\tiny $C_2$};
\end{tikzpicture}

\caption{The shaded region realizes a slit torus containing $C_1$. When the ratio of the twist and the circumference of $C_1$ is rational, the vertical flow in this slit torus is periodic. We then have a single vertical cylinder $D$ that fills out this region.}
\label{fig:s:cyl:adj:one:other:cyl}
\end{figure}

 Let $\cC$ be the equivalence class of $C_1$ and $\cD$ the equivalence class of $D$. We have
 $$
 P(C_1,\cD)= 1 \Rightarrow P(C_i,\cD)=1, \quad \forall C_i \in \cC,
 $$
 \noindent which means that $C_i$ is filled by  vertical cylinders $\cM$-parallel to $D$.  On the other hand, we also have $P(D,\cC)=1$ since $D \subset \overline{C_1}\cup \overline{C_2}$. Therefore, any vertical cylinder in $\cD$ intersects only cylinders in $\cC$. It follows that a cylinder in $\cC$ can only be adjacent to cylinders in $\cC$. By assumption, we know that there are at least two equivalence classes of horizontal cylinders. Since the surface is connected, there must be a cylinder in $\cC$ which is adjacent to a cylinder in another equivalence class, and we get a contradiction.
\end{proof}

\begin{lemma}\label{lm:loop}
Let $\cM$ be an affine manifold with rank at least two.  Assume that $M \in \cM$ admits a cylinder decomposition in the horizontal direction such that there are  at least two equivalence classes of $\cM$-parallel cylinders. Let $C$ and $C'$ be two horizontal cylinders in $M$. Suppose that the top boundary of $C$ is contained in the bottom of cylinders that are not $\cM$-parallel to $C$, and  there is a saddle connection which is contained in both the top and bottom borders of $C'$. Then $C$ and $C'$ cannot belong to the same equivalence class.
\end{lemma}
\begin{proof}
Assume by contradiction that $C$ and $C'$ are $\cM$-parallel. Let us denote the equivalence class of $C$ by $\cC$. By assumption, there is a transverse cylinder $D'$ contained in $C'$ (we can suppose that $D'$ is vertical). Since $C$ is $\cM$-parallel to $C'$, there must be a vertical cylinder $D$ in the equivalence class of $D'$ that crosses $C$.  But by assumption,  any vertical cylinder crossing $C$ must cross cylinders not in $\cC$. Thus we have $P(D,\cC) < 1$, while $P(D',\cC)=1$, which is a contradiction.
\end{proof}

\begin{remark}
 Lemma~\ref{lm:loop} already appeared in \cite[Lem. 6.3]{AulicinoNguyenWright}.
\end{remark}

\begin{lemma}\label{lm:sim:cyl:2:cl}
Let $\cM$ be an affine manifold defined over $\bQ$ with rank at least two, and $M$ a translation surface in $\cM$.  Let $\cS:=\{C_i | i=1,\dots,s \}$ denote the set of simple horizontal cylinders of $M$. Then there exists a surface $M'\in \cM$, which is also decomposed into cylinders in the horizontal direction with at least $s$ simple cylinders and at least two equivalence classes.
\end{lemma}

\begin{proof}
We follow the argument of \cite[Lem. 8.6]{WrightCylDef} combined with the hypothesis that $\cM$ is defined over $\bQ$. Let ${\rm Pres}(M,\cM) \subset T^{\bR}_{M}\cM$ and ${\rm Twist}(M,\cM)\subset T^{\bR}_{M}\cM$ be as defined above.  Using Proposition~\ref{prop:SimpCylPres}, we can assume that $M$ is square-tiled.

Obviously we only have to consider the case where all of the horizontal cylinders of $M$ belong to the same equivalence class, which means that the core curves $c_1,\dots,c_k$ of these cylinders span a subspace of dimension one in $(T_M^{\bR}\cM)^\ast$. By definition we have
$$
{\rm Pres}(M,\cM) = \ker(c_1) \cap \dots \cap \ker(c_k)=\ker(c_1) \subset T^{\bR}_M\cM.
$$
Therefore, ${\rm Pres}(M,\cM)$ has codimension one. We know that $p({\rm Twist}(M,\cM))$ is an isotropic subspace of $p(T^{\bR}_M\cM)$ by \cite[Lem. 8.10]{WrightCylDef}, thus $p({\rm Twist}(X,\cM))$ has codimension at least two in $p(T^{\bR}_M\cM)$, which implies that ${\rm Twist}(M,\cM)$ has codimension at least two in $T^{\bR}_M\cM$. It follows that ${\rm Twist}(M,\cM) \subsetneq {\rm Pres}(M,\cM)$.

Let $\eta \in {\rm Pres}(M,\cM)\setminus {\rm Twist}(M,\cM)$ with all of its coordinates in $\bQ$, and small Euclidean norm. If $||\eta||$ is small enough, then $M'=M+\imath\eta$ is a well-defined surface in $\cM$ and all of the horizontal simple cylinders in $M$ remain simple in $M'$. Observe that all of the coordinates of $M'$ are rational.  Thus, $M'$ is a square tiled surface admitting a cylinder decomposition in the horizontal direction with more cylinders than $M$. 

 If there are two equivalence classes of horizontal cylinders on $M'$, then we are done.  Otherwise we can repeat the deformations along vectors in ${\rm Pres}(M',\cM)\setminus{\rm Twist}(M',\cM)$ to get more horizontal cylinders. However, a cylinder decomposition of any surface in genus $g$ cannot have more than $3g-3$ cylinders. Therefore, after finitely many steps we get a horizontally periodic surface in $\cM$ with th maximal number of horizontal cylinders and at least $s$ simple cylinders. If there is only one equivalence class, then the argument above shows that this procedure can be repeated to produce a surface with more horizontal cylinders, which is impossible. Therefore, we can conclude that there are at least two equivalence classes of horizontal cylinders.
\end{proof}

\begin{lemma}\label{lm:3:f:cyl}
Let $\cM$ be a rank two affine manifold in genus three.  Suppose that $M \in \cM$ is decomposed into three horizontal cylinders $C_1, C_2, C_3$, with core curves $c_1,c_2,c_3$, respectively.  If $\{c_1,c_2,c_3\}$ span a Lagrangian subspace of $H_1(M,\bZ)$, then $C_1,C_2,C_3$ cannot all be free.
\end{lemma}
\begin{proof}
This lemma is a consequence of Theorem~\ref{thm:Wright:Cyl:Def} and the fact that the projection of $T_M^{\bR}\cM$ in $H^1(M,\bR)$ is symplectic. The shearing of $C_i$ corresponds to a vector $\xi_i \in T^{\bR}_M\cM \subset H^1(M,\Sigma,\bZ)$ which maps to $c_i$ via the identification of $H^1(X,\Sigma, \bZ)$ with $H_1(M\setminus\Sigma, \bZ)$ (see \cite[Lemma 2.4, Remark 2.5]{WrightCylDef}). Since the projection from $H^1(M,\Sigma,\bZ)$ to $H^1(M,\bZ)$ can be also identified with the projection from $H_1(M\setminus\Sigma,\bZ)$ to $H_1(M,\bZ)$, the images of $\xi_1,\xi_2,\xi_3$ in $H^1(M,\bZ)$ span a $3$-dimensional isotropic space. But this is impossible since $p(T^{\bR}_M\cM)$ is a symplectic subspace of dimension four in $H^1(M,\bR)$.
\end{proof}

%

\subsection{Cylinder Collapsing}

Let $\cH_g$ denote  the  vector bundle of Abelian differentials over the moduli space $\cM_g$. Let $\cM$ be an affine submanifold of $\cH(\kappa) \subset \cH_g$.  Let $M$ be a surface in $\cM$ and $\cC=\{C_0,C_1,\dots,C_k\}$ be an equivalence class of $\cM$-parallel horizontal cylinders on $M$. We do not assume here that $M$ is horizontally periodic. The  deformations $\{a^{\cC}_t(M), t \in \bR\}$ of $M$ by stretching simultaneously the cylinders in $\cC$ define a path in $\cH(\kappa)$. If $a_t^{\cC}(M)$ admits a limit $M'$ in $\cH_g$ as $t \rightarrow -\infty$, then we will say that $M'$ is obtained from $M$ by {\em collapsing the equivalence class $\cC$}. In Figure~\ref{fig:ex:cyl:collapse}, we represent some limits of  cylinder collapsing deformations in $\cH_2=\cH(2)\sqcup\cH(1,1)$.

\begin{figure}[htb]
\begin{minipage}[t]{0.4\linewidth}
\centering
\begin{tikzpicture}[scale=0.3, inner sep=0.1mm, vertex/.style={circle, draw=black, fill=white, minimum size=1mm}]
\fill[green!30] (-12,5) -- (-10,3) -- (-1,3) -- (-3,5) -- cycle;
\draw (-12,8) -- (-12,5) -- (-10,3) -- (-5,3) -- (-5,0) -- (-1,0) -- (-1,3) -- (-3,5) -- (-7,5) --( -7,8)  -- cycle;
\draw (-10,3) -- (-7,5) -- (-5,3) -- (-3,5);
\draw (-12,5) -- (-7,5) (-5,3) -- (-1,3);

\node at (-10,4) [vertex] {\tiny $1$};
\node at (-7,4) [vertex] {\tiny $2$};
\node at (-5,4) [vertex] {\tiny $3$};
\node at (-3,4) [vertex] {\tiny $4$};

\draw (0.5,4) node {$\simeq$};

\fill[green!30] (4,8) -- (9,8) -- (7,10) -- cycle;
\fill[green!30] (2,3) -- (6,3) -- (9,5) -- (4,5) -- cycle;
\fill[green!30] (2,0) -- (4,-2) -- (6,0) -- cycle;

\draw (2,3) -- (2,0) -- (4,-2) -- (6,0) -- (6,3) -- (9,5) -- (9,8) -- (7,10) -- (4,8) -- (4,5) -- cycle;
\draw (4,8) -- (9,8) (9,5) -- (4,5) -- (6,3) -- (2,3) (2,0) -- (6,0);

\node at (4,4) [vertex] {\tiny $4$};
\node at (4,-1) [vertex] {\tiny $3$};
\node at (6,4) [vertex] {\tiny $1$};
\node at (7,9) [vertex] {\tiny $2$};

\foreach \x in {(-12,8), (-10,3), (-7,8), (-5,3), (-1,3), (2,3), (4,8),(4,-2), (6,3), (9,8)} \filldraw[fill=white] \x circle (4pt);
\foreach \x in {(-12,5), (-7,5) , (-5,0) ,( -3,5), (-1,0), (2,0), ( 4,5), (6,0), (7,10), (9,5)} \filldraw[fill=black] \x circle (4pt);

\foreach \x in {(3,4),(5,-1)} \draw \x +(0,-0.2) -- +(0,0.2);
\foreach \x in {(3,-1), (8,9)} \draw \x +(-0.1,-0.1) -- +(-0.1,0.3) +(0.1,-0.3) -- +(0.1,0.1);
\foreach \x in {(5.5,9),(7.5,4)} \draw \x +(-0.3,-0.4) -- +(-0.3,0) +(0,-0.2) -- +(0,0.2) +(0.3,0) -- +(0.3,0.4);

\draw[very thick, ->, >=stealth] (1,-2) -- (1,-5);

\draw (-3,-9) -- (-3,-12) -- (1,-12) -- (1,-9) -- (4,-9) -- (4,-6) -- (-1,-6) -- (-1,-9) -- cycle;
\draw (-1,-9) -- (1,-9);

\foreach \x in {(-3,-9), (-1,-6), (-1,-12), (1,-9), (4,-6)} \filldraw[fill=white] \x circle (4pt);
\foreach \x in {(-3,-12), (-1,-9), (1,-12), (2,-6), (4,-9)} \filldraw[fill=black] \x circle (4pt);

\foreach \x in {(-2,-9), (0,-12)} \draw \x +(0,-0.2) -- +(0,0.2);
\foreach \x in {(-2,-12), (3,-6)} \draw \x +(-0.1,-0.2) -- +(-0.1,0.2) +(0.1,-0.2) -- +(0.1,0.2);
\foreach \x in {(0.5,-6), (2.5,-9)} \draw \x +(-0.2,-0.2) -- +(-0.2,0.2) +(0,-0.2) -- +(0,0.2) +(0.2,-0.2) -- +(0.2,0.2);
\end{tikzpicture}

\end{minipage}
\begin{minipage}[t]{0.5\linewidth}
\flushright
\begin{tikzpicture}[scale=0.3]
\fill[green!30] (2,9) -- (2,7) -- (6,7) -- (6,9) -- cycle;
\draw (0,3) -- (4,3) --(4,0) -- (7,0) -- (7,3) -- (9,7) -- (6,7) -- (6,9) -- (2,9) -- (2,7) -- cycle;
\draw (2,7) -- (6,7) (4,3) -- (7,3);
\foreach \x in {(0,3), (2,9), (4,3), (6,9), (7,3)} \filldraw[fill=white] \x circle (4pt);
\foreach \x in {(2,7), (4,0), (6,7), (7,0), (9,7)} \filldraw[fill=black] \x circle (4pt);

\foreach \x in {(4,9),(2,3)} \draw \x +(0,-0.2) -- +(0,0.2);
\foreach \x in {(7.5,7),(5.5,0)} \draw \x +(-0.1,-0.2) -- +(-0.1,0.2) +(0.1,-0.2) -- +(0.1,0.2);

\draw[black, very thick, ->, >=stealth] (3,0) -- (3,-3);

\draw (0,-8) -- (4,-8) -- (4,-11) -- (7,-11) -- (7,-8) -- (9,-4) -- (2,-4) -- cycle;
\draw (4,-8) -- (7,-8);

\foreach \x in {(0,-8), (2,-4), (4,-8), (4,-11), (6,-4), (7,-8), (7,-11), (9,-4)} \filldraw[fill=black] \x circle (4pt);

\foreach \x in {(4,-4),(2,-8)} \draw \x +(0,-0.2) -- +(0,0.2);
\foreach \x in {(7.5,-4),(5.5,-11)} \draw \x +(-0.1,-0.2) -- +(-0.1,0.2) +(0.1,-0.2) -- +(0.1,0.2);

\end{tikzpicture}
\end{minipage}

 \caption{Collapsing $\cM$-parallel cylinders of a surface in $\cH(1,1)$, the collapsed cylinder is shaded. On the left, the limit remains in $\cH(1,1)$, on the right the limit belongs to $\cH(2)$.}
 \label{fig:ex:cyl:collapse}
\end{figure}
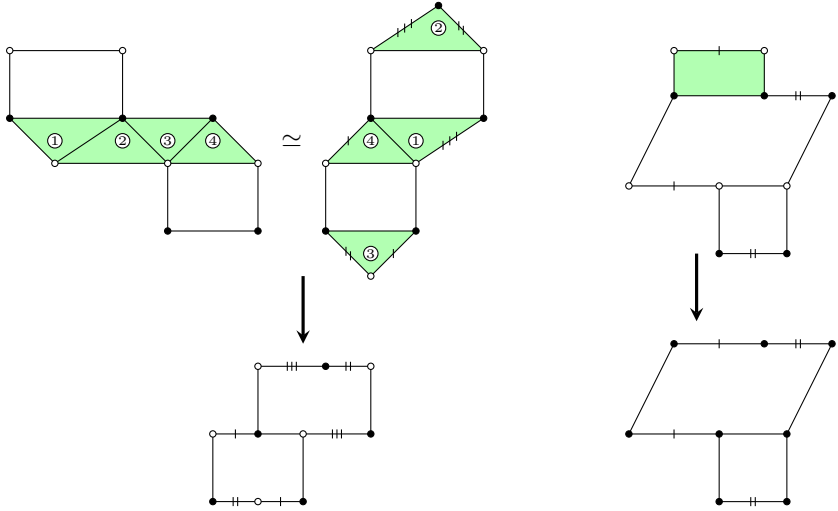

The following proposition is an important tool for our proof of the main theorem.

\begin{proposition}
\label{prop:RankkImpRankkBd}
Assume that
\begin{itemize}
 \item[(a)] The family $\cC$ does not fill $M$.

 \item[(b)] If $C$ and $C'$ are two horizontal cylinders ($C$ and $C'$ may be the same) such that there is a saddle connection that is contained in the bottom boundary of $C$ and in the top boundary of $C'$, then either $C \not\in \cC$ or  $C'\not\in  \cC$.

 \item[(c)] There exists a vertical saddle connection $\s$ in $C_0$ joining two distinct singularities, and any other saddle connection contained in one of the cylinders $C_i, \, i=0,\dots,k$, has non-zero real part.
\end{itemize}

Then as $t\rightarrow -\infty$, $a^{\cC}_t(M)$ converges to a surface $M'$ which belongs to another stratum $\cH(\kappa') \subset \cH_g$ with $|\kappa'| =|\kappa|-1$. Moreover, $M'$ is contained  in an affine invariant submanifold $\cM'$ of $\cH(\kappa')$ which has the same rank as $\cM$.
\end{proposition}

\begin{remark}\hfill
\begin{itemize}
\item[$\bullet$] There is an implicit assumption in Proposition \ref{prop:RankkImpRankkBd} that is $M$ has at least two singularities, or equivalently $|\kappa| \geq 2$.

\item[$\bullet$] The condition (b) on $\cC$ implies in particular that no cylinder in the family $\cC$ has a saddle connection contained in both of its top and bottom boundary.
\end{itemize}
\end{remark}

\begin{proof}
We first show that $a_t^{\cC}(M)$ admits a limit in $\cH(\kappa')$, with $|\kappa'| =|\kappa|-1$,  as $t \rightarrow -\infty$.  Let $\s'\neq \s$ be a saddle connection of $M$. There exists $\eps_1>0$ such that the length of any subsegment of $\s'$ outside of the union $\cup_{i=0}^k C_i$ is at least $\eps_1$. Condition (b) implies that if $\s'$ is contained in $\cup_{i=0}^k C_i$, then $\s'$ is actually contained in one of the cylinders $C_0,\dots,C_k$ because Condition (b) excludes the existence of adjacent cylinders in $\cC$.  By condition (c), there exists $\eps_2 >0$ such that if $\s'$ is contained in one of the cylinders $C_0,\dots,C_k$ and $\s'\neq \s$, then the real part ${\rm Re}(\s')$ of $\s'$ satisfies $|{\rm Re}(\s')| > \eps_2$. From these observations, we see that if $t <0$ and $|t|$ is large enough then any saddle connection $\s' \neq \s$ has length bounded away from zero in $a_t^{\cC}(M)$, where the bound depends on $M$ and not $t$.

\medskip

We can now apply the {\em collapsing a pair of zeros} procedure described in \cite[Sect. 8.2]{EskinMasurZorich} to $\s$ to get a surface $M'_t \in \cH(\kappa')$ with  $|\kappa'| =|\kappa|-1$. It is not difficult to check that as $t \rightarrow -\infty$, $M'_t$ converges to a surface $M' \in \cH(\kappa')$. It is a well known fact that the intersection of a neighborhood of $M'$ in $\cH_g$ with $\cH(\kappa)$ consists of surfaces obtained by a {\em breaking up a zero} construction applied to surfaces in a neighborhood of $M'$ in $\cH(\kappa')$ (see \cite[Sect. 5]{KontsevichZorichConnComps}).  Therefore, given any neighborhood $\cU$ of $M'$, we have $a_t^{\cC}(M) \in \cU$ for $t<0$ and $|t|$ large enough. We can then conclude that $M'$ is the limit in $\cH_g$ of the path $a_t^{\cC}(M)$ as $t \rightarrow -\infty$.

\medskip

Let $\overline{\cM}$ be the closure of $\cM$ in $\cH_g$. Let $\cM'$ be the intersection $\overline{\cM}\cap \cH(\kappa')$. By \cite{Filip2}, we know that $\cM$ is an algebraic subvariety of $\cH_g$, therefore $\overline{\cM}\setminus \cM$ has dimension strictly smaller than $\cM$.   In particular, we have $\dim_\bC\cM'< \dim_\bC\cM$. Note that $\cM'$ may have several components since there may be more than one way to degenerate 
surfaces in $\cH(\kappa)$ to surfaces in $\cH(\kappa')$.

\medskip

By definition, $\cM'$ is a closed $\mathrm{GL}^+(2,\bR)$-invariant subset of $\cH(\kappa')$. Thus each component of $\cM'$ must be an affine submanifold of $\cH(\kappa')$, and therefore  is  locally defined by linear equations with real coefficients. We can consider the tangent space of $\cM'$ at $M'$, denoted by $T_{M'}\cM'$, as the union of the tangent spaces of its irreducible components at $M'$. In particular, we have
$$
T_{M'}\cM'=T_{M'}^\bR\cM'\otimes \bC,
$$
\noindent where $T_{M'}^\bR\cM'$ is the union of finitely many linear subspaces of $H^1(M',\Sigma',\bR)$ with $\Sigma'$ being the set of singularities of $M'$.  We will show that there exists an irreducible component $\cM'_0$ of $\cM'$, containing $M'$, such that  $\dim_\bC \cM'_0=\dim_\bC\cM -1$, and ${\rm rank }\, \cM'_0={\rm rank }\, \cM$.

\medskip

Pick a basis $\cB$ of $H_1(M,\Sigma,\bZ)$ where $\sigma$ is an element of  $\cB$, and any other element of $\cB$ is either a saddle connection in the complement of $\cup_{0=1}^kC_i$, or a saddle connection in the closure of one of the cylinders $C_i, \, i=0,\dots,k$. A saddle connection in the latter case either crosses each core curve of $C_i$ once or is contained in the boundary of $C_i$.  Using $\cB$, we can identify $H^1(M,\Sigma;\bR)$ with  $(\bR)^{\cB}$. Set $V:=T^{\bR}_M\cM \subset H^1(M,\Sigma;\bR)$, and denote by $p: H^1(M,\Sigma;\bR) \rightarrow H^1(M,\bR)$ the canonical projection. Recall that we have $\dim_\bR p(V)=2\ell$, where $\ell$ is the rank of $\cM$.

\medskip

From Theorem~\ref{thm:Wright:Cyl:Def}, we can freely stretch and shear $\cC$ while keeping the rest of the surface unchanged, the new surfaces always belong to $\cM$. Let  $\xi:=(\xi(s))_{s \in \cB}$ be the tangent vector to the path in $\cM$ corresponding to the shearing operation of $\cC$.  We have $\xi(\sigma)=1$, and $\xi(s)=0$ for any $s$ which does not cross any cylinder in the family $\cC$.

\medskip

Using the period mapping, we can identify a neighborhood of $M$ with an open subset $U \subset V\otimes_\bR\bC$. Let $X$ be the image of $M$ by this identification. Recall that ${\rm Re}(X(\s))=0$, and for any other saddle connection $s$ contained in $\overline{C}_i$, we must have $|{\rm Re}(X(s))| > \eps_2$.

\medskip

Set $V':=\{v \in V | \ v(\sigma)=0\} \subset V$. Since $V'$ is a subspace of $V$ which is defined by a unique linear equation, $V'$ has codimension at most one.  Observe that $\xi \not\in V'$ since $\xi(\s)=1$, from which we derive that $V'$ is a proper linear subspace of codimension one in $V$. For any fixed $v \in V'$, and $\varepsilon >0$ small enough, the vector $X_{\eps v}:=X+\eps v$ belongs to $U$,  hence corresponds to a surface in $\cM$ close to $M$.  Since $v \in H^1(M,\Sigma;\bR)$,  the cylinders in the family $\cC$ persist in the surface $X_{\eps v}$.  Moreover, if $\eps$ is small enough, then for any saddle connection $s \neq \sigma$ which is contained in one of the cylinders $C_i$, we have $\mathrm{Re} (X_{\eps v}(s)) \neq 0$. It follows that the cylinder collapsing operation applied to the surface represented by $X_{\eps v}$ also yields a surface in $\cH(\kappa')$.

\medskip

Set $\cB':=\cB\setminus\{\sigma\}$. Recall that by the choice of $\cB$, any element of $\cB$ is either a saddle connection contained in the closure of one of the cylinders $C_1,\dots,C_k$, or disjoint from those cylinders. Let  $s$ be a saddle connection  in $\cB'$. If $s$ is disjoint from  $\cup_{i=0}^k \overline{C}_i$, then clearly $s$ persists in $M'$. If $s$ is contained in the closure of a cylinder $C_i$, as $t \rightarrow -\infty$, $s$ degenerates to the union of one or more  horizontal saddle connections in $M'$. Nevertheless, we claim that $\cB'$ still represents  a basis of $H_1(M',\Sigma';\bZ)$, where  $\Sigma'$ is the set of singularities of $M'$. To see this we first notice that the pairs $(X,\Sigma\cup \sigma)$ and $(X',\Sigma')$ are homotopy equivalent, thus $H_1(M',\Sigma';\bZ)$ can be identified with $H_1(M,\Sigma\cup\s; \bZ)$.  The natural projection $\rho: H_1(M,\Sigma;\bZ) \rightarrow H_1(M,\Sigma\cup\s; \bZ)$ is surjective. Since $\rho(\s)=0 \in H_1(M,\Sigma\cup \s; \bZ)$, $\cB'$ must be 
mapped to a basis of $H_1(M,\Sigma\cup\s;\bZ)$.

\medskip

Since a vector in $V'$ is uniquely determined by its evaluations on $\cB'$, the collapsing of the cylinders in $\cC$ provides us with an embedding of $V'$ into $T^{\bR}_{M'}\cM'\subset H^1(M',\Sigma';\bR)$.  Since $T^\bR_{M'}\cM'$ is the union of finitely many subspaces of $H^1(M',\Sigma',\bR)$, there must exist an irreducible component $\cM'_0$ of $\cM'$  containing $M'$ such that $T^\bR_{M'}\cM'_0$ contains $V'$. In particular, we must have $\dim_\bR T^\bR_{M'}\cM'_0 \geq \dim_\bR V' =\dim_\bR V -1$. It follows that  $\dim_\bC\cM'_0 \geq \dim_\bC \cM -1$. Since $\dim_\bC\cM'_0\leq \dim_\bC\cM-1$, we can conclude that $\dim_\bC \cM'_0=\dim_\bC\cM-1$, and $V'$ is isomorphic to $T^{\bR}_{M'}\cM'_0$.

\medskip

It remains to show that $\dim p(V')=\dim p(V)$. Note that in this case, we can identify $H^1(M',\bR)$ with $H^1(M,\bR)$, and $p(V')$ with $p(T^\bR_{M'}\cM'_0)$. Assume that $\dim p(V') < \dim p(V)$. Since $\cM'_0$ is an affine submanifold of $\cH(\kappa')$, $p(V')$ is a symplectic subspace of $H^1(M,\bR)$ by \cite{AvilaEskinMollerForniBundle}, therefore $\dim p(V') \leq \dim p(V)-2$. But we have $\dim V'=\dim V-1$, therefore $\dim p(V') \geq \dim p(V) -1$, and we get a contradiction. Thus, $\dim p(V')=\dim p(V)$, which means that $\cM'_0$ and $\cM$ have the same rank. The proof of the proposition is now complete.
\end{proof}

\begin{remark}
On first glance, Proposition \ref{prop:RankkImpRankkBd} may seem troubling because if there are two equivalence classes, we may only see one equivalence class of cylinders in the limit, which gives the impression of degenerating from rank two to rank one.  Indeed such degenerations will arise in Lemma \ref{H22hypDoubleCov}.  However, this apparent contradiction is only an illusion because the surface can be deformed prelimit before collapsing cylinders to produce a family of translation surfaces that cannot lie in a rank one orbit closure.  In the aforementioned lemma, this is especially clear because all rank one orbit closures in $\cH(4)$ are Teichm\"uller curves, so any deformation that is not by \splin ~implies that the degenerate surfaces  must lie in an affine submanifold of higher rank.

\medskip

Furthermore, Proposition \ref{prop:RankkImpRankkBd} does not claim that in the presence of several zeros, it is always possible to collapse cylinders to get a rank $k$ manifold in the same genus.  For example, consider the Prym locus in $\cH(2,2)^{\rm hyp}$.  It is $4$-dimensional and has rank two.  Indeed, the intersection of the boundary of the affine manifold with all lower strata in the same genus is the empty set.
\end{remark}

\begin{remark}
 A more detailed account and general results on the closure of $\cM$ in a partial compactification of $\cH_g$ are given in a new preprint by Mirzakhani-Wright~\cite{MirzWrigBoundary}.
\end{remark}

\section{Two Cylinders}

The following lemmas are valid in any stratum $\cH(\kappa)$ of any genus $g \geq 3$.

\begin{lemma}
\label{SC2Cyls}
If $M$ is a union of two non-homologous cylinders, then at least one of them has a saddle connection on its top and bottom.
\end{lemma}

\begin{proof}
If not, then they would be homologous.
\end{proof}

\begin{lemma}
\label{3PlusCyls}
Let $\cM$ be a rank two affine submanifold of $\cH(\kappa)$, where $g\geq 3$, with rational affine field of definition. Then $\cM$ contains a horizontally periodic translation surface with at least three cylinders.
\end{lemma}

\begin{proof}
By Theorem~\ref{thm:RankkImpkCyl}, we know that $\cM$ contains a surface $M$ which is horizontally periodic with at least two cylinders. Obviously we only need to consider the case where $M$ has exactly two cylinders $C_1$ and $C_2$. Note that in this case the two cylinders are not $\cM$-parallel by Theorem \ref{thm:RankkImpkCyl}.   By Lemma \ref{SC2Cyls}, without loss of generality, assume that $C_1$ contains a simple cylinder $C_1'$ formed by considering the foliation in the direction that connects the saddle connection on the top of $C_1$ with its copy on the bottom of $C_1$.  Using $\SL(2,\bR)$, we can assume that $C'_1$ is vertical.

Since the affine field of definition of $\cM$ is $\bQ$, by Proposition~\ref{prop:SimpCylPres} there exists a square-tiled surface $M'$ close to $M$, on which $C'_1$ is also a vertical simple cylinder. Note that $M'$ is vertically periodic, and has at least two vertical cylinders. If $M'$ has more than two vertical cylinders, then we are done. Thus, we only need to consider the case where $M'$ has only one vertical cylinder other than $C'_1$. Rotate $M'$ by $\pi/2$ so that $C'_1$ is now a horizontal simple cylinder. Let $C'_2$ be the other horizontal cylinder of $R_{\pi/2}\cdot M'$ (where $R_\theta=\left(\begin{smallmatrix} \cos\theta & -\sin\theta \\ \sin\theta & \cos\theta \end{smallmatrix}\right)$ ). If $C'_1$ and $C'_2$ are $\cM$-parallel, then by Theorem~\ref{thm:RankkImpkCyl}, there exists a horizontally periodic surface in a neighborhood of $R_{\pi/2}\cdot M'$ with at least three cylinders and we are done.  Therefore, we only need to consider the case: $C'_1$ and $C'_2$
  are free.

Recall that $C'_1$ is simple, thus there exists a pair of homologous saddle connections $\sigma_1,\sigma_2$ in $C'_2$ that cut out a slit torus $T'$ containing $C'_1$. Since $C'_1$ and $C'_2$ are free,  we can freely twist $C'_1$ and $C'_2$ so that $\sigma_i$ is vertical, and the twist of $C'_1$ is zero, that is $C'_1$ can be represented by a rectangle. It follows that $\cM$ contains a surface which has  a vertical cylinder $C''_1$ whose closure is a slit torus. By a slight abuse of notation, we still denote this surface by $M'$. Let $\sigma_0$ denote the unique vertical saddle connection in $C'_1$, then the two boundaries of $C''_1$ are $\sigma_0\cup\sigma_1$ and $\sigma_0\cup\sigma_2$.

Since $\cM$ is defined over $\bQ$, there exists a square-tiled surface $M''$ close to $M'$  on which $\sigma_0,\sigma_1,\sigma_2$ persist and are also vertical (Proposition~\ref{prop:SimpCylPres}). In particular, $C''_1$ persists on $M''$. Since the vertical direction on $M''$ is periodic, $M''$ has vertical cylinders other than $C''_1$. Again we only have to consider the case where $M''$ has two vertical cylinders $C''_1$ and $C''_2$, both are free otherwise we can conclude by Theorem~\ref{thm:RankkImpkCyl}. Rotate $M''$ by $\pi/2$ so that $C''_1$ and $C''_2$ are horizontal. Remark that one of $\sigma_1,\sigma_2$ is contained in the top boundary of $C''_2$, and the other one is contained in the bottom boundary of $C''_2$. Let $\eta_1,\eta_2$ be a pair of saddle connections in $C''_2$ that cut out a parallelogram whose top and bottom sides are $\sigma_1$ and $\sigma_2$. We can freely twist $C''_1$ and $C''_2$ so that there exist two vertical cylinders: $D_1$ is a cylinder in 
 $C''_1$ that contains $\sigma_0$,
 and $D_2$ is the vertical cylinder through $\sigma_1,\sigma_2$ which is bounded by $\eta_1,\eta_2$, and the saddle connections in the boundary of $D_1$.  Since $g\geq 3$, $D_1\cup D_2$ cannot fill the whole surface, therefore there exists in $\cM$ a vertically periodic surface with at least three cylinders. The lemma is then proved.
\end{proof}

We also need the following lemma,  which is specific for the cases $\cH(m,n)$, with $m+n=4$, and strengthens Lemma~\ref{3PlusCyls} a little in those cases.

\begin{lemma}
 \label{lm:Hmn:2cyl:imp:3cyl:1:sim}
 Let $\cM$ be a rank two submanifold of $\cH(m,n)$, with $m+n=4$. Assume that $\cM$ contains a surface $M$ which is horizontally periodic with two cylinders, one of which is simple. Then $\cM$ contains a horizontally periodic surface with at least three cylinders, one of which is simple and not free.
 \end{lemma}

 \begin{proof}
 Let $C_1$ and $C_2$ be the horizontal cylinders of $M$, where $C_1$ is simple. If $C_1$ and $C_2$ are $\cM$-parallel then $\Tw(M,\cM) \neq \CP(M,\cM)$, thus we can use the arguments in Lemma~\ref{lm:sim:cyl:2:cl} to get a square-tiled surface $M'$ close to $M$ on which $C_1$ and $C_2$ persist  (they are always $\cM$-parallel).  The cylinder $C_1$ is still a simple cylinder, but $M'$ has at least three horizontal cylinders, hence the lemma is proved for this case.

Assume now that $C_1$ and $C_2$ are both free. Since $C_1$ is simple, the top (resp. bottom) border of $C_1$ contains only one singularity of $M$.  Remark that the bottom border (resp. top border) of $C_1$ is properly contained in the top border (resp. bottom border)  of $C_2$, and any horizontal saddle connection which is not a border of $C_1$ is contained in both top and bottom borders of $C_2$. From these observations, we  derive that if the top border of $C_2$ contains only one singularity then its bottom border only contains  the same singularity, which means that $M$ has only one singularity.

Since we have assumed that $M \in \cH(m,n)$, there must exist a saddle connection $\sigma$ connecting two distinct zeros  which is contained in both top and bottom borders of $C_2$. Let $C'$ be the simple cylinder in $C_2$ consisting of trajectories crossing $\sigma$ and no other horizontal saddle connections. We can assume $C'$ is vertical.

Again, since $\cM$ is defined over $\bQ$, we can find in a neighborhood of $M$ a square-tiled surface $M'$ on which $C'$ is still a simple vertical cylinder. Note that the  surface $M'$ is vertically periodic. We claim that $C'$ is not free. Indeed, if $C'$ is free, then we can collapse it so that the two zeros in its borders collide, and the resulting surface, denoted by $M''$, belongs to $\cH(4)$. From Proposition~\ref{prop:RankkImpRankkBd}, we know that $M''$ must belong to a rank two affine submanifold $\cM'$ of $\cH(4)$.

From Theorem \ref{NWANWThm}, we have $\cM'=\tilde{\cQ}(3,-1^3)$, therefore $M'' \in \tilde{\cQ}(3,-1^3)$. Observe that by construction $M''$ admits a cylinder decomposition in the horizontal direction into two cylinders,  one of which is simple. But it is impossible for surfaces in $\tilde{\cQ}(3,-1^3)$ to have such a cylinder decomposition. This is because there is an involution $\tau$ with four fixed points in $M''$, and in particular $\tau$ must permute the two cylinders in this decomposition. Thus we have a contradiction, which implies that $C'$ is not free.

Let $\cC'$ be the equivalence class of $C'$ in $M'$, then $\cC'$ contains at least two cylinders. Since $P(C',\{C_2\})=1$, for any $C''\in \cC'$, we also have $P(C'',\{C_2\})$, which means that $C''$ is contained in the closure of $C_2$. It follows that the cylinders in $\cC'$ do not fill $M'$. Therefore, $M'$ has at least three cylinders in the vertical direction. The proof of the lemma is now complete.
 \end{proof}

\section{Three Cylinders}

\begin{conv}
Throughout this section, we will use the following convention for specifying cases.  Case $n$.$R$.$\square$) will denote a horizontally periodic translation surface with $n$ cylinders such that pinching the core curves of every cylinder results in a degenerate surface on a list specified below and denoted by a Roman numeral $R$.  The final term will often be omitted, and is used only if there is a need to specify a subcase.
\end{conv}

We would like to stress the fact that even though our main results only concern strata with two zeros in genus three, some of the following lemmas are actually valid for all strata in genus three.  We begin by considering all possible topological configurations for a degenerate surface resulting from pinching the core curves of every cylinder in a $3$-cylinder diagram in genus three.  We adopt the usual terminology and call a connected component of a degenerate Riemann surface (after removing nodes) a \emph{part}.  We also pinch the core curves of cylinders by letting their heights go to infinity, while their circumferences remain fixed.  With respect to the Abelian differential on the surface this implies that each infinite cylinder gives rise to a pair of simple poles.  Whenever we use the terminology \emph{pair of poles} we will specifically be referring to the two poles arising from pinching a single core curve of a cylinder.

\begin{lemma}
\label{3CylDeg}
If a horizontally periodic genus three translation surface $M$ decomposes into exactly three cylinders, then pinching the core curves of those cylinders, by letting the heights of the cylinders go to infinity so that each cylinder becomes a pair of simple poles of an Abelian differential, degenerates the surface to one of three possible surfaces:
\begin{itemize}
\item 3.I) A sphere with three pairs of simple poles,
\item 3.II) A sphere and a torus joined by three pairs of simple poles, or
\item 3.III) A sphere with a pair of simple poles joined to a torus by two pairs of simple poles.
\end{itemize}
\end{lemma}

\begin{proof}
If the surface decomposes into two parts, then there must be at least two nodes joining the two parts because the core curve of the cylinder of an Abelian differential can never be a separating curve.  If there are exactly two nodes (pairs of poles) between the two parts, then the third pair of poles must lie on one of the parts.  This accounts for all pairs of poles and Case 3.III) is the only possibility.

Finally, if there are three pairs of simple poles between the two parts, again the parts are completely determined because all cylinders are accounted for and we get Case 3.II).  Since each part must carry at least two simple poles and a sphere must have at least three simple poles on it, it is not possible to have three or more parts with three cylinders in genus three.
\end{proof}

\begin{remark}
We observe that there is an obvious homological relation among the core curves of cylinders when the surface degenerates to each of Cases 3.II) and 3.III).  Namely, the homology class of one cylinder is equal to the sum of the other two (Case 3.II)), and the core curves of two cylinders are equal in homology (Case 3.III)).  However, in Case 3.I), the core curves of the three cylinders span a Lagrangian subspace of homology because degenerating to a sphere indicates that there is no homological relation among the cylinders.
\end{remark}

\subsection{Case 3.II)}

The goal of this section is to eliminate Case 3.II) by proving that if there is a surface satisfying Case 3.II) in a rank two affine manifold $\cM$ in a stratum with two zeros in genus three, then there is a translation surface in $\cM$ with four horizontal cylinders.

\begin{figure}[htb]
\centering
\begin{tikzpicture}[scale=0.50]
\draw (0,0)--(0,2)--(-.5,5)--(3.5,5)--(4,2)--(4.5,4)--(6.5,4)--(6,2)--(6,0)--cycle;
\draw [dashed] (0,2)--(6,2);
\draw(1.5,5) node[above] {\tiny $\varepsilon_1$};
\draw(5.5,4) node[above] {\tiny $\varepsilon_2$};
\draw(1,0) node[below] {\tiny $\varepsilon_1$};
\draw(5,0) node[below] {\tiny $\varepsilon_2$};
\draw(1.5,3) node {\tiny $C_1$};
\draw(5.5,3) node {\tiny $C_2$};
\draw(3,1) node {\tiny $C_3$};
\foreach \x in {(0,2),(4,2),(6,2)} \filldraw[fill=black] \x circle (2pt);
\draw(0,2) node[left] {\tiny $x_1$};
\draw(4,2) node[below] {\tiny $x_1$};
\draw(6,2) node[right] {\tiny $x_1$};
\end{tikzpicture}
\caption{Case 3.II): A surface in genus three with a simple zero $x_1$, and $\varepsilon_i$ denotes a collection of saddle connections}
\label{CaseIIFig}
\end{figure}
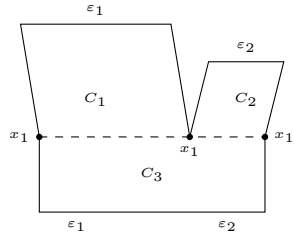

\begin{lemma}
\label{lm:3CylsCaseIILem}
Let $\cM$ be a rank two affine manifold in a stratum in genus three with $k \geq 2$ zeros.  If every translation surface in every rank two manifold $\cM'$ in every stratum in genus three with at most $k-1$ zeros admits a double covering to a half-translation surface, then there does not exist a horizontally periodic translation surface $M \in \cM$ satisfying Case 3.II) such that $\Tw(M, \cM) = \CP(M, \cM)$.
\end{lemma}

\begin{proof}
By contradiction, assume $M \in \cM$ satisfies Case 3.II) and $\Tw(M, \cM) = \CP(M, \cM)$.  By Theorem \ref{thm:RankkImpkCyl}, there are at least two equivalence classes of cylinders.  In fact, all three cylinders must be free because of the homological relation $c_1 + c_2 = c_3$, where $c_i$ is a core curve of the cylinder $C_i$.  Any relation between two of the cylinders induces a relation with the third.  Note that the zero $x_1$ in Figure \ref{CaseIIFig} on the bottom of $C_2$ cannot occur on the top of $C_2$ because $x_1$ is a simple zero.  Therefore, we can twist and collapse $C_2$ while fixing the rest of the surface to reach a translation surface in a lower stratum in genus three.  The resulting surface $M'$ must be contained in a rank two affine manifold by Proposition \ref{prop:RankkImpRankkBd}, or the fact that $C_1$ and $C_3$ are free.

Finally, we claim that $M'$ cannot be a double covering of a half translation surface.  First of all, $C_1$ and $C_3$ have unequal circumferences, so any double covering would induce an involution sending $C_i$ into itself, for $i = 1,3$.  In the argument above, at least one of the cylinders $C_1$ or $C_2$, is not simple and so we could label $C_1$ to be the cylinder that is not simple.  If neither is simple, then choose one of them.  However, if $C_1$ is not simple, then an involution of order two would send all of the saddle connections on its top into the single saddle connection on its bottom.  This is impossible and contradicts the assumption that the rank two manifold in the boundary of $\cM$ admits a double cover to a stratum of half-translation surfaces.
\end{proof}

\begin{corollary}
 \label{cor:no:C3II:in:Hmn}
 If $\cM$ is a rank two submanifold in $\cH(m,n)$, with $m+n=4$, then $\cM$ does not contain any horizontally periodic surface $M$ satisfying Case 3.II) such that $\Tw(M,\cM)=\CP(M,\cM)$.
\end{corollary}
\begin{proof}
 By Theorem \ref{NWANWThm}, the only rank two submanifold of $\cH(4)$ is $\widetilde{\cQ}(3,-1^3)$ consisting of translation surfaces that are double coverings  of quadratic differentials in $\cQ(3,-1^3)$. Thus the corollary follows immediately from Lemma~\ref{lm:3CylsCaseIILem}.
\end{proof}

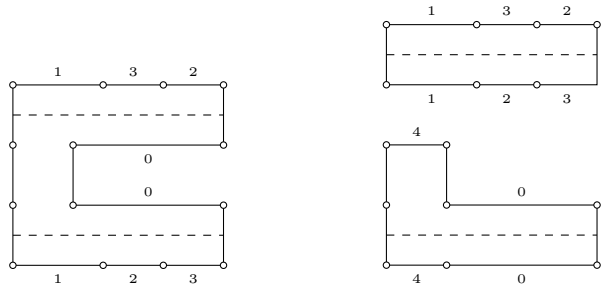
\begin{figure}[htb]
\centering
\begin{minipage}[b]{0.4\linewidth}
\centering
\begin{tikzpicture}[scale=0.40]
\draw (-3,0)--(-3,6)--(4,6)--(4,4)--(-1,4)--(-1,2)--(4,2)--(4,0)--cycle;
\foreach \x in {(-3,0),(-3,6),(4,6),(4,4),(-1,4),(-1,2),(4,2),(4,0),(0,6),(2,6),(0,0),(2,0),(-3,2),(-3,4)} \filldraw[fill=white] \x circle (3pt);
\draw [dashed] (-3,5)--(4,5);
\draw [dashed] (-3,1)--(4,1);
\draw(1.5,2) node[above] {\tiny 0};
\draw(-1.5,6) node[above] {\tiny 1};
\draw(1,6) node[above] {\tiny 3};
\draw(3,6) node[above] {\tiny 2};
\draw(1.5,4) node[below] {\tiny 0};
\draw(-1.5,0) node[below] {\tiny 1};
\draw(1,0) node[below] {\tiny 2};
\draw(3,0) node[below] {\tiny 3};
\end{tikzpicture}
\end{minipage}
\begin{minipage}[t]{0.4\linewidth}
\centering
\begin{tikzpicture}[scale=0.40]
\draw (-3,5)--(-3,7)--(4,7)--(4,5)--cycle;
\foreach \x in {(-3,5),(-3,7),(4,7),(0,7),(2,7),(0,5),(2,5)} \filldraw[fill=white] \x circle (3pt);
\draw [dashed] (-3,6)--(4,6);
\draw(-1.5,7) node[above] {\tiny 1};
\draw(1,7) node[above] {\tiny 3};
\draw(3,7) node[above] {\tiny 2};
\draw(-1.5,5) node[below] {\tiny 1};
\draw(1,5) node[below] {\tiny 2};
\draw(3,5) node[below] {\tiny 3};
\draw (-3,-1)--(-3,3)--(-1,3)--(-1,1)--(4,1)--(4,-1)--cycle;
\foreach \x in {(-3,-1),(-3,1),(-3,3),(-1,3),(-1,1),(4,1),(4,-1),(-1,-1)} \filldraw[fill=white] \x circle (3pt);
\draw [dashed] (-3,0)--(4,0);
\draw(1.5,1) node[above] {\tiny 0};
\draw(-2,3) node[above] {\tiny 4};
\draw(1.5,-1) node[below] {\tiny 0};
\draw(-2,-1) node[below] {\tiny 4};
\end{tikzpicture}
\end{minipage}
\caption{Constructing surfaces satisfying Case 3.III) from genus two surfaces}
\label{CaseIIISurgFig}
\end{figure}

\subsection{Case 3.III)}

As for Case 3.II), the goal of this section is to prove that if a translation surface satisfies this case, then it is always possible to find a translation surface in the affine submanifold $\cM$ with at least four cylinders.  Before doing so, we describe a surgery that allows us to construct all possible cylinder diagrams satisfying Case 3.III) in genus three.

Let $M$ be a translation surface satisfying Case 3.III) in a rank two affine manifold $\cM$ in genus three.  Let $C_1$ and $C_2$ be the two homologous cylinders, and $C_3$ be the third one. If we cut $M$ along the core curves of $C_1$ and $C_2$, depicted as dashed lines in Figure \ref{CaseIIISurgFig}(left), then reglue as in Figure \ref{CaseIIISurgFig}(right), we get two translation surfaces in genus two which are horizontally periodic: the first one admits a $1$-cylinder diagram denoted by $M'_1$, the second one admits a $2$-cylinder diagram, and will be denoted by $M'_2$.  Note that $M'_2$ contains $C_3$.

\begin{lemma}
\label{lm:C3III:no:H31}
 If $M$ admits a cylinder decomposition in Case 3.III), then  $M$ does not belong to $\cH(3,1)$.
\end{lemma}
\begin{proof}
Remark that the singularities of $M$ are also the singularities of $M'_1$ and $M'_2$. Since $M'_i\in \cH(2)\cup\cH(1,1)$, the lemma follows.
\end{proof}

Recall that in genus two there are two $1$-cylinder diagrams (one in each stratum), and three $2$-cylinder diagrams.  By considering all possible combinations, it is possible to produce all possible cylinder diagrams in genus three satisfying Case 3.III).

Since $C_1$ and $C_2$ are homologous, they are $\cM$-parallel.  If $M$ has two equivalence classes of cylinders, then $C_3$ is free. We first observe

\begin{lemma}
\label{lm:C3III:no:nf:sim:cyl}
If $M$ is horizontally periodic with two horizontal simple cylinders, or a non-free simple cylinder, then $M$ does not satisfy Case 3.III).
\end{lemma}
\begin{proof}
 Suppose that $M$ satisfies Case 3.III). Using the notations above, we see that at least one of $C_1,C_2$ is simple. It follows that the unique horizontal cylinder of $M'_1$ has one boundary consisting of a single saddle connection, which is impossible since $M'_1$ has genus two.
\end{proof}

\begin{lemma}
\label{3CylsCaseIIILemRedNonSimpCyl}
If $\cM$ is a rank two affine manifold in genus three and $\cM$ contains a horizontally periodic translation surface satisfying Case 3.III), and $C_3$ is a free simple cylinder, then $\cM$ contains a horizontally periodic translation surface $M'$ satisfying Case 3.III) such that the cylinder $C_3' \subset M'$ is free and not simple.  Furthermore, if $C_3$ has a double zero on its boundaries, then $C_3'$ has a double zero on its boundaries as well.
\end{lemma}

\begin{proof}
See Figure \ref{CaseIIILemSimpCylDefPfFig} for the complete argument of this proof.  If $C_3$ is a simple cylinder, twist it so that there is no vertical saddle connection (from the zero on its top to the zero on its bottom) contained in $C_3$.  Observe that collapsing $C_3$ yields a translation surface in the interior of $\cM$ because $C_3$ does not contain any vertical saddle connections, so the distance between every pair of zeros is bounded away from zero.

We define an \emph{extended cylinder deformation} in the following way.  The collapse of $C_3$ determines a path in $\cM$ which is a closed line segment $\ell$ in period coordinates.  Therefore, it is natural to consider the  real line $L$ in $H^1(X,\Sigma,\bR+\imath\bR)$ containing this segment.  In the case above, the  segment and its endpoints lie in the interior of $\cM$, which implies that there is an open subset $U_{\ell} \subset L$ such that $\ell \subset U_{\ell} \subset \cM$. We call the points in $U_\ell \setminus \ell$  {\em extended cylinder deformations} of $M$.

We consider the extended cylinder deformation of $C_3$, which realizes the surface depicted on the right side of Figure \ref{CaseIIILemSimpCylDefPfFig} so that we reach a new translation surface in the interior of $\cM$.  Observe that many horizontal trajectories in each of $C_1$ and $C_2$ persist under this deformation.  However, there is a new horizontal trajectory whose boundary entirely contains the top and bottom boundaries of the cylinders of $C_2$ and $C_1$, respectively.  This horizontal trajectory determines a new horizontal cylinder, which we call $C'_3$.  Since $C_3'$ can be deformed via the extended deformation of $C_3$, $C'_3$ must also be free.
\end{proof}

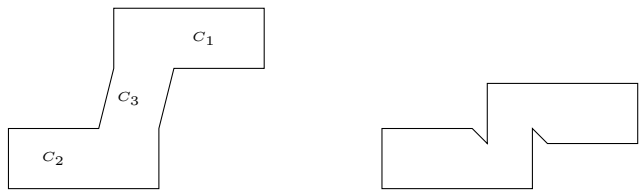
\begin{figure}
\centering
\begin{minipage}[t]{0.4\linewidth}
\centering
\begin{tikzpicture}[scale=0.40]
\draw (-4,0)--(-4,2)--(-1,2)--(-0.5,4)--(-0.5,6)--(4.5,6)--(4.5,4)--(1.5,4)--(1,2)--(1,0)--cycle;
\draw (-2.5,1) node {\tiny $C_2$} (0,3) node {\tiny $C_3$} (2.5,5) node {\tiny $C_1$};
\end{tikzpicture}
\end{minipage}
\begin{minipage}[t]{0.4\linewidth}
\centering
\begin{tikzpicture}[scale=0.40]
\draw (-4,0)--(-4,2)--(-1,2)--(-0.5,1.5)--(-0.5,3.5)--(4.5,3.5)--(4.5,1.5)--(1.5,1.5)--(1,2)--(1,0)--cycle;
\end{tikzpicture}
\end{minipage}
\caption{Deformation of translation surface to Case 3.III) without a simple cylinder.}
\label{CaseIIILemSimpCylDefPfFig}
\end{figure}

Our goal now is to show

\begin{proposition}
\label{prop:3III:imply:4cyls}
Let $\cM$ be a rank two submanifold in one of the strata $\cH(m,n)$, with  $m+n=4$. Suppose that $\cM$ contains a surface $M$ which is horizontally periodic and satisfies Case 3.III). Then $\cM$ contains a
surface admitting a cylinder decomposition with four cylinders.
\end{proposition}

\begin{proof}
It suffices to assume that $\Tw(M, \cM) = \CP(M, \cM)$ because otherwise we are done by Lemma \ref{lm:WrightTwistPresLem}.  Since the cylinders must be divided into at least two equivalence classes and $C_1$ and $C_2$ are homologous, thus $\cM$-parallel, $C_3$ is free. By Lemma~\ref{3CylsCaseIIILemRedNonSimpCyl}, we can assume that $C_3$ is not simple. By Lemma~\ref{lm:C3III:no:H31}, we only need to consider the case $\cM \subset \cH(2,2)$, thus both $M'_1$ and $M'_2$ belong to $\cH(2)$. Since there is only one $1$-cylinder diagram, and one $2$-cylinder diagram in $\cH(2)$, it is easy to check that there is only one cylinder diagram we need to consider which is shown in Figure~\ref{CaseIIILemPfFig}.

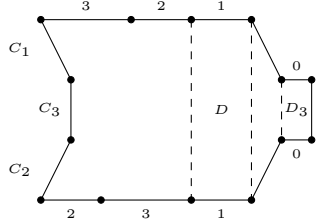
\begin{figure}[htb]
\centering
\begin{minipage}[t]{0.3\linewidth}
\centering
\begin{tikzpicture}[scale=0.40]
\draw (-4,0)--(-3,2)--(-3,4)--(-4,6)--(3,6)--(4,4)--(5,4)--(5,2)--(4,2)--(3,0)--cycle;
\foreach \x in {(-4,0),(-2,0),(1,0),(3,0),(-4,6),(-1,6),(1,6),(3,6),(4,4),(5,4),(5,2),(4,2),(-3,2),(-3,4)} \filldraw[fill=black] \x circle (3pt);
\draw(-4,1) node[left] {\tiny $C_2$};
\draw(-3,3) node[left] {\tiny $C_3$};
\draw(-4,5) node[left] {\tiny $C_1$};
\draw(4.5,3) node {\tiny $D_3$};
\draw(2,3) node {\tiny $D$};
\draw [dashed] (1,6)--(1,0);
\draw [dashed] (3,6)--(3,0);
\draw [dashed] (4,4) -- (4,2);
\draw(4.5,2) node[below] {\tiny 0};
\draw(2,6) node[above] {\tiny 1};
\draw(-2.5,6) node[above] {\tiny 3};
\draw(0,6) node[above] {\tiny 2};
\draw(4.5,4) node[above] {\tiny 0};
\draw(2,0) node[below] {\tiny 1};
\draw(-3,0) node[below] {\tiny 2};
\draw(-0.5,0) node[below] {\tiny 3};
\end{tikzpicture}
\end{minipage}
\caption{Case 3.III):  finding two vertical cylinders.}
\label{CaseIIILemPfFig}
\end{figure}

Remark that in this diagram, if we cut $M$ horizontally along the top of $C_2$ and the bottom of $C_1$, and reglue the two boundaries of two components, we then get a slit torus $M''_1$ that contains $C_3$, and a surface $M''_2$ in $\cH(2)$ horizontally periodic with a single cylinder. Note that in $M''_2$ we have a marked horizontal simple closed geodesic $c$, which corresponds to the slit of $M''_1$, and a marked point $x \in c$, which corresponds to the singularity in the boundary of $C_3$.

Note also that $\overline{C}_3$ contains a transverse simple cylinder $D_3$ crossing the saddle connection $(0)$, it is easy to see that $D_3$ is free. Twist $C_3$ as in Figure~\ref{CaseIIILemPfFig}, so that the free simple cylinder $D_3$ is vertical.

We claim that we can twist $\cC = \{C_1, C_2\}$ so that there is another simple vertical cylinder passing exactly once through all three horizontal cylinders.  To see this, we consider the splitting of $M$ into the connected sum of $M''_1$ and $M''_2$ described above.  Let $C$ denote the unique horizontal cylinder in $M''_2$. Recall that we have a distinguished core curve $c$ of $C$, and a marked point $x \in c$. Note that the boundary of $C$ consists of three horizontal saddle connections, and one can always find two simple cylinders which are disjoint and cross the core curves of $C$ once. Therefore, there always exists a transverse simple cylinder $D$ crossing the core curves of $C$ once whose closure does not contain $x$. Since twisting simultaneously $C_1$ and $C_2$ is the same as twisting $C$, we can assume $D$ is vertical.

To reconstruct $M$, we have to cut $M''_2$ along $c$, and glue the two copies of $c$ with the two sides of the slit in $M'_1$ so that $x$ is identified with both endpoints of the slit. Since  $x$ is not contained in  the closure of $D$, we see that $D$ gives rise to a simple cylinder in $M$ crossing  all of the horizontal cylinders once and disjoint from $D_3$. By a slight abuse of notation, we also denote this cylinder by $D$.

Using Lemma~\ref{prop:SimpCylPres}, we can assume that $M$ is a square-tiled surface, which is vertically periodic. If $M$ has four vertical cylinders, then we are done. Assume that $M$ has only three vertical cylinders. We claim that all three vertical cylinders are free. Let $D'$ denote the third vertical cylinder. We already have $D_3$ is free. Since the closures of $D$ and $D_3$ are disjoint, $D_3$ is only adjacent to $D'$. Thus by Lemma~\ref{lm:f:sim:cyl}, $D'$ is free, it follows immediately that $D$ is free.

Next, we claim that the cylinder decomposition of $M$ in the vertical direction does not satisfy Case 3.II) or Case 3.III). Indeed, if this cylinder decomposition satisfies Case 3.III), then there must be two homologous cylinders, which is impossible as we have three free cylinders. Since Case 3.II) is already excluded, we conclude that this cylinder decomposition satisfies Case 3.I). But in Case 3.I), the core curves of the cylinders span a Lagrangian of $H_1(M,\bZ)$, and we get a contradiction by Lemma~\ref{lm:3:f:cyl}. Thus there must be four vertical cylinders in $M$.
\end{proof}

\section{Three Cylinders: Case 3.I)}
\label{CaseIReductSect}

It turns out that most of the 3-cylinder diagrams satisfy this case.  We state the main result of this section here.  The proof is given at the end of this section.


\begin{theorem}
\label{thm:C3I:imply:4cyl}
Let $\cM$ be a rank two affine manifold in one of the strata $\cH(m,n)$, with $m+n=4$.  If  $\cM$ contains a  horizontally periodic translation surface with three cylinders satisfying Case 3.I), then there exists $M \in \cM$ horizontally periodic with at least four cylinders.
\end{theorem}

Our approach to  prove this theorem is by studying horizontally periodic translation surfaces in genus three that satisfy Case 3.I) and one of the following two non-exclusive properties:

\begin{itemize}
\item[(a)] There exists $i \in \{1,2,3\}$ such that there is a horizontal saddle connection contained in both top and bottom of $C_i$.
\item[(b)] There exists $i \in \{1,2,3\}$ such that $C_i$ is semi-simple (see Definition~\ref{def:semi-sim:cyl}).
\end{itemize}

Though these properties may seem arbitrary, we prove in Lemma \ref{3CylIException} that there is exactly one 3-cylinder diagram satisfying Case 3.I) that does not satisfy one of these properties  in strata  $\cH(m,n)$ with $m + n = 4$.  We prove in Lemma~\ref{lm:ExceptionalCaseReduction}, that the existence of a translation surface satisfying this ``exceptional case'' in a rank two affine manifold implies the existence of a translation surface in the same affine manifold satisfying one of these two properties.

In light of this fact, it suffices to thoroughly study translation surfaces satisfying at least one of the properties above.  We start with the very specific case when two of the cylinders of $M$ are simple, cf. Proposition \ref{prop:3CylsI2SimpCyls}.  This serves as an elementary case from which we can build to greater generality.  Proposition \ref{prop:C3I:1:n-f:sim:cyl} serves as the next step by proving that if $M$ has exactly one non-free simple cylinder, then $\cM$ contains a surface with four cylinders.

If one of the non-free cylinders of $M$ contains a simple cylinder\footnote{The phrases ``$C$ contains a simple cylinder,'' and ``$C$ contains the same saddle connection on its top and bottom,'' are equivalent, and we pass freely between them.}, then $\cM$ contains a surface with four cylinders, cf. Proposition \ref{prop:C3I:loop:in:c1:c2}.  Also, if the free cylinder on $M$ is simple, then the same conclusion holds, cf. Proposition \ref{prop:C3I:c3:is:sim}.  These results are summarized in Proposition \ref{prop:C3I:1sim:cyl}, which says that if one of the three cylinders is simple, then $\cM$ contains a surface with four cylinders.  This leads to Proposition \ref{prop:C3I:loop}, which completes the case of surfaces satisfying Property (a),  that is, that if $M$ has a cylinder containing a simple cylinder, then $\cM$ contains a translation surface with four cylinders.

Finally, Proposition \ref{prop:C3I:semi-sim:cyl}, which concerns semi-simple cylinders, combined with the aforementioned lemmas concerning the exceptional case, completes the proof of Theorem \ref{thm:C3I:imply:4cyl}.

Again we remind the reader that some of the results in this section are written so that they apply to \emph{all $3$-cylinder diagrams in genus three} satisfying Case 3.I).

\subsection{Two Simple Cylinders}

When two of the horizontal cylinders of $M$ are simple, we label them by $C_1$ and $C_2$, the third one is denoted by $C_3$.

\begin{lemma}
 \label{lm:C3I:2SimCyl:para}
Let $\cM$ be a rank two affine manifold in genus three.  Let $M \in \cM$ be a horizontally periodic translation surface with three cylinders such that $C_1$ and $C_2$ are both simple cylinders.  Either $C_1$ and $C_2$ must be $\cM$-parallel, or there exists a horizontally periodic in $\cM$ with at least four cylinders.
\end{lemma}

\begin{proof}
 By Lemma~\ref{lm:3:f:cyl}, we know that $C_1,C_2, C_3$ cannot be all free. If they all belong to the same equivalence class, then we must have a horizontally periodic surface in $\cM$ with more cylinders by Theorem~\ref{thm:RankkImpkCyl}. Thus we only need to consider the case  $C_1,C_2,C_3$ fall into two equivalence classes. Note that two simple cylinders cannot be adjacent, therefore the saddle connections in the boundaries of both $C_1$ and $C_2$ are included the boundary of $C_3$. If $C_1$ or $C_2$ is $\cM$-parallel to $C_3$, then we only have one equivalence class by Lemma~\ref{lm:n:f:sim:cyl}. Thus $C_3$ must be free, and $C_1,C_2$ are $\cM$-parallel.
\end{proof}

In what follows we will always assume that $C_1$ and $C_2$ are $\cM$-parallel, and $C_3$ is a free cylinder.   Let $h_i, \ell_i, t_i$ denote respectively the height, width (circumference), and twist of $C_i$. Denote by $a$ and $a'$ (resp. $b$ and $b'$) the saddle connections in the boundary of $C_1$ (resp. $C_2$).

\begin{lemma}
 \label{lm:C3I:2SimCyl:C1C2isom}
 The cylinders $C_1$ and $C_2$ are isometric, that is, they have the same width, height, and twist. Moreover, one can twist $C_3$ so that any vertical trajectory  through $C_1$ or $C_2$ passes exactly once through $C_3$ before closing itself.
\end{lemma}

\begin{proof}
Since $C_1$ and $C_3$ are not $\cM$-parallel, we can twist them so that $a'$ is right above $a$, and $t_1=0$ (see Figure~\ref{fig:Case3I:2simp:cyl}). It follows that there exists a vertical cylinder $C'_1$ crossing only $C_1$ and  $C_3$ such that $C_1\subset \overline{C}'_1$. Since $C_2$ is $\cM$-parallel to $C_1$, there must be a vertical cylinder $C_2'$ in the same class as $C'_1$ passing through $C_2$.

\begin{figure}[htb]
\centering
\begin{tikzpicture}[scale=0.35]
\fill[gray!30] (0,0) -- (3,0) -- (3,5) -- (0,5) -- cycle;
\fill[gray!30] (7,0) -- (10, 0) -- (10,3) -- (11,5)  -- (8,5) -- (7,3) -- cycle;

\draw[thin] (0,0) -- (4,0) (6,0) -- (11,0) (0,3) -- (4,3) (6,3) -- (11,3);
\draw[thin] (0,0) -- (0,5) -- (3,5) -- (3,0);
\draw[thin] (7,0) -- (7,3) -- (8,5) -- (11,5) -- (10,3) -- (10,0);
\draw[thin]  (8,5) -- (8,0) (9,5) -- (9,0) (10,5) -- (10,3);

\draw (5,0) node {$\dots$} (5,3) node {$\dots$} (12,0) node {$\dots$} (12,3) node {$\dots$};

\foreach \x in {(0,5), (0,3), (0,0), (3,5), (3,3), (3,0), (8,5), (7,3), (7,0), (11,5), (10,3), (10,0)} \filldraw[fill=black] \x circle (3pt);
\draw (1.5,0) node[below] {\tiny $a$} (1.5,3) node[below] {\tiny $a'$} (8.5,0) node[below] {\tiny $b$} (8.5,3) node[below] {\tiny $b'$};
\draw (0,1.5) node[left] {\tiny $c$} (3,1.5) node[right] {\tiny $c'$} (7,1.5) node[left] {\tiny $d$} (10,1.5) node[right] {\tiny $d'$};

\draw (1.5,4) node {\tiny $C_1$} (7,4) node {\tiny $C_2$} (5,1.5) node {\tiny $C_3$};
\end{tikzpicture}
\caption{Case 3.I) with two simple cylinders: the shaded regions correspond to two slit tori.}
\label{fig:Case3I:2simp:cyl}
\end{figure}
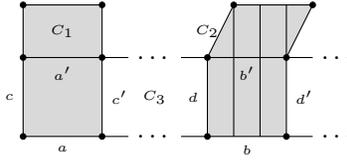

Clearly $C'_2$ can only cross $C_2$ and $C_3$. Let $n_i$ be the number of times $C'_2$ crosses $C_i, i=2,3$.  In fact, $n_3 \geq n_2$ because $C_2'$ cannot pass through $C_2$ without passing through $C_3$.  We have $P(C_1', C_3) = P(C_2', C_3)$ implies
$$\frac{h_3}{h_1 + h_3} = \frac{n_3h_3}{n_2h_2 + n_3h_3}.$$
This simplifies to the relation
$$1 \geq \frac{n_2}{n_3} = \frac{h_1}{h_2}.$$

If we twist the cylinders $C_2$ and $C_3$ to get a vertical cylinder crossing $C_2$ and $C_3$ once, and consider the vertical cylinder $\cM$-parallel to it, we get the following relation
$$\frac{h_3}{h_2 + h_3} = \frac{n_3'h_3}{n_1'h_1 + n_3'h_3},$$
where $n_3' \geq n_1'$.  However, this forces
$$1 \geq \frac{n_1'}{n_3'} = \frac{h_2}{h_1},$$
which implies $h_1 = h_2$, $n_2 = n_3$, and $n_1' = n_3'$.  But the condition $n_2=n_3$ can only be satisfied when $b'$ is right above $b$, which means that when $C_3$ is twisted so that $a'$ is right above $a$, then $b'$ is right above $b$.

Now we prove that $C_1$ and $C_2$ are isometric.  Let $c$ and $c'$ be the vertical saddle connections in $C_3$ that join the left endpoints of $a'$ to the left endpoint of $a$, and the right endpoint of $a'$ to the right endpoint of $a$ respectively. Similarly, let $d$ and $d'$ be the vertical saddle connections in $C_3$ that join the left endpoints of $b'$ to the left endpoint of $b$, and the right endpoint of $b'$ to the right endpoint of $b$, respectively. Note that $c$ and $c'$ (resp. $d$ and $d'$) cut out a slit torus denoted by $T_1$ (resp. denoted by $T_2$) which is the closure of $C'_1$ (resp. $C'_2$) (see Figure~\ref{fig:Case3I:2simp:cyl}).

Let $D_1$ be any simple cylinder in $T_1$ disjoint from the slit.  Then $D_1$ corresponds to a simple cylinder on $M$. Note that the complement of $D_1$ in $T_1$ is a parallelogram bounded by  the borders of $D_1$ and the pair $c,c'$. Since $C'_2$ is $\cM$-parallel to $C'_1$, there must exist a cylinder $D_2$ crossing $C'_2$ which is $\cM$-parallel to $D_1$. We claim that $D_2$ is contained in $T_2$. Indeed, let $\cC'$ be the equivalence class of $C'_1$ and $C'_2$, then we have $\cC'=\{C'_1,C'_2\}$ since any cylinder in this equivalence class must cross $C_1$ or $C_2$.  We have
$$
P(D_1,\cC')=1 \Rightarrow P(D_2,\cC')=1,
$$
which means that $D_2$ is contained in the union  $T_1\cup T_2$. If $D_2$ intersects $T_1$ since it is parallel to $D_1$ it must cross both $c$ and $c'$, thus it cannot be contained in $T_1\cup T_2$. We derive that $D_2$ must be contained in $T_2$.  Moreover, since we have $h_1=h_2$, and the heights of $c$ and $d$ are both equal to $h_3$, it is not difficult to check that
$$
\frac{\Area(D_1)}{\Area(T_1)} =\frac{\Area(D_2)}{\Area(T_2)}.
$$

We can now use \cite[Lem. 8.1]{AulicinoNguyenWright} to conclude that $T_1$ and $T_2$ are isometric. The lemma is then proved.
\end{proof}

\begin{remark}
 Both Lemmas~\ref{lm:C3I:2SimCyl:para} and \ref{lm:C3I:2SimCyl:C1C2isom} are valid in all strata of genus three.
\end{remark}

\begin{proposition}
\label{prop:3CylsI2SimpCyls}
Let $\cM$ be a rank two affine manifold in genus three in a stratum with two zeros.  If $M \in \cM$ is a horizontally periodic translation surface satisfying Case 3.I) and two of the cylinders are simple, then there is a translation surface in $\cM$ with four cylinders.
\end{proposition}
\begin{proof}
A horizontally periodic surface in genus three with two zeros has exactly six horizontal saddle connections. In this case all of the horizontal saddle connections of $M$ are contained in the boundary of $C_3$. We have $a,a'$ and $b,b'$, in the boundaries of $C_1$ and $C_2$, respectively, and two other ones, denoted by $e$ and $f$, that are contained in both top and bottom border of $C_3$. We can twist $C_3$ so that a saddle connection joining the left endpoint of $a'$ to the left endpoint of $a$ is vertical. By Lemma~\ref{lm:C3I:2SimCyl:C1C2isom}, we know that $b'$ must lie right above $b$. It is now easy to check that there are only three diagrams for $M$ as shown in  Figure~\ref{fig:2SimpCylPf}. We also twist $C_1$ and $C_2$ so that there are two vertical cylinders $C'_1,C'_2$ crossing $C_3$ once such that $C_i \subset \overline{C}'_i, \, i=1,2$.

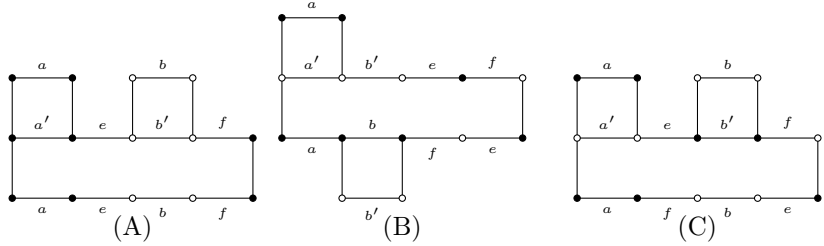
\begin{figure}
\centering
\begin{minipage}[t]{0.3\linewidth}
\centering
\begin{tikzpicture}[scale=0.4]
\draw[thin] (0,4) -- (0,0) -- (8,0) -- (8,2) -- (6,2) -- (6,4) -- (4,4) -- (4,2) -- (2,2) -- (2,4) -- cycle;
\draw[thin] (0,2) -- (2,2) (4,2) -- (6,2);

\draw (1,4) node[above] {\tiny $a$} (1,2) node[above] {\tiny $a'$} (1,0) node[below] {\tiny $a$} (3,2) node[above] {\tiny $e$} (3,0) node[below] {\tiny $e$} (5,4) node[above] {\tiny $b$} (5,2) node[above] {\tiny $b'$} (5,0) node[below] {\tiny $b$} (7,2) node[above] {\tiny $f$} (7,0) node[below] {\tiny $f$};

\foreach \x in {(0,4), (0,2), (0,0), (2,4),(2,2), (2,0), (8,2), (8,0)} \filldraw[fill=black] \x circle (3pt);
\foreach \x in {(4,4), (4,2), (4,0), (6,4),(6,2), (6,0)} \filldraw[fill=white] \x circle (3pt);

\draw (4,-1) node {(A)};

\end{tikzpicture}
\end{minipage}
\begin{minipage}[t]{0.3\linewidth}
\begin{tikzpicture}[scale=0.4]
\draw[thin] (0,6) -- (0,2) -- (2,2) -- (2,0) -- (4,0) -- (4,2) -- (8,2) -- (8,4) -- (2,4) -- (2,6) -- cycle;
\draw[thin] (0,4) -- (2,4) (2,2) -- (4,2);

\draw (1,6) node[above] {\tiny $a$} (1,4) node[above] {\tiny $a'$} (1,2) node[below] {\tiny $a$} (3,4) node[above] {\tiny $b'$} (3,2) node[above] {\tiny $b$} (3,0) node[below] {\tiny $b'$} (5,4) node[above] {\tiny $e$} (7,2) node[below] {\tiny $e$} (7,4) node[above] {\tiny $f$} (5,2) node[below] {\tiny $f$};

\foreach \x in {(0,6), (0,2), (2,6), (2,2), (4,2), (6,4), (8,2)} \filldraw[fill=black] \x circle (3pt);

\foreach \x in {(0,4), (2,4), (2,0), (4,4), (4,0), (6,2), (8,4)} \filldraw[fill=white] \x circle (3pt);
\draw (4,-1) node {(B)};
\end{tikzpicture}
\end{minipage}
\begin{minipage}[t]{0.3\linewidth}
\centering
\begin{tikzpicture}[scale=0.4]
\draw[thin] (0,4) -- (0,0) -- (8,0) -- (8,2) -- (6,2) -- (6,4) -- (4,4) -- (4,2) -- (2,2) -- (2,4) -- cycle;
\draw[thin] (0,2) -- (2,2) (4,2) -- (6,2);

\draw (1,4) node[above] {\tiny $a$} (1,2) node[above] {\tiny $a'$} (1,0) node[below] {\tiny $a$} (3,2) node[above] {\tiny $e$} (3,0) node[below] {\tiny $f$} (5,4) node[above] {\tiny $b$} (5,2) node[above] {\tiny $b'$} (5,0) node[below] {\tiny $b$} (7,2) node[above] {\tiny $f$} (7,0) node[below] {\tiny $e$};

\foreach \x in {(0,4),  (0,0), (2,4), (2,0), (4,2), (6,2), (8,0)} \filldraw[fill=black] \x circle (3pt);
\foreach \x in {(0,2), (2,2), (4,4),  (4,0), (6,4), (6,0), (8,2)} \filldraw[fill=white] \x circle (3pt);

\draw (4,-1) node {(C)};
\end{tikzpicture}
\end{minipage}
\caption{$3$-cylinder diagrams with two simple cylinders}
\label{fig:2SimpCylPf}
\end{figure}
\begin{itemize}
\item[$\bullet$] In Case (A), we immediately have four vertical cylinders.

\item[$\bullet$]  In Case (B), consider the cylinders $E$ and $F$ that are contained in $C_3$ and cross only $e$ and $f$ respectively. It is easy to see that $E$ and $F$ are free since any cylinder parallel to $E$ or $F$ must intersect $C'_1\cup C'_2$. Twisting $E$ in one direction followed by $F$ so that the horizontal trajectories on $C_3$ persist, we can find a surface in $\cM$ close to $M$ with four horizontal cylinders.

\item[$\bullet$]  For Case (C) there is a vertical cylinder through $e$ and $f$ which is free. Thus we can freely change lengths of $e$ and $f$ (which are equal). In particular, we can assume that each of  $C_1$ and $C_2$ are each constructed from a standard square, and $C_3$ is constructed from the union of four standard squares. Now if we twist $C_1$ and $C_2$ by $A = \left(\begin{smallmatrix} 1 & 1/3 \\ 0 & 1 \end{smallmatrix}\right)$, and $C_3$ by $A^{-1}$, then $M$ has a decomposition into four cylinders in the vertical direction.
\end{itemize}

The proof of the proposition is now complete.
\end{proof}

\subsection{One Non-Free Simple Cylinder}
\label{sec:C3I:1simp:cyl}
\begin{proposition}
\label{prop:C3I:1:n-f:sim:cyl}
Let $\cM$ be a rank two submanifold in genus three with $k\geq 2$ zeros. Assume that $\cM$ contains a surface $M$ admitting a cylinder decomposition in the horizontal direction with three cylinders in Case 3.I). Denote the horizontal cylinders of $M$ by $C_1,C_2,C_3$ and suppose that $C_1$ and $C_2$ are $\cM$-parallel while $C_3$ is free. If one of the  cylinders $C_1,C_2$ is simple, but the other one is not, then there exists a surface in $\cM$ which is horizontally periodic with four cylinders.
\end{proposition}

Without loss of generality, we can assume that $C_1$ is simple, but $C_2$ is not. If $C_1$ is only adjacent to $C_2$, then by Lemma~\ref{lm:n:f:sim:cyl} all the horizontal cylinders must belong to the same equivalence class. Thus $C_1$ must be adjacent to $C_3$. We first show

\begin{lemma}
\label{lm:sim:cyl:adj:f:cyl}
If $C_1$ is adjacent to $C_3$, then no saddle connection in the top border of $C_2$ occurs also in its bottom border.
\end{lemma}
\begin{proof}
If such a saddle connection exists, then $C_2$ contains a vertical simple cylinder $D_2$. Let  $\cD$ denote the equivalence class of $D_2$. Since $C_1$ is $\cM$-parallel to $C_2$, there must be a cylinder $D_1\in \cD$ such that $D_1 \cap C_1$ has non-zero area. But $C_1$ is simple and adjacent to $C_3$, hence we have $P(D_1,C_3) > 0$, while $P(D_2,C_3)=0$.  We then get a contradiction.
\end{proof}

\begin{proof}[Proof of Proposition~\ref{prop:C3I:1:n-f:sim:cyl}] We only need to consider two cases:

\medskip

\noindent \textbf{Case 1:} $C_1$ is only adjacent to $C_3$.   Let $\sigma$ denote the saddle connection on the top of $C_1$ that is contained in the bottom of $C_3$.  Twisting $C_1$ and $C_3$ independently, we can get a vertical cylinder $C_1'$ containing $C_1$ and crossing $C_1$ and $C_3$ only once.  There must exist $C_2'$ crossing $C_2$ that is $\cM$-parallel to $C_1'$.  The cylinder $C_2'$ cannot pass through $C_1$, so assume that $C_2'$ passes through $C_i, \; i=2,3$, $n_i$ times.  In fact, $n_3 \geq n_2$ because the borders of $C_2$ are contained in the borders of $C_3$.  Letting $h_i$ denote the height of $C_i$, for all $i=1,2,3$, the equality $P(C_1', \{C_3\}) = P(C_2', \{C_3\})$ implies
$$\frac{h_3}{h_1 + h_3} = \frac{n_3h_3}{n_2h_2+n_3h_3}.$$
It follows
$$1 \geq \frac{n_2}{n_3} = \frac{h_1}{h_2}.$$

Likewise, it is possible to make a symmetric argument by twisting $C_2$ and $C_3$ to get a vertical cylinder $C_2''$  passing once through $C_2$ and $C_3$ (see Figure~\ref{C1orC2SimpFigP2}).  Then there must exist a cylinder $C_1''$ crossing $C_1$ that is $\cM$-parallel to $C_2''$.  Let $C_1''$ pass through $C_i$, $n_i'$ times, for $i = 1, 2, 3$.  Since the top of $C_1$ is identified to the bottom of $C_3$ and every saddle connection on the top of $C_2$ is identified to a saddle connection on the bottom of $C_3$, we have the relation $n_3' \geq n_1' + n_2'$.  The equality $P(C_1'', \{C_3\}) = P(C_2'', \{C_3\})$ implies
$$\frac{n_3'h_3}{\sum_i n_i'h_i} = \frac{h_3}{h_2+h_3}.$$
This simplifies to the relation
$$\frac{n_3'-n_1'}{n_2'} = \frac{h_1}{h_2} \leq 1.$$
However, this implies that $n_3' \leq n_1' + n_2'$, which yields $n_1' + n_2' = n_3'$, as well as $n_2 = n_3$ and $h_1 = h_2$.

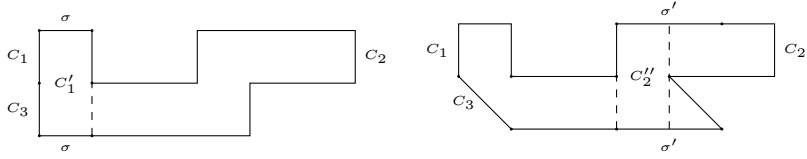
\begin{figure}
\centering
\begin{minipage}[t]{0.45\linewidth}
\centering
\begin{tikzpicture}[scale=0.35]
\draw (-10,0)--(-10,4)--(-8,4)--(-8,2)--(-4,2)--(-4,4)--(2,4)--(2,2)--(-2,2)--(-2,0)--cycle;
\foreach \x in {(-10,0),(-10,2),(-10,4),(-8,4),(-8,2),(-8,0)} \draw \x circle (1pt);
\draw [dashed] (-8,0)--(-8,2);
\draw(-9,4) node[above] {\tiny $\sigma$};
\draw(-9,0) node[below] {\tiny $\sigma$};
\draw(-10,3) node[left] {\tiny $C_1$};
\draw(-10,1) node[left] {\tiny $C_3$};
\draw(2,3) node[right] {\tiny $C_2$};
\draw(-9,2) node {\tiny $C_1'$};
\end{tikzpicture}
\end{minipage}
\begin{minipage}[b]{0.45\linewidth}
\centering
\begin{tikzpicture}[scale=0.35]
\draw (2,0)--(0,2)--(0,4)--(2,4)--(2,2)--(6,2)--(6,4)--(12,4)--(12,2)--(8,2)--(10,0)--cycle;
\foreach \x in {(6,0),(6,2),(6,4),(10,4),(8,2),(10,0),(2,0),(0,2),(2,2)} \draw \x circle (1pt);
\draw [dashed] (8,0)--(8,4);
\draw [dashed] (6,0)--(6,2);
\draw(8,4) node[above] {\tiny $\sigma'$};
\draw(8,0) node[below] {\tiny $\sigma'$};
\draw(0,3) node[left] {\tiny $C_1$};
\draw(1,1) node[left] {\tiny $C_3$};
\draw(12,3) node[right] {\tiny $C_2$};
\draw(7,2) node {\tiny $C''_2$};
\end{tikzpicture}
\end{minipage}
\caption{Cylinders in Proposition~\ref{prop:C3I:1:n-f:sim:cyl}: Case 1}
\label{C1orC2SimpFigP2}
\end{figure}

Recall that $C_1'$ is the vertical cylinder containing $C_1$.  Let $\cC'$ be the equivalence class of $C_1'$.  Then
$$1 = P(C_2, \cC') = P(C_1, \cC').$$
Thus, $C_2$ is filled by cylinders in $\cC'$.  However, the equality, $n_2 = n_3$ implies that no cylinder $\cM$-parallel to $C_1'$ can pass through $C_3$ more than once before entering $C_2$. We derive that as subsets of the borders of $C_3$, the bottom border of $C_2$ must lie right above the top border of $C_2$. Let $R_2$ be the subset of $C_3$ which is filled out by vertical trajectories joining saddle connections in the bottom border of $C_2$ to saddle connections in the top border of $C_2$. Remark that  $R_2$ is a union of rectangles in $C_3$.

Let $R_1$ denote the rectangle which is the intersection $C'_1\cap C_3$. Note that $R_2$ is disjoint from $C'_1$. We claim that $R_1\cup R_2 \varsubsetneq C_3$. Indeed, if it is the case then any saddle connection in the top border of $C_3$ is either contained in the bottom border of $C_1$ or the bottom border of $C_2$. Thus we have $c_3=c_1+c_2$, where $c_i$ is the core curve of $C_i$, which contradicts the hypothesis of Case 3.I).

Recall that $C_2$ is not a simple cylinder. Without loss of generality, let $C_2$ contain two or more saddle connections on its bottom that are identified to the top of $C_3$.  It is possible to twist $C_2$ (and $C_1$) so that there is a vertical closed trajectory passing once through each of $C_2$ and $C_3$  and intersects only one the saddle connections in the bottom of $C_2$.  Let $D_2$ denote the cylinder corresponding to this closed trajectory, then $D_2$ does not fill $C_2$.

Let $\cD$ be the equivalence class of $D_2$. There must exist a vertical cylinder $D_1$ in $\cD$ crossing $C_1$. Since $C_1$ is simple, the union of $C_1$ and $R_1$ is a slit torus. Thus $D_1$ must fill $C_1\cup R_1$, and we have $P(C_1,\cD)=1$. It follows that $P(C_2,\cD)=1$, and since $D_2$ does not fill $C_2$,  there must exist another cylinder $D'_2$ in $\cD$ that crosses $C_2$. Thus $\cD$ contains at least three vertical cylinders.  But $\cD$ does not fill $C_3$ since $R_1\cup R_2 \varsubsetneq C_3$. Thus by applying the result of \cite{SmillieWeissMinSets}, we can conclude that $\cM$ contains a vertically periodic surface with at least four vertical cylinders.

\bigskip

\noindent \textbf{Case 2:} $C_1$ is adjacent to both $C_2$ and $C_3$.
Without loss of generality, assume the bottom of $C_1$, denoted $\sigma'$, is attached to the top of $C_2$ and the top of $C_1$, denoted $\sigma$, is
identified to a saddle connection in the bottom of $C_3$.

We claim that after appropriate twisting and stretching, there is a cylinder $C_1'$ passing exactly once through every cylinder.  We first twist $C_1$ and $C_2$ so that $t_1=0$. Let $\sigma''$ be a saddle connection on the bottom of $C_2$ lying below $\sigma'$ (see Figure~\ref{C1orC2SimpFigP1}). By Lemma~\ref{lm:sim:cyl:adj:f:cyl}, $\sigma''$ must be identified to the top of $C_3$.  Consider a vertical trajectory $\gamma$ ascending from $\sigma''$ through $C_2$ and $C_1$ to $\sigma$, which passes through $\sigma'$ and no other saddle connection.  Then after twisting $C_3$ while fixing $C_1$ and $C_2$, we see that the copy of $\sigma$ in the bottom of $C_3$ can be arranged so that the trajectory $\gamma$ after traversing $C_3$ closes when intersecting $\sigma''$ again in the top of $C_3$.  This determines a cylinder $C_1'$ as claimed that passes exactly once through $C_i$, for each $i$.

Either $C_1'$ is free, or it is not.  If $C_1'$ is not free, then there is a cylinder $C_2'$ that is $\cM$-parallel to $C_1'$.  Let $h_i$ denote the height of $C_i$, for all $i$, and let $C_2'$ pass through the cylinder $C_i$, $n_i$ times.  The cylinders $C_1'$ and $C_2'$ must satisfy the equality $P(C_1', \{C_3\}) = P(C_2', \{C_3\})$, which yields
$$\frac{h_3}{\sum_i h_i} = \frac{n_3 h_3}{\sum_i n_ih_i},$$
and simplifies to
$$n_1h_1 + n_2h_2 = n_3(h_1 + h_2).$$
Some observations are in order.  Since every vertical trajectory passing downwards from $C_1$ enters $C_2$, and every vertical trajectory passing downwards from $C_2$ enters $C_3$ by Lemma~\ref{lm:sim:cyl:adj:f:cyl}, we have $n_3 \geq n_2 \geq n_1$.  The above equality can be transformed to
$$\left(1 - \frac{n_1}{n_3} \right)h_1 + \left(1 - \frac{n_2}{n_3} \right) h_2 = 0.$$
Noting that $n_1/n_3 \leq 1$, $n_2/n_3 \leq 1$, and obviously $h_1, h_2 > 0$ implies that this equality can only hold if $n_1 = n_2 = n_3$.  In other words, any cylinder $\cM$-parallel to $C_1'$ must pass through every horizontal cylinder an equal number of times.

If there are two cylinders that are $\cM$-parallel to $C_1'$, then we are done because we would have a surface vertically periodic with only one equivalence class of three cylinders (by Theorem~\ref{thm:RankkImpkCyl}).  Let $\mathcal{C}'$ be the equivalence class of cylinders $\cM$-parallel to $C_1'$.  Let $h_i'$ be the height of $C_i'$, for $i = 1,2$.  Let $\ell_i$ be the circumference of $C_i$, for $i = 1,2$.  Let $n$ be the number of times $C_2'$ passes through each horizontal cylinder.  Letting $n=0$ is equivalent to saying that $C_1'$ is free.  It will be clear to the reader that a contradiction is achieved regardless of the value of $n$.  We compute the portion $P(C_1, \mathcal{C}') = P(C_2, \mathcal{C}')$, which yields
$$\frac{h_1'h_1 + nh_2'h_1}{\ell_1 h_1} = \frac{h_1'h_2 + nh_2'h_2}{\ell_2 h_2}.$$
This simplifies to $\ell_1 = \ell_2$, which is clearly a contradiction because $C_2$ is a simple cylinder whose bottom is identified to the top of $C_1$, and so they cannot possibly have equal circumferences without being the same cylinder.

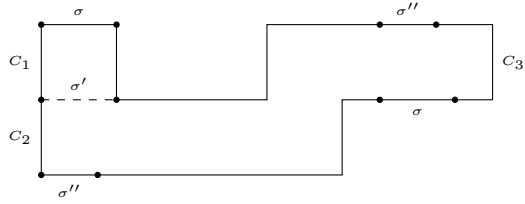
\begin{figure}[htb]
\centering
\begin{tikzpicture}[scale=0.50]
\draw (0,0)--(0,4)--(2,4)--(2,2)--(6,2)--(6,4)--(12,4)--(12,2)--(8,2)--(8,0)--cycle;
\foreach \x in {(0,0),(0,2),(0,4),(2,4),(2,2),(1.5,0),(9,2),(11,2),(9,4),(10.5,4)} \filldraw[fill=black] \x circle (2pt);
\draw [dashed] (0,2)--(2,2);
\draw(1,2) node[above] {\tiny $\sigma'$};
\draw(9.75,4) node[above] {\tiny $\sigma''$};
\draw(1,4) node[above] {\tiny $\sigma$};
\draw(10,2) node[below] {\tiny $\sigma$};
\draw(.75,0) node[below] {\tiny $\sigma''$};
\draw(0,3) node[left] {\tiny $C_1$};
\draw(0,1) node[left] {\tiny $C_2$};
\draw(12,3) node[right] {\tiny $C_3$};

\end{tikzpicture}
\caption{Cylinders in Proposition~\ref{prop:C3I:1:n-f:sim:cyl}: Case 2}
\label{C1orC2SimpFigP1}
\end{figure}

\end{proof}


\subsection{$C_1$ or $C_2$ Contains a Simple Cylinder}


\begin{proposition}
\label{prop:C3I:loop:in:c1:c2}
Let $\cM$ be a rank two affine manifold in a stratum in genus three with $k \geq 2$ zeros.  If $M \in \cM$ is a horizontally periodic translation surface satisfying Case 3.I) and $C_1$ or $C_2$ contains a saddle connection $\sigma$ on its top and bottom, i.e. $C_1$ or $C_2$ contains a simple cylinder, then there exists a horizontally periodic translation surface $M' \in \cM$ that either satisfies Case 3.I) and has two simple cylinders, or $M'$ has at least four cylinders.

In particular, if $k = 2$, then there exists $M' \in \cM$ horizontally periodic with four cylinders.
\end{proposition}

\begin{proof}
Without loss of generality, let $C_1$ be the cylinder with a saddle connection $\sigma$ on its top and bottom.  Let $C_1'$ be the simple cylinder in $C_1$ which is formed by the trajectories from $\sigma$ to itself. We can suppose that $C'_1$ is vertical.  We see that $C_1'$ cannot be free because $C_1$ is not free.  Therefore, there is a cylinder $C_2'$ that is $\cM$-parallel to $C_1'$.  Proposition~\ref{prop:SimpCylPres} guarantees that there is a nearby square-tiled  surface on which $C_1'$ is also a simple vertical cylinder.  By a slight abuse of notation, we denote this square-tiled surface by $M$.

Observe that $C_2' \subset C_1 \cup C_2$ because it must have zero proportion in $C_3$ to satisfy the equality $0 = P(C_1', \{C_3\}) = P(C_2', \{C_3\})$.  Thus, there are at least three vertical cylinders on $M$. If there are four vertical cylinders then we are done. Therefore, we only need to consider the case where $M$ is decomposed into three cylinders in the vertical direction. Let $C'_3$ be the third vertical cylinder which necessarily crosses $C_3$. Clearly, $C'_3$ is not $\cM$-parallel to $C'_1$ and $C'_2$.

Consider the cylinder decomposition in the vertical direction. We know that Case 3.II) is excluded by Corollary~\ref{cor:no:C3II:in:Hmn}.  Case 3.III) can be excluded by noting that $C_1'$ is a non-free simple cylinder and that can never occur in Case 3.III) (Lemma~\ref{lm:C3III:no:nf:sim:cyl}).  Therefore, we must be in Case 3.I).  If $C'_2$ is not simple, we conclude by Proposition~\ref{prop:C3I:1:n-f:sim:cyl}.  Otherwise, $C'_1$ and $C'_2$ are both simple.  The final claim follows from Proposition \ref{prop:3CylsI2SimpCyls}.
\end{proof}


\subsection{$C_3$ Is Simple}

\begin{proposition}
 \label{prop:C3I:c3:is:sim}
Let $\cM$ be a rank two affine manifold in a stratum in genus three with $k \geq 2$ zeros.  If there exists  $M \in \cM$ which is horizontally periodic satisfying Case 3.I) such that $C_3$ is a simple cylinder, then there exists a horizontally periodic translation surface $M' \in \cM$ that either satisfies Case 3.I) and had two simple cylinders, or $M'$ has at least four cylinders.

In particular, if $k = 2$, then there exists $M' \in \cM$ with four cylinders.
\end{proposition}
\begin{proof}
We have two cases:
\begin{itemize}
\item[$\bullet$] $C_3$ is only adjacent to one of $C_1,C_2$. Without loss of generality, assume that $C_3$ is only adjacent to $C_1$, then Lemma~\ref{lm:f:sim:cyl} implies that $C_1$ is also free, which contradicts the hypothesis that $C_1$ and $C_2$ are $\cM$-parallel.

\item[$\bullet$] $C_3$ is adjacent to both of $C_1,C_2$. We can assume that the top border of $C_3$ is contained in the bottom border of $C_1$ and the bottom border of $C_3$ is contained in the top border of $C_2$. If one of $C_1$ and $C_2$ contains a saddle connection in both its top and bottom borders, then we are done by Proposition~\ref{prop:C3I:loop:in:c1:c2}. Otherwise, the top border of $C_1$ and the bottom border of $C_2$ contain the same saddle connections, which means that the core curves of $C_1$ and $C_2$ are homologous, thus we have a contradiction to the assumption of Case 3.I).
\end{itemize}
\end{proof}

As an immediate consequence, we have

\begin{proposition}
\label{prop:C3I:1sim:cyl}
Let $\cM$ be a rank two invariant submanifold in $\cH(m,n)$, with $m+n=4$. Assume that $\cM$ contains a horizontally periodic surface $M$ with three horizontal cylinders in Case 3.I), and one of the cylinders is simple. Then $\cM$ contains a horizontally periodic surface with at least four cylinders.
\end{proposition}
\begin{proof}
 If all of the horizontal cylinders of $M$ are $\cM$-parallel, then $\Tw(M,\cM) \neq  \CP(M,\cM)$, and we can conclude by Lemma~\ref{lm:WrightTwistPresLem}. By Lemma~\ref{lm:3:f:cyl} we know that the cylinders cannot be all free. Thus we only need to consider  the case where there are two equivalence classes. If there are two simple cylinders, then we can conclude by Proposition~\ref{prop:3CylsI2SimpCyls}. Suppose that there is only one simple cylinder.  If the simple cylinder is $\cM$-parallel to another cylinder, then the proposition follows from Proposition~\ref{prop:C3I:1:n-f:sim:cyl}. Otherwise we have a free simple cylinder, and the proposition follows from Proposition~\ref{prop:C3I:c3:is:sim}.
\end{proof}

\begin{proposition}
\label{prop:C3I:loop}
Let $\cM$ be a rank two invariant submanifold in $\cH(m,n)$, with $m+n=4$. Assume that $\cM$ contains a horizontally periodic  surface $M$ with three horizontal cylinders. If there exists a horizontal saddle connection $\s$ that is contained in both top and bottom border of the same cylinder, then $\cM$ contains a horizontally periodic surface with four cylinders.
\end{proposition}
\begin{proof}
  Let $C$ be the horizontal cylinder that contains $\s$ in both top and bottom borders. There exists a simple cylinder $D \subset C$ crossing only $\s$. We can suppose that $D$ is vertical. By Lemma~\ref{lm:sim:cyl:2:cl}, there exists a square-tiled surface $M'$ close to $M$ on which $D$ is also a vertical simple cylinder, and there are at least two equivalence classes of vertical cylinders.
  From the proof Lemma~\ref{lm:Hmn:2cyl:imp:3cyl:1:sim}, we derive that $M'$ has at least three vertical cylinders. Assume that $M'$ has exactly three vertical cylinders. Consider the cylinder decomposition in the vertical direction of $M'$. Recall that Case 3.II) is excluded by Corollary~\ref{cor:no:C3II:in:Hmn}. If this cylinder decomposition satisfies Case 3.III), then we are done by Proposition~\ref{prop:3III:imply:4cyls}, otherwise we are in Case 3.I), and the proposition follows from Proposition~\ref{prop:C3I:1sim:cyl}.
\end{proof}

\begin{remark}
Observe that the property of a cylinder containing the same saddle connection on top and bottom is equivalent to saying a cylinder contains a simple cylinder.
\end{remark}

\subsection{Semi-Simple Cylinders}
\begin{definition}\label{def:semi-sim:cyl}
A cylinder is \emph{semi-simple} if the boundary of one side of the cylinder consists of a single saddle connection.  Of course, simple cylinders are semi-simple.  We say that a cylinder is \emph{strictly semi-simple} if it is semi-simple, but not simple.
\end{definition}

\begin{proposition}
\label{prop:C3I:semi-sim:cyl}
Let $\cM$ be a rank two invariant submanifold in a stratum with exactly two zeros in genus three. If $\cM$ contains a horizontally periodic surface $M$ with three horizontal cylinders in Case 3.I) and one of the cylinders is strictly semi-simple, then there exists $M' \in \cM$ with four cylinders.
\end{proposition}

We first show the following

\begin{lemma}
\label{lm:2adj:cyl:n:para}
Let $C$ and $D$ be two horizontal cylinders in $M$.  Assume that there are two horizontal saddle connections $\s_1,\s_2$ such that
\begin{itemize}
 \item $\s_1$ is contained in the top boundary of $C$ and in the bottom boundary of $D$

 \item $\s_2$ is contained in the bottom boundary of $C$ and in the top boundary of $D$
\end{itemize}
If $C$ and $D$ are not $\cM$-parallel, then there exists in $\cM$ a surface $M'$ admitting a cylinder decomposition in the vertical direction with at least one simple vertical cylinder, and two equivalence classes of parallel cylinders.
\end{lemma}
\begin{proof}
We can twist $C$ and $D$ so that there exists a simple vertical cylinder $V$ crossing only $\s_1$ and $\s_2$ (among the horizontal saddle connections) and contained in $\overline{C}\cup \overline{D}$ (see Figure~\ref{fig:2adj:cyl:n:para}).

\begin{figure}[htb]
  \centering
  \begin{tikzpicture}[scale=0.50]
  \fill[blue!30] (2,4) -- (2,0) -- (3,0) -- (3,4) -- cycle;

  \draw[thin, dashed] (0,0) -- (2,0) (5,0) -- (6,0) (0,2) -- (1,2) (3,2) -- (6,2) (0,4) -- (2,4) (5,4) -- (6,4);
  \draw (2,0) -- (5,0) (1,2) -- (3,2) (2,4) -- (5,4);
  \draw (2,0) -- (2,4) (3,0) -- (3,4);

  \foreach \x in {(2,0),(5,0),(1,2),(3,2),(2,4),(5,4)} \filldraw[fill=black] \x circle (2pt);

  \draw(4,4) node[above] {\tiny $\sigma_1$};
  \draw(1.5,2) node[below] {\tiny $\sigma_2$};
  \draw(3.5,0) node[below] {\tiny $\sigma_1$};

  \draw(0,3) node {\tiny $C$};
  \draw(0,1) node {\tiny $D$};
  \draw (2.5,3) node {\tiny $V$};
  \end{tikzpicture}
\caption{}
\label{fig:2adj:cyl:n:para}
\end{figure}

From Proposition~\ref{prop:SimpCylPres}, there exists a surface $M'$ in $\cM$ admitting a cylinder decomposition in the vertical direction with a simple cylinder. By Lemma~\ref{lm:sim:cyl:2:cl}, we can find such a surface with at least two equivalence classes of vertical cylinders.
\end{proof}

\bigskip
\begin{proof}[Proof of Proposition~\ref{prop:C3I:semi-sim:cyl}]
We only need to consider the case where the horizontal cylinders of $M$ belong to two equivalence classes. As usual we denote the three horizontal cylinders of $M$ by $C_1,C_2,C_3$, where $C_1,C_2$ are $\cM$-parallel and $C_3$ is free. By Proposition~\ref{prop:C3I:loop}, we can assume that no horizontal saddle connection is contained in both top and bottom of the same cylinder. We will show that there always exist two cylinders that satisfy the conditions of Lemma~\ref{lm:2adj:cyl:n:para}.

\begin{itemize}
 \item[$\bullet$] {\bf Case 1:} One of $C_1$ and $C_2$ is strictly semi-simple. Without loss of generality, we can assume that $C_1$ is strictly semi-simple and the bottom border of $C_1$ consists of a single saddle connection. We have two subcases:
 \begin{itemize}
 \item[-] \underline{Case 1.a:} The bottom border of $C_1$ is included in the top border of $C_2$.  Let $\s_1$ be a saddle connection in the top border of $C_2$, which is not the one in the bottom of $C_1$.  By the hypothesis that no saddle connection is contained in both top and bottom borders of the same cylinder, we derive  that $\s_1$ must be contained in the bottom border of $C_3$. Note also that any  saddle connection in the top border of $C_3$ must be contained in the bottom border of $C_2$, since it cannot be contained in the bottom of $C_1$ nor $C_3$.  Therefore there is a saddle connection $\s_2$ in the bottom of $C_2$ that is contained in the top of $C_3$.

\item[-] \underline{Case 1.b:} The bottom border of $C_1$ is contained in the top border of $C_3$. By the hypothesis, any saddle in the top of $C_3$ cannot belong to the bottom of $C_3$, thus there exists a saddle connection $\s_1$ in the top of $C_3$ that is contained in the bottom of $C_2$. By the same assumption, any saddle connection in the top of $C_2$ either belongs to the bottom of $C_1$ or $C_3$. But the bottom of $C_1$ is already contained in the top of $C_3$, thus there must exists a saddle connection $\s_2$ in the top of $C_2$, which is contained in the bottom of $C_3$.
\end{itemize}

\item[$\bullet$] {\bf Case 2:} $C_3$ is semi-simple. Without loss of generality, we can assume that  the bottom border of $C_3$ consists of a single saddle connection, which  is contained in the top of $C_2$. We claim that there exists a saddle connection in the top of $C_3$ that is contained in the bottom of $C_2$. Assume that no saddle connection in the top of $C_3$ is contained in the bottom of $C_2$. Note that, by assumption, all of the saddle connections in the top of $C_1$  must be contained in the bottom of $C_2$. It follows that the top of $C_1$ and the bottom of $C_2$ contain the same set of saddle connections, which means that $C_1$ and $C_2$ are homologous. Therefore, we have a contradiction to the hypothesis of Case 3.I).
\end{itemize}

In all cases let $V$ be the vertical simple cylinder crossing each of $C_2$ and $C_3$ once, whose existence is guaranteed by Lemma~\ref{lm:2adj:cyl:n:para}. Since $V$ does not cross $C_1$, there must exist  another vertical cylinder $V'$, which is $\cM$-parallel to $V$ crossing $C_1$. By Lemma~\ref{lm:sim:cyl:2:cl}, we can find a square-tiled surface close to $M$ on which $V$ and $V'$ persist, $V$ remains simple, and $M'$ has at least two equivalence classes of vertical cylinders. In particular, $M'$ has at least three vertical cylinders. If $M'$ has four vertical cylinders, then we are done. Otherwise, observe that we are not in Case 3.II) or Case 3.III) since one of the non-free vertical cylinders is simple. Thus we are in Case 3.I), and we conclude by Propositions \ref{prop:C3I:1:n-f:sim:cyl} and \ref{prop:C3I:c3:is:sim}.
\end{proof}

\subsection{Case 3.I): The Exceptional Case}
\label{ExceptionalCaseSect}

\begin{figure}[htb]
\centering
\begin{minipage}{0.24\linewidth}
\centering
\begin{tikzpicture}[scale=0.50]
\draw (0,0)--(0,2)--(-2,2)--(-2,4)--(-4,4)--(-4,6)--(0,6)--(0,4)--(2,4)--(2,2)--(4,2)--(4,0)--cycle;
\foreach \x in {(0,0),(0,2),(-2,2),(-2,4),(-4,4),(-4,6),(0,6),(0,4),(2,4),(2,2),(4,2),(4,0),(2,0),(-2,6)} \draw \x circle (1pt);
\draw(-3,6) node[above] {\tiny $2$};
\draw(-1,6) node[above] {\tiny $4$};
\draw(1,4) node[above] {\tiny $1$};
\draw(3,2) node[above] {\tiny $3$};
\draw(1,0) node[below] {\tiny $1$};
\draw(3,0) node[below] {\tiny $2$};
\draw(-1,2) node[below] {\tiny $4$};
\draw(-3,4) node[below] {\tiny $3$};
\draw(-2,5) node {\tiny $C_1$};
\draw(0,3) node {\tiny $C_2$};
\draw(2,1) node {\tiny $C_3$};
\end{tikzpicture}
\end{minipage}
\caption{The exceptional 3-cylinder diagram satisfying  Case 3.I)}
\label{ExCaseFig}
\end{figure}
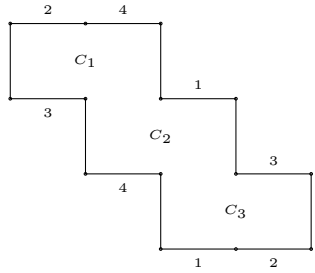

\begin{lemma}
\label{3CylIException}
Given a horizontally periodic translation surface $M \in \mathcal{H}(m,n)$ such that $M$ has exactly three cylinders $C_1, C_2, C_3$ satisfying Case 3.I), there exists exactly one $3$-cylinder diagram such that
\begin{itemize}
\item no horizontal saddle connection is contained in both the top and bottom of the same cylinder, and
\item none of the $C_i$ is semi-simple.
\end{itemize}
The unique $3$-cylinder diagram satisfying these conditions is given in Figure \ref{ExCaseFig}.
\end{lemma}

\begin{proof}
Assume that the cylinder decomposition in the horizontal direction of $M$ satisfies the two conditions of the lemma. Let $\cD$ be the dual graph of this cylinder decomposition. By definition, $\cD$ has three vertices and six edges, any vertex has valency at least four and no loop. Since total valency of the vertices is $12$, all vertices must have valency equal to four.  If there is a pair of vertices that are connected by only one edge, then each vertex in this pair is connected to the remaining one by three edges, and we get a contradiction. It follows that any pair of vertices are connected by two edges, this gives us a unique choice for $\cD$.

\begin{figure}[htb]
\begin{minipage}[t]{0.4\linewidth}
\centering
\begin{tikzpicture}[scale=0.4, inner sep=0.2mm, vertex/.style={circle, draw=black, fill=blue!30, minimum size=1mm},>= stealth]

\node (above) at (0,4) [vertex] {\tiny $C_1$};
\node (left) at (-3,0) [vertex] {\tiny $C_2$};
\node (right) at (3,0) [vertex] {\tiny $C_3$};

\draw[->] (above) to (left) ;
\draw[->] (left) to [out=75, in=210] (above);

\draw[->] (left) to (right);
\draw[->] (right) to[out=210, in=-30] (left);

\draw[->] (right) to (above);
\draw[->] (above) to[out=-30, in=105] (right);

\draw[thin] (-4,-2) -- (0,-2) -- (0,-4) -- (2,-4) -- (2,-6) -- (4,-6) -- (4,-8) -- (0,-8) -- (0,-6) -- (-2,-6) -- (-2,-4) -- (-4,-4) -- cycle;
\draw (-2,-4) -- (0,-4) (0,-6) -- (2,-6);

\foreach  \x in {(-4,-2), (-2,-4), (0,-2), (0,-6), (2,-8), (2,-4), (4,-6)} \filldraw[fill=white] \x circle (3pt);
\foreach \x in {(-4,-4), (-2,-2), ( -2,-6), (0,-4), (0,-8), (2,-6), (4,-8)} \filldraw[fill=black] \x circle (3pt);

\draw (-3,-2) node[above] {\tiny 1} (3,-8) node[below] {\tiny 1} (-1,-2) node[above] {\tiny 2} (-1,-6) node[below] {\tiny 2} (1,-4) node[above] {\tiny 3} (1,-8) node[below] {\tiny 3} (-3,-4) node[below] {\tiny 4} (3,-6) node[above] {\tiny 4} (-1,-4) node[above] {\tiny 5} (1,-6) node[above] {\tiny 6};

\draw (-3,-3) node {\tiny $C_1$} (0,-5) node {\tiny $C_2$} (3,-7) node {\tiny $C_3$};

\end{tikzpicture}
\end{minipage}
\begin{minipage}[t]{0.4\linewidth}
\centering
\begin{tikzpicture}[scale=0.4, inner sep=0.2mm, vertex/.style={circle, draw=black, fill=blue!30, minimum size=1mm},>= stealth]

\node (above) at (0,4) [vertex] {\tiny $C_1$};
\node (left) at (-3,0) [vertex] {\tiny $C_2$};
\node (right) at (3,0) [vertex] {\tiny $C_3$};

\draw[->] (above) to (left) ;
\draw[->] (left) to [out=75, in=210] (above);

\draw[->] (left) to (right);
\draw[->] (left) to[out=-30, in=210] (right);

\draw[->] (right) to (above);
\draw[->] (above) to[out=-30, in=105] (right);

%

\draw[thin] (-4,-2) -- (-4,-8) -- (2,-8) -- (2,-6) -- (4,-6) -- (4,-4) -- (2,-4) -- (2,-2) -- cycle;
\draw[thin] (-4,-4) -- (2,-4) (-4,-6) -- (2,-6) ;

\foreach  \x in {(-4,-2), (-4,-6),(-2,-8), (0,-2), (2,-2), (2,-6), (4,-6)} \filldraw[fill=white] \x circle (3pt);
\foreach \x in {(-4,-4), (-4,-8),(-2,-2), (0,-8), (2,-4), (2,-8), (4,-4)} \filldraw[fill=black] \x circle (3pt);

\draw (-3,-2) node[above] {\tiny 1} (-1,-8) node[below] {\tiny 1} (-1,-2) node[above] {\tiny 2} (-3,-8) node[below] {\tiny 2} (1,-2) node[above] {\tiny 3} (3,-6) node[below] {\tiny 3} (-1,-4) node[above] {\tiny 4} (3,-4) node[above] {\tiny 5} (1,-8) node[below] {\tiny 5} (-1,-6) node[above] {\tiny 6};

\draw (1,-5) node {\tiny $C_1$} (-3,-7) node {\tiny $C_2$} (-3,-3) node {\tiny $C_3$};
\end{tikzpicture}

\end{minipage}

\caption{Admissible dual graphs having three vertices with valency equal four and no loop: the one on the right gives a diagram with two semi-simple cylinders.}
\label{fig:adm:DG:3vert:val4:noloop}
\end{figure}
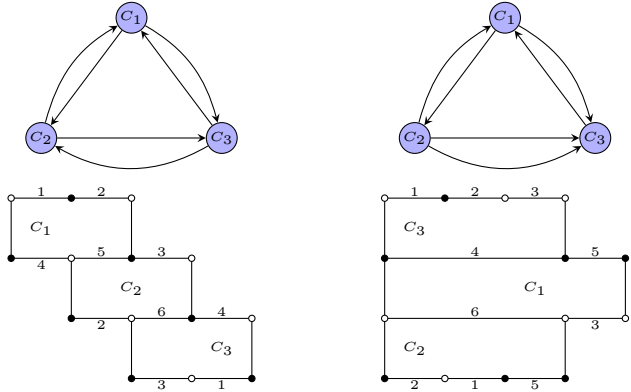

To get the cylinder diagram we need first an orientation for each edge. There are only two configurations of the orientations that are not forbidden, and to each of these configurations we have a unique set of compatible cyclic orderings at the vertices. We therefore have two corresponding cylinder diagrams as shown in Figure~\ref{fig:adm:DG:3vert:val4:noloop}. But in one of the diagrams we have semi-simple cylinders. Thus there is only one diagram that satisfies the conditions of the lemma, it is easy to see that this diagram gives a surface in $\cH(2,2)^{\rm odd}$.
\end{proof}

We call the translation surface in Figure \ref{ExCaseFig} \emph{the exceptional case}.

%
%

\begin{lemma}
\label{lm:ExceptionalCaseReduction}
If a rank two manifold $\cM$ contains a horizontally periodic surface with three cylinders arranged in the exceptional case, then $\cM$ contains a translation surface with four cylinders.
\end{lemma}

\begin{proof}
It suffices to assume that $\Tw(M, \cM) = \CP(M, \cM)$ by Lemma~\ref{lm:WrightTwistPresLem}, so there are two equivalence classes.  Remark that in this exceptional case, the combinatorial properties of every cylinder are the same. Thus without loss of generality, we can assume that $C_1,C_2$ are $\cM$-parallel, and $C_3$ is free.

Observe that  the hypothesis of Lemma~\ref{lm:2adj:cyl:n:para} is satisfied.  Thus, $\cM$ contains a translation surface $M$ containing a vertical simple cylinder $V$ crossing $C_2$ and $C_3$.  Since $V$ does not cross $C_1$, there must exist another vertical cylinder $V'$ in the equivalence class of $V$ that crosses $C_1$. By Proposition~\ref{prop:SimpCylPres}, we can find in a neighborhood of $M$ in $\cM$ a square-tiled surface on which $V$ and $V'$ persist, and $V$ remains simple. By Lemma~\ref{lm:sim:cyl:2:cl}, we get another square-tiled surface $M''$ close to $M'$ (hence $M''$ satisfies the same properties regarding $V$ and $V'$), on which we have at least two equivalence classes of vertical cylinders. In particular, $M''$ has at least three vertical cylinders. Since we have one simple cylinder that is not free, we are not in Case 3.III).  Since Case 3.II) is excluded by Lemma~\ref{lm:3CylsCaseIILem}, we conclude that the cylinder decomposition in the vertical direction
  of $M''$ satisfies Case 3.I).
Therefore, the lemma follows from Proposition~\ref{prop:C3I:1sim:cyl}.
\end{proof}

\subsection{Proof of Theorem~\ref{thm:C3I:imply:4cyl}}

\begin{proof}[Proof of Thm. \ref{thm:C3I:imply:4cyl}]
By Lemma \ref{3CylIException}, every cylinder diagram satisfying Case 3.I) either has a cylinder containing a simple cylinder, has a cylinder which is semi-simple, or satisfies the exceptional case.

By Proposition \ref{prop:C3I:semi-sim:cyl}, if a translation surface with three cylinders in $\cM$ satisfies Case 3.I) and has a strictly semi-simple cylinder, then either there is a translation surface in $\cM$ with four cylinders and we are done, or there is a translation surface satisfying Case 3.I) with two simple cylinders.  In the latter case, we can conclude by Proposition \ref{prop:3CylsI2SimpCyls}.  If the translation surface has a simple cylinder, then we conclude by Propositions \ref{prop:C3I:1:n-f:sim:cyl} and \ref{prop:C3I:c3:is:sim}, which address every possibility for a simple cylinder in the cylinder diagram.

If one of the cylinders contains a simple cylinder, then we conclude by Proposition \ref{prop:C3I:loop}.

Finally, if the cylinder diagram satisfies the exceptional case, then we conclude by Lemma \ref{lm:ExceptionalCaseReduction}.
\end{proof}

\begin{corollary}
 \label{cor:Hmn:4cyl:diag}
 Let $\cM$ be a rank two invariant submanifold in one of the strata $\cH(m,n)$, with $m+n=4$. Then $\cM$ contains a horizontally periodic surface with four cylinders.
\end{corollary}

\section{Rank Two Invariant Submanifolds in $\cH(m,n)$, $m + n = 4$}

\subsection{$4$-Cylinder Diagrams in $\mathcal{H}(m,n)$, where $m + n = 4$}

The following lemma is valid for all translation surfaces in genus three.

\begin{lemma}
\label{lm:4CylDeg}
If a translation surface $M$ in genus three decomposes into four cylinders, then pinching the core curves of those cylinders degenerates the surface (in the sense of Lemma \ref{3CylDeg}) to one of four possible surfaces:
\begin{itemize}
\item 4.I) Two spheres joined by four pairs of simple poles.
\item 4.II) Two spheres joined by two pairs of simple poles such that each sphere has a pair of simple poles.
\item 4.III) Two spheres joined by three pairs of simple poles such that one sphere carries an additional pair of simple poles.
\item 4.IV) Two spheres and a torus such that the spheres have three simple poles and the torus has two simple poles.
\end{itemize}

If $M\in \cH(m,n)$, then Case 4.IV) cannot occur.
\end{lemma}

\begin{proof}
The core curves of the four cylinders must always be linearly dependent in homology in genus three.  This implies that the degenerate surface must have at least two parts.  There must be at least two pairs of poles between the parts because the core curve of a cylinder is never a separating curve for an Abelian differential.  If the two remaining pairs of poles were on the same part, then this would be impossible in genus three because every sphere must have at least three simple poles.  If each pair of poles lies on a different part, this yields Case 4.II).

If there are three pairs of poles between the two parts, then the other pair of poles lies on one of the parts which implies that both parts have genus zero and all pairs of poles are accounted for.  This yields Case 4.III).

If there are four pairs of poles between the two parts, then all of them are accounted for, and the surface satisfies Case 4.I).

Next we consider the case of three parts.  Each part must have at least two simple poles, which accounts for at least three of the four pairs of poles.  If the fourth cylinder were contained on a single part, then the condition that a sphere must carry at least three simple poles implies that we would have a surface of genus greater than three.  Hence, the fourth cylinder lies between two parts, and the only part with two simple poles is forced to be a torus.  This is exactly Case 4.IV).

Finally, there cannot be four parts with four pairs of poles in genus three.

Since each part must contain a zero, if $M \in \cH(m,n)$, then we only have two parts, which means that Case 4.IV) cannot occur.
\end{proof}

\

\subsection*{Conventions and Notation:}
\begin{itemize}
\item[$\bullet$] In what follows, we denote by $c_i$ a core curve of $C_i$, and by $\alpha_i$ the corresponding element in $(T^{\bR}_{M}\cM)^\ast$,
$h(C_i),\wth(C_i),\twist(C_i)$ are respectively the height, width (circumference), and twist of $C_i$. We sometimes write $h_i,\wth_i$, and $\twist_i$  instead of $h(C_i), \wth(C_i), \twist(C_i)$.

\item[$\bullet$] By ``twisting'' or ``shearing''  a family of cylinders, we will mean applying a matrix $\left(\begin{smallmatrix} 1 & t \\ 0 & 1 \end{smallmatrix} \right)$ to each cylinder in this family, and keeping the other cylinders unchanged.

\item[$\bullet$] Let us suppose that $M$ is horizontally periodic, and denote by $C_1,\dots,C_n$ its horizontal cylinders. For each $i \in \{1,\dots,n\}$, we pick a saddle connection $s_i$ contained in $C_i$ joining a zero in the bottom border and a zero in the top border of $C_i$. Remark that every element of $H^1(M,\Sigma,\bR)$ is completely determined by its values on the $s_i$ and the horizontal saddle connections. Let $\Twi_i$ be the element in $H^1(M,\Sigma;\bR)$ satisfying $\Twi_i(s_j)=\delta_{ij}$, and $\Twi_i(s)=0$ for any horizontal saddle connection $s$ on $M$. Twisting a cylinder $C_i$ gives a path in the stratum,  the tangent vector to this path is $h_i\Twi_i$. If we twist a family of cylinders $C_{i_1},\dots,C_{i_k}$ simultaneously, then the tangent vector to the corresponding path in the stratum is given by $v=h_{i_1}\Twi_{i_1}+\dots+h_{i_k}\Twi_{i_k}$. Note that $\Twi_i$ vanishes on all the core curves of the horizontal cylinders.
\end{itemize}

\subsection{$\mathcal{H}(3,1)$}

We first observe

\begin{lemma}
\label{lm:H31:C4III}
If $M$ is a horizontally periodic translation surface with four cylinders in $\mathcal{H}(3,1)$, then $M$ satisfies Case 4.III).
\end{lemma}

\begin{proof}
By the formula for the degree of the canonical bundle, $\sharp(\text{zeros}) - \sharp(\text{poles}) = 2g'-2$, where $g'$ is the genus of a part of a degenerate Riemann surface, and $\sharp(\cdot)$ is the sum of the orders of the elements in the set.  By inspection of the parts of the surface for each of the degenerate surfaces in Lemma~\ref{lm:4CylDeg}, and by counting the total order of the poles on each part of the degenerate surface, we immediately see that only the sphere carrying five simple poles in Case 4.III) can admit a zero of order three because the formula above would read $3-5=-2$.  Therefore, it is the only possible case in $\mathcal{H}(3,1)$.
\end{proof}

From Proposition~\ref{prop:DG:complete}, we have

\begin{lemma}
 \label{lm:H31:4cyl:diag}
 Let $M$ be a horizontally periodic surface in $\cH(3,1)$. If $M$ has four horizontal cylinders, then the cylinder decomposition in the horizontal direction is given by one of the diagrams in Figure~\ref{fig:H31:4cyl:diag}.

 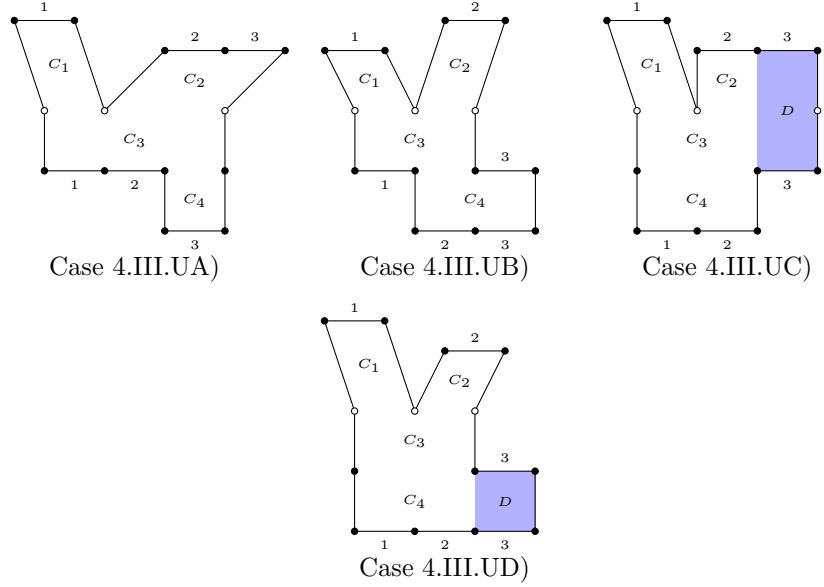
\begin{figure}[htb]
 \centering
 \begin{minipage}[t]{0.3\linewidth}
 \centering
 \begin{tikzpicture}[scale=0.4]
 \draw (-1,7) -- (0,4) -- (0,2) -- (4,2) -- (4,0) -- (6,0) -- (6,4) -- (8,6) -- (4,6) -- (2,4) -- (1,7) -- cycle;
 \foreach \x in {(-1,7), (0,2), (1,7),(2,2),(4,6),(4,2),(4,0), (6,6), (6,2),(6,0), (8,6)} \filldraw[fill=black] \x circle (3pt);
 \foreach \x in {(0,4),(2,4),(6,4)} \filldraw[fill=white] \x circle (3pt);

 \draw (0,7) node[above] {\tiny $1$} (1,2) node[below] {\tiny $1$} (5,6) node[above] {\tiny $2$} (3,2) node[below] {\tiny $2$} (7,6) node[above] {\tiny $3$} (5,0) node[below] {\tiny $3$};

 \draw (0.5,5.5) node {\tiny $C_1$} (5,5) node {\tiny $C_2$} (3,3) node {\tiny $C_3$} (5,1) node {\tiny $C_4$};

 \draw  (3,-2) node[above] {Case 4.III.UA)};
 \end{tikzpicture}
 \end{minipage}
\begin{minipage}[t]{0.3\linewidth}
\centering
 \begin{tikzpicture}[scale=0.4]
 \draw (-1,6) -- (0,4) -- (0,2) -- (2,2) -- (2,0) -- (6,0) -- (6,2) -- (4,2) -- (4,4) -- (5,7) -- (3,7) -- (2,4) -- (1,6) -- cycle;
 \foreach \x in {(-1,6), (0,2), (1,6),(3,7), (2,2), (2,0), (5,7),(4,2),(4,0), (6,2),(6,0)} \filldraw[fill=black] \x circle (3pt);
  \foreach \x in {(0,4),(2,4),(4,4)} \filldraw[fill=white] \x circle (3pt);

  \draw (0,6) node[above] {\tiny $1$} (1,2) node[below] {\tiny $1$} (4,7) node[above] {\tiny $2$} (3,0) node[below] {\tiny $2$} (5,2) node[above] {\tiny $3$} (5,0) node[below] {\tiny $3$};
  \draw (0.5,5) node {\tiny $C_1$} (3.5,5.5) node {\tiny $C_2$} (2,3) node {\tiny $C_3$} (4,1) node {\tiny $C_4$};

 \draw  (3,-2) node[above] {Case 4.III.UB)};
\end{tikzpicture}
\end{minipage}
\begin{minipage}[t]{0.3\linewidth}
\centering
\begin{tikzpicture}[scale=0.4]
\fill[blue!30] (4,6) -- (4,2) -- (6,2) -- (6,6) -- cycle;
\draw (-1,7) -- (0,4) -- (0,0) -- (4,0) -- (4,2) -- (6,2) -- (6,6) -- (2,6) -- (2,4) -- (1,7)  -- cycle;
 \foreach \x in {(-1,7), (0,2), (0,0), (1,7), (2,6), (2,0), (4,6), (4,2), (4,0), (6,6), (6,2)} \filldraw[fill=black] \x circle (3pt);
  \foreach \x in {(0,4),(2,4),(6,4)} \filldraw[fill=white] \x circle (3pt);

  \draw (0,7) node[above] {\tiny $1$} (1,0) node[below] {\tiny $1$} (3,6) node[above] {\tiny $2$} (3,0) node[below] {\tiny $2$} (5,6) node[above] {\tiny $3$} (5,2) node[below] {\tiny $3$};

  \draw (0.5,5.5) node {\tiny $C_1$} (3,5) node {\tiny $C_2$} (2,3) node {\tiny $C_3$} (2,1) node {\tiny $C_4$} (5,4) node {\tiny $D$};

 \draw  (3,-2) node[above] {Case 4.III.UC)};
\end{tikzpicture}
\end{minipage}
\begin{minipage}[b]{0.5\linewidth}
\centering
\begin{tikzpicture}[scale=0.4]
\fill[blue!30] (4,2) -- (4,0) -- (6,0) -- (6,2) -- cycle;
\draw (-1,7) -- (0,4) -- (0,0) -- (6,0) -- (6,2) -- (4,2) -- (4,4) -- (5,6) -- (3,6) -- (2,4) -- (1,7)  -- cycle;
 \foreach \x in {(-1,7), (0,2), (0,0), (1,7), (3,6), (2,0), (5,6), (4,2), (4,0), (6,2), (6,0)} \filldraw[fill=black] \x circle (3pt);
  \foreach \x in {(0,4),(2,4),(4,4)} \filldraw[fill=white] \x circle (3pt);

  \draw (0,7) node[above] {\tiny $1$} (1,0) node[below] {\tiny $1$} (4,6) node[above] {\tiny $2$} (3,0) node[below] {\tiny $2$} (5,2) node[above] {\tiny $3$} (5,0) node[below] {\tiny $3$};

  \draw (0.5,5.5) node {\tiny $C_1$} (3.5,5) node {\tiny $C_2$} (2,3) node {\tiny $C_3$} (2,1) node {\tiny $C_4$} (5,1) node {\tiny $D$};

 \draw  (3,-2) node[above] {Case 4.III.UD)};
\end{tikzpicture}
\end{minipage}
\caption{$4$-cylinder diagrams in $\cH(3,1)$}
\label{fig:H31:4cyl:diag}
 \end{figure}
\end{lemma}


\begin{theorem}
\label{H31Rank2Class}
There are no rank two affine manifolds in $\mathcal{H}(3,1)$.
\end{theorem}
\begin{proof}
 Let $\cM$ be an invariant rank two submanifold in $\cH(3,1)$. By Corollary~\ref{cor:Hmn:4cyl:diag}, $\cM$ contains a horizontally periodic surface with four cylinders. From Lemma~\ref{lm:H31:4cyl:diag}, the cylinder decomposition of $M$ is given by one of the diagrams in Figure~\ref{fig:H31:4cyl:diag}.

 Since a horizontally periodic surface in $\cH(3,1)$ has at most four horizontal cylinders, Theorem~\ref{thm:RankkImpkCyl} implies that the cylinders of $M$ must fall into at least two equivalence classes.  We will find a contradiction for each of these diagrams. In what follows, we refer to Figure~\ref{fig:H31:4cyl:diag} for notations and details of the proofs.

\begin{itemize}
 \item[a)] {\bf Case 4.III.UA):} Observe that we have $c_3=c_1+c_2 \in H_1(M,\bZ)$. Therefore, if two of $C_1,C_2,C_3$ are $\cM$-parallel, then they all belong to the same equivalence class. Consequently, if there are exactly two equivalence classes, then the classes must be $\{C_1,C_2,C_3\}$ and $\{C_4\}$. But from Lemma \ref{lm:n:f:sim:cyl}, we know that $C_1$ and $C_3$ are not $\cM$-parallel. Thus there must be at least three equivalence classes.  In particular, either $C_1$ or $C_2$ is free.

 Assume that  $C_1$ is free. We can then collapse it so that there is a unique saddle connection that crosses $C_1$ that is reduced to a point. The resulting surface $M'$ is contained in $\cH(4)$. By Proposition~\ref{prop:RankkImpRankkBd} and Theorem \ref{NWANWThm}, it follows that $M'$ is contained in $\tilde{\cQ}(3,-1^3) \subset \cH(4)^{\rm odd}$. But a surface in $\tilde{\cQ}(3,-1^3)$ does not admit the cylinder diagram of $M'$ since there must exist an involution that fixes one cylinder and exchanges the other two (See Case (O4) in \cite{AulicinoNguyenWright}).

 If $C_2$ is free, then we can twist it by an appropriate amount and collapse it to get a surface in $\cH(4)^{\rm hyp}$. But the arguments above show that this is also impossible. Thus we get a contradiction, which means that the cylinder diagram of $M$ cannot be in Case 4.III.UA).\\

 \item[b)] {\bf Case 4.III.UB):} Again we have $c_3=c_1+c_2 \in H_1(M,\bZ)$, and $C_1$ cannot be $\cM$-parallel to $C_3$. It follows that there are least three equivalence classes. We claim that $C_2$ must be free. Indeed, since $C_1$ and $C_3$ do not belong the same equivalence class, $C_2$ is neither $\cM$-parallel to $C_1$ nor $C_3$. Remark that $C_4$ contains a transverse simple cylinder thus $C_2$ is not $\cM$-parallel to $C_4$ by Lemma~\ref{lm:loop}. We can then conclude that $C_2$ is free. Collapse $C_2$ so that the two singularities of $M$ collide, we then get a surface in $\cH^{\rm hyp}(4)$. But this is a contradiction with Proposition~\ref{prop:RankkImpRankkBd}, since we have no rank two invariant submanifolds in $\cH^{\rm hyp}(4)$.

 \item[c)] {\bf Case 4.III.UC):} We first claim that $C_2$ and $C_3$ are not $\cM$-parallel. Remark  that there always exists a transverse cylinder $D$ which is contained in $\overline{C}_2\cup\overline{C}_3$ crossing the saddle connection number $3$. We can assume that  $D$ is vertical. Let $\cC$ be the equivalence class of $C_2$. If $\cC$ contains $C_3$, then it also contains $C_1$. Since we have at least two equivalence classes of horizontal cylinders, $C_4$ must be free. There must exist a vertical cylinder $D'$ in the  equivalence class of $D$ that crosses $C_1$. But such a cylinder also crosses $C_4$, which is a contradiction since $P(D,\{C_4\})=0$.

 We now claim that $C_2$ and $C_4$ are not $\cM$-parallel. Assume by contradiction that this is the case, then $C_1$ and $C_3$ must be free. Let $D$ be the vertical cylinder through $C_2$ and $C_3$ constructed above. There must exist a vertical cylinder $D'$ in the equivalence class of $D$ that crosses $C_4$. Since that $D$ does not cross $C_1$, neither does $D'$. Let $n_i, \ i=2,3,4,$ be the number of times $D'$ crosses $C_i$. It is easy to see that we have $n := n_2 = n_3 = n_4$. It follows that
 $$
 P(D',\{C_3\})=\frac{h_3}{h_2+h_3+h_4} \text{, while } P(D,\{C_3\})=\frac{h_3}{h_2+h_3}.
 $$
 Thus we cannot have $P(D,\{C_3\})=P(D',\{C_3\})$ unless $h_4=0$, which is impossible.  We can now conclude that $C_2$ is free. Collapsing $C_2$ so that a single saddle connection is reduced to a point gives us a surface $M'$ in $\cH^{\rm odd}(4)$ which admits no involution acting by $-{\rm Id}$ on the flat metric structure. But by Proposition~\ref{prop:RankkImpRankkBd}, this  surface must belong to $\tilde{\cQ}(3,-1^3)$. We have again a contradiction, which means that this case cannot occur.

 \item[d)] {\bf Case 4.III.UD):} If there are three equivalence classes of horizontal cylinders or more, either $C_1$ or $C_2$ must be free. In both cases, by collapsing a free simple cylinder gives us a surface in $\cH^{\rm odd}(4)$ but not in $\tilde{\cQ}(3,-1^3)$. Thus the horizontal cylinders must fall into two equivalence classes, which are $\cC:=\{C_1,C_2,C_3\}$ and $\{C_4\}$.

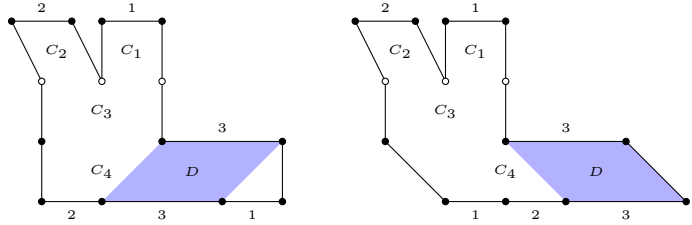
\begin{figure}[htb]
\centering
\begin{minipage}{0.4\linewidth}
\centering
\begin{tikzpicture}[scale=0.40]
\fill[blue!30] (4,2) -- (2,0) -- (6,0) -- (8,2) -- cycle;
\draw (-1,6) -- (0,4) -- (0,0) -- (8,0) -- (8,2) -- (4,2) -- (4,6) -- (2,6) -- (2,4) -- (1,6)  -- cycle;
 \foreach \x in {(-1,6), (0,2), (0,0),(1,6), (2,6), (2,0), (4,6),(4,2),(6,0), (8,2), (8,0)} \filldraw[fill=black] \x circle (3pt);
  \foreach \x in {(0,4),(2,4),(4,4)} \filldraw[fill=white] \x circle (3pt);

  \draw (0,6) node[above] {\tiny $2$} (1,0) node[below] {\tiny $2$} (3,6) node[above] {\tiny $1$} (4,0) node[below] {\tiny $3$} (6,2) node[above] {\tiny $3$} (7,0) node[below] {\tiny $1$};

  \draw (0.5,5) node {\tiny $C_2$} (3,5) node {\tiny $C_1$} (2,3) node {\tiny $C_3$} (2,1) node {\tiny $C_4$} (5,1);
  \draw (5,1) node {\tiny $D$};

\end{tikzpicture}
\end{minipage}
\begin{minipage}{0.4\linewidth}
\centering
\begin{tikzpicture}[scale=0.40]
\fill[blue!30] (4,2) -- (6,0) -- (10,0) -- (8,2) -- cycle;
\draw (-1,6) -- (0,4) -- (0,2)-- (2,0) -- (10,0) -- (8,2) -- (4,2) -- (4,6) -- (2,6) -- (2,4) -- (1,6)  -- cycle;
 \foreach \x in {(-1,6), (0,2), (2,0),(1,6), (2,6), (4,0), (4,6),(4,2),(6,0), (8,2), (10,0)} \filldraw[fill=black] \x circle (3pt);
  \foreach \x in {(0,4),(2,4),(4,4)} \filldraw[fill=white] \x circle (3pt);

  \draw (0,6) node[above] {\tiny $2$} (3,0) node[below] {\tiny $1$} (3,6) node[above] {\tiny $1$} (5,0) node[below] {\tiny $2$} (6,2) node[above] {\tiny $3$} (8,0) node[below] {\tiny $3$};

  \draw (0.5,5) node {\tiny $C_2$} (3,5) node {\tiny $C_1$} (2,3) node {\tiny $C_3$} (4,1) node {\tiny $C_4$} (5,1);

  \draw (7,1) node {\tiny $D$};

\end{tikzpicture}
\end{minipage}
\caption{The two deformations in the Proof of Case 4.III.UD)}
\label{Case4IIIUDPfFig}
\end{figure}

Let $C_i$ have height $h_i$, for all $i$.  Let $\ell(k)$ denote the length of the saddle connection labeled by $k$. Let $D$ be the simple cylinder in $C_4$ whose core curves only cross the saddle connection $3$ (see Figure~\ref{Case4IIIUDPfFig}).  Observe that $D$ is free, since any other cylinder parallel to $D$ must cross $C_3$.   We can stretch the cylinder $D$ so that $\ell(3) \geq \ell(1) + \ell(2)$.  Depending on whether $h_1 \geq h_2$ or $h_1 \leq h_2$, we perform one of the following deformations depicted in Figure~\ref{Case4IIIUDPfFig}.  If $h_2 \geq h_1$, then twist $\cC$ and $C_4$ so that $C_2$ lies directly over $2$ in the bottom of $C_4$, and $C_1$ lies directly over $3$ in the bottom of $C_4$.  If $h_2 \leq h_1$, then twist $\cC$ and $C_4$ so that $C_1$ lies directly over $1$ in the bottom of $C_4$, and $C_2$ lies directly over $3$ in the bottom of $C_4$.  We consider only the case where 
 $h_2 \geq h_1$.

By Proposition \ref{prop:SimpCylPres}, there is a vertical cylinder $D_1$ passing through $C_2$, $C_3$ and $C_4$ whose closure contains $C_2$.  Since $C_1$ is parallel to $C_2$ and $D_1$ does not intersect $C_1$, there is a vertical cylinder $D_2$ that is $\cM$-parallel to $D_1$ and intersects $C_1$.  Let $D_1$ pass through each of $C_2$, $C_3$, $C_4$, $n$ times.  Let $D_2$ pass through $C_i$, $n_i$ times for $i \in \{1, 3, 4\}$.  Clearly $D_2$ does not pass through $C_2$.  Some observations are in order.  Note that all trajectories in $D_2$ ascending from $C_3$ enter $C_1$, and all trajectories ascending from $C_1$ enter $C_4$ followed eventually by $C_3$.  Hence, $n_1 = n_3$.  The assumption that $C_1$ lies directly over saddle connection $3$ combined with the assumption that $\ell(3) \geq \ell(1) + \ell(2)$ implies $n_4 > n_1$.  Then $P(D_1, \{C_4\}) = P(D_2, \{C_4\})$ implies
$$\frac{nh_4}{n(h_2 + h_3 + h_4)} = \frac{n_4h_4}{n_1h_1 + n_3h_3 + n_4h_4},$$
which simplifies to
$$(n_4h_2 - n_1h_1) + (n_4 - n_1)h_3 = 0.$$
Since we have $h_2\geq h_1$ by assumption, this equation holds only if $n_4 = n_1$, which is a contradiction.  Thus, $\cM$ contains no horizontally periodic surface with four cylinders.
\end{itemize}
The proof of the  theorem is now complete.
\end{proof}

\begin{remark}
In all of the cases above, the rank two property was obvious after collapsing cylinders even without the use of Proposition \ref{prop:RankkImpRankkBd}.
\end{remark}

\begin{remark}
Also, the degeneration argument itself was unnecessary in these cases because in each case there was a Lagrangian subspace of equivalence classes prelimit.
\end{remark}

\subsection{$\cH^{\rm hyp}(2,2)$}

As a direct consequence of Proposition~\ref{prop:DG:complete}, we have the following

\begin{lemma}
\label{Hhyp224CylDiags}
There are two $4$-cylinder diagrams in $\mathcal{H}^{\rm hyp}(2,2)$, and they are pictured in Figure \ref{H4hyp4CylDiags}.
\end{lemma}

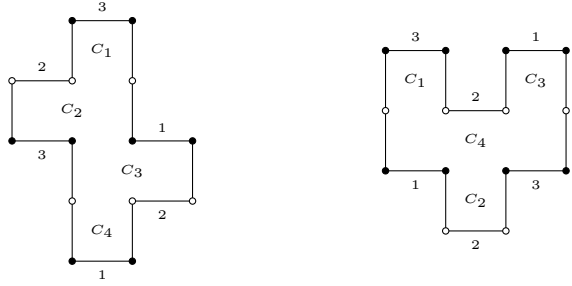
\begin{figure}[htb]
\centering
\begin{minipage}{0.4\linewidth}
\centering
\begin{tikzpicture}[scale=0.40]
\draw (0,0)--(0,4)--(-2,4)--(-2,6)--(0,6)--(0,8)--(2,8)--(2,4)--(4,4)--(4,2)--(2,2)--(2,0)--cycle;
\foreach \x in {(0,0),(0,4),(-2,4),(0,8),(2,8),(2,4),(4,4),(2,0)} \filldraw[fill=black] \x circle (3pt);
\foreach \x in {(-2,6),(0,6),(4,2),(2,2),(0,2),(2,6)} \filldraw[fill=white] \x circle (3pt);

\draw(3,4) node[above] {\tiny $1$};
\draw(-1,6) node[above] {\tiny $2$};
\draw(1,8) node[above] {\tiny $3$};
\draw(1,0) node[below] {\tiny $1$};
\draw(3,2) node[below] {\tiny $2$};
\draw(-1,4) node[below] {\tiny $3$};
\draw(2,3) node {\tiny $C_3$};
\draw(1,7) node {\tiny $C_1$};
\draw(0,5) node {\tiny $C_2$};
\draw(1,1) node {\tiny $C_4$};
\end{tikzpicture}
\end{minipage}
\begin{minipage}{0.4\linewidth}
\centering
\begin{tikzpicture}[scale=0.40]
\draw (0,0)--(0,4)--(2,4)--(2,2)--(4,2)--(4,4)--(6,4)--(6,0)--(4,0)--(4,-2)--(2,-2)--(2,0)--cycle;
\foreach \x in {(0,0),(0,4),(2,4),(4,4),(6,4),(6,0),(4,0),(2,0)} \filldraw[fill=black] \x circle (3pt);
\foreach \x in {(0,2),(2,-2),(2,2),(4,-2),(4,2),(6,2)} \filldraw[fill=white] \x circle (3pt);

\draw(1,4) node[above] {\tiny $3$};
\draw(3,2) node[above] {\tiny $2$};
\draw(5,4) node[above] {\tiny $1$};
\draw(1,0) node[below] {\tiny $1$};
\draw(3,-2) node[below] {\tiny $2$};
\draw(5,0) node[below] {\tiny $3$};

\draw(1,3) node {\tiny $C_1$};
\draw(3,-1) node {\tiny $C_2$};
\draw(5,3) node {\tiny $C_3$};
\draw(3,1) node {\tiny $C_4$};
\end{tikzpicture}
\end{minipage}
\caption{The two $4$-cylinder diagrams in $\Hyp$: Case 4.I.HA) (left) and Case 4.I.HB) (right)}
\label{H4hyp4CylDiags}
\end{figure}

Let us now show

\begin{lemma}
\label{Hhyp22UniqueDiagRed}
If $\cM$ is a rank two manifold in $\mathcal{H}^{\rm hyp}(2,2)$, then $\cM$ contains the horizontally periodic translation surface $M$ satisfying Case 4.I.HA).
\end{lemma}
\begin{proof}
By Lemma \ref{Hhyp224CylDiags}, both $4$-cylinder diagrams in $\mathcal{H}^{\rm hyp}(2,2)$ satisfy Case 4.I). Note that the horizontal cylinders of $M$ must fall into at least two equivalence classes by Theorem~\ref{thm:RankkImpkCyl}. We will show that $M$ cannot satisfy Case 4.I.HB).  Let $C_1,C_2,C_3$ denote the simple cylinders, and $C_4$ the largest one. Since we have the relation $c_1+c_2+c_3=c_4$, if all of the simple cylinders belong to the same equivalence class, then this equivalence class also contains $C_4$, and we have a contradiction to the assumption that there are at least two equivalence classes. If one of the simple cylinders is $\cM$-parallel to $C_4$, then Lemma~\ref{lm:n:f:sim:cyl} implies that there is only one equivalence class. Thus, we derive that there is at least a free simple cylinder.

Collapsing a simple cylinder so that the zeros collide yields a translation surface in $\cH^{\rm hyp}(4)$. By Proposition~\ref{prop:RankkImpRankkBd}, this surface must belong to a rank two invariant submanifold of $\cH^{\rm hyp}(4)$.  But there is no rank two affine manifold in $\cH^{\rm hyp}(4)$ by Theorem \ref{NWANWThm}. Therefore, Case 4.I.HB) does not occur for surfaces in a rank two invariant submanifold of $\cH^{\rm hyp}(2,2)$.
\end{proof}

\begin{lemma}
\label{H22hypEqClasses}
If $M$ satisfies Case 4.I.HA) and its orbit closure $\cM$ has rank two, then there are two equivalence classes of cylinders: $\mathcal{C} = \{C_1, C_4\}$ and $\mathcal{C}' = \{C_2, C_3\}$.
\end{lemma}

\begin{proof}
Again by Theorem~\ref{thm:RankkImpkCyl}, we know that there are at least two equivalence classes of horizontal cylinders.  We first notice that the simple cylinders cannot be free, otherwise by collapsing such a cylinder so that the zeros collide, we get a surface in $\cH^{\rm hyp}(4)$ whose orbit closure is a rank two invariant submanifold by Proposition~\ref{prop:RankkImpRankkBd}. In particular, we derive that $C_1$ and $C_4$ are not free. If $C_1$ is $\cM$-parallel to $C_2$, then we only have one equivalence class by Lemma~\ref{lm:n:f:sim:cyl}, thus $C_1$ and $C_2$ are not $\cM$-parallel. By the same argument, $C_3$ and $C_4$ are not $\cM$-parallel.

Assume that $C_1$ and $C_3$ are $\cM$-parallel. Since $C_4$ is not free and not $\cM$-parallel to $C_1$, it must be $\cM$-parallel to $C_2$. But since we have the following relation $c_1+c_3=c_2+c_4 \in H_1(M,\bZ)$, this would imply that $C_3$ and $C_4$ are $\cM$-parallel, thus there is only one equivalence class. Hence, we derive that $C_1$ and $C_4$ must be $\cM$-parallel.

It remains to show that $C_2$ and $C_3$ are $\cM$-parallel. Twist $C_1$ and $C_2$ independently so that there is a vertical cylinder $D$ contained in $\overline{C}_1\cup\overline{C}_2$ crossing the saddle connection $3$. There must exist another vertical cylinder $D'$ in the equivalence class of $D$ that crosses $C_4$. Clearly $D'$ must cross $C_3$. If $C_3$ is free then $P(D,\{C_3\})=0$, while $P(D',\{C_3\})>0$. Thus, we must have $C_2$ and $C_3$ are $\cM$-parallel. The lemma is then proved.
\end{proof}

\begin{lemma}
\label{H22hypDoubleCov}
If $M$ satisfies Case 4.I.HA) and its orbit closure $\cM$ has rank two, then $M$ admits a double covering to a half-translation surface in $\mathcal{Q}(1^2,-1^2)$.
\end{lemma}

\begin{proof}
By Lemma \ref{H22hypEqClasses}, the two equivalence classes must be $\mathcal{C} = \{C_1, C_4\}$ and $\mathcal{C}' = \{C_2, C_3\}$. We prove that this surface must admit a double covering to a half-translation surface in the stratum $\mathcal{Q}(1^2,-1^2)$.  Twist every cylinder in the equivalence classes of the cylinders $C_1$ and $C_2$ (which are in different equivalence classes) to get a vertical cylinder $V_1$ through saddle connection $3$.  This cylinder will contain $C_1$, but it does not pass through $C_4$ and $C_3$.  Therefore, there must be a cylinder $V_2$ that is $\cM$-parallel to $V_1$ that crosses $C_4$, hence $C_3$.

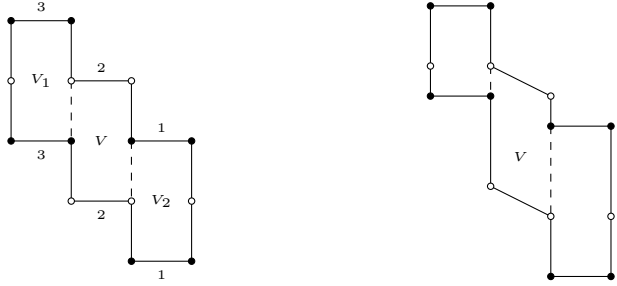
\begin{figure}[htb]
\centering
\begin{minipage}{0.45\linewidth}
\centering
\begin{tikzpicture}[scale=0.40]
\draw  (0,8) -- (0,4) -- (2,4) -- (2,2) -- (4,2) -- (4,0) -- (6,0) -- (6,4) -- (4,4) -- (4,6) -- (2,6) -- (2,8) -- cycle;
\draw[dashed] (2,6) -- (2,4) (4,4) -- (4,2);
\foreach \x in {(0,8),(0,4),(2,8),(2,4),(4,4),(4,0),(6,4),(6,0)} \filldraw[fill=black] \x circle (3pt);
\foreach \x in {(0,6),(2,6),(2,2),(4,6),(4,2),(6,2)} \filldraw[fill=white] \x circle (3pt);
\draw (1,8) node[above] {\tiny $3$} (1,4) node[below] {\tiny $3$} (3,6) node[above] {\tiny $2$} (3,2) node[below] {\tiny $2$} (5,4) node[above] {\tiny $1$} (5,0) node[below] {\tiny $1$};
\draw (1,6) node {\tiny $V_1$} (5,2) node {\tiny $V_2$} (3,4) node {\tiny $V$};
\end{tikzpicture}
\end{minipage}
\begin{minipage}{0.45\linewidth}
\centering
\begin{tikzpicture}[scale=0.40]
\draw[dashed] (2,1) -- (2,0) (4,-1) -- (4,-4);
\draw (0,0)--(0,3)--(2,3)--(2,1)--(4,0)--(4,-1)--(6,-1)--(6,-6)--(4,-6)--(4,-4)--(2,-3)--(2,0)--cycle;
\foreach \x in {(0,3),(0,0),(2,3),(2,0),(4,-1),(4,-6),(6,-1),(6,-6)} \filldraw[fill=black] \x circle (3pt);
\foreach \x in {(0,1),(2,1),(2,-3),(4,0),(4,-4),(6,-4)} \filldraw[fill=white] \x circle (3pt);
\draw (3,-2) node {\tiny $V$};
\end{tikzpicture}
\end{minipage}
\caption{Proof of Lemma \ref{H22hypDoubleCov}}
\label{H22hypLemPfFig}
\end{figure}

Note that $V_2$ does not cross $C_1$.  For all $i$, let $V_2$ pass through $C_i$, $n_i$ times.  We have $n_3=n_2+n_4$.  Let the height of $C_i$ be $h_i$, for all $i$, and compute
$$P(V_1, \mathcal{C}) = P(V_2, \mathcal{C}) \Rightarrow \frac{h_1}{h_1+h_2} = \frac{n_4h_4}{n_2h_2 + n_3h_3+n_4h_4},$$
which implies
$$\frac{h_4}{h_1} = \frac{n_2h_2+n_3h_3}{n_4h_2} \geq \frac{n_3}{n_4}\frac{h_3}{h_2}.$$

Similarly, if we twist $C_3$ and $C_4$ to get a vertical cylinder $V_1'$ through the saddle connection $1$ containing $C_4$, then there is a cylinder $V_2'$ that is $\cM$-parallel to $V_1'$  through $C_1$.  Letting $V_2'$ pass through $C_i$, $n_i'$ times. Remark that $n'_4=0$ and $n'_2=n'_1+n'_3$. We then have
$$P(V_1', \mathcal{C}) = P(V_2', \mathcal{C}) \Rightarrow \frac{h_4}{h_3 + h_4} = \frac{n_1' h_1}{n_1'h_1' + n_2'h_2 + n_3'h_3},$$
which simplifies to
$$\frac{h_1}{h_4} = \frac{n_2'h_2+n'_3h_3}{n_1'h_3} \geq \frac{n'_2}{n'_1}\frac{h_2}{h_3}.$$
Combining these two inequalities yields
$$ 1 =\frac{h_4}{h_1}\frac{h_1}{h_4} \geq \frac{n_3}{n_4}\frac{n'_2}{n'_1} \geq 1$$
because $n_3 \geq n_4$ and $n'_2 \geq n'_1$.  The equality occurs if any only if $n_2 = n'_3 = 0$, which implies that $h_1/h_4 = h_2/h_3$, and $V_2$ must be entirely contained in $C_3$ and $C_4$.  In particular, it cannot pass through the saddle connection $2$, and this is only possible if $2$ lies directly over itself.  This means exactly that if the sides of $C_2$ are twisted so that they are vertical, then $C_3$ must also have vertical sides.

Moreover, there must be a vertical cylinder $V$ passing vertically from $2$ to itself that is contained entirely in $C_2$ and $C_3$.  Remark that $V$ is free. From the relation $P(C_2,\{V\})=P(C_3,\{V\}))$ we derive that $\wth(C_2)=\wth(C_3)$ and $\wth(C_1)=\wth(C_4)$, where $\wth(C_i)$ is the circumference of $C_i$.

Without loss of generality, let $h_3 \geq h_2$.  Twist $V$ as in Figure \ref{H22hypLemPfFig}, so that the zeros $v_1$ and $v_2$ lie on the same horizontal saddle connection.  In fact, if $h_3 > h_2$, then there will be exactly one horizontal saddle connection joining $v_1$ and $v_2$ in $V$.  Therefore, if we collapse $V$, we would degenerate to a rank two orbit closure by Proposition \ref{prop:RankkImpRankkBd} in $\mathcal{H}^{\rm hyp}(4)$, which does not exist.  Such a contradiction implies that we must indeed have $h_3 = h_2$, which in turn implies $h_1 = h_4$.

We now show that $C_1$ and $C_4$ have the same twist. Let us twist $C_1$ and $C_2$ so that the vertical cylinder $V_1$ exists. Assume that the twist of $C_4$ is non-zero. In this case collapsing $C_1$ and $C_4$ simultaneously only destroys a single saddle connection (which is contained in $C_1$).  Thus, the resulting surface belongs to $\cH^{\rm hyp}(4)$. Since such a surface must be contained in a rank two invariant submanifold by  Proposition \ref{prop:RankkImpRankkBd}, we then get a contradiction. Therefore the twist of $C_4$ must be zero, which means that $C_1$ and $C_4$ are isometric.

Observe now that we have an involution of $M$ that sends $C_1$ and $C_2$ to $C_4$ and $C_3$, respectively. This involution fixes the two singularities of $M$ and two other points. Thus $M$ is the double covering of a quadratic differential in $\tilde{\cQ}(1^2,-1^2)$.
\end{proof}

\begin{theorem}
\label{thm:H22hypClass}
The Prym locus $\tilde{\mathcal{Q}}(1^2,-1^2)$ is the unique rank two invariant submanifold in $\mathcal{H}^{\rm hyp}(2,2)$.
\end{theorem}

\begin{proof}
By Corollary~\ref{cor:Hmn:4cyl:diag} and Lemma~\ref{Hhyp22UniqueDiagRed}, every rank two invariant submanifold $\cM$ in $\mathcal{H}^{\rm hyp}(2,2)$ contains a horizontally periodic translation surface $M$ satisfying Case 4.I.HA).  Lemma \ref{H22hypDoubleCov} implies that $M$ is contained in the Prym locus $\tilde{\cQ}(1^2,-1^2)$.  Set $V:=T^{\bR}_M\cM$ and $W:=T^{\bR}_M\tilde{\cQ}(1^2,-1^2)$. In what follows, we identify $M$ with a point in $H^1(M,\Sigma,\bR+\imath\bR)$. Pick any vector $v\in V$, for $\eps >0$ small enough $M+\eps v$ is also a surface in $\cM$ admitting a cylinder decomposition in the horizontal direction with the same diagram. Lemma~\ref{H22hypDoubleCov} applied to $M+\eps v$ then implies that $M+\eps v \in \tilde{\cQ}(1^2,-1^2)$. Thus we can conclude that $V \subset W$ and $\cM \subset \tilde{\cQ}(1^2,-1^2)$.  Moreover,
$$
\dim_\bC\cM = \dim_\bR V \leq \dim_\bR W =\dim_\bC\tilde{\cQ}(1^2,-1^2)=4.
$$
But since $\cM$ has rank two, by definition its dimension is at least four. Hence we have $\dim_\bC\cM=\dim_\bC\tilde{\cQ}(1^2,-1^2)$, and consequently $\cM=\tilde{\cQ}(1^2,-1^2)$. The proof of the theorem is now complete.
\end{proof}

\begin{lemma}
 \label{lm:H2:in:Hyp22}
 The locus $\tilde{\cQ}(1^2,-1^2)$ is also the set $\covershyp$ of unramified double coverings of $\cH(2)$ in $\cH^{\rm hyp}(2,2)$.
\end{lemma}
\begin{proof}
Recall that by definition $M=(X,\omega)$, where $X$ is a Riemann surface of genus three, and $\omega$ is a holomorphic $1$-form on $X$ having two double zeros. For any $M=(X,\omega) \in \tilde{\cQ}(1^2,-1^2) \subset \cH^{\rm hyp}(2,2)$ we have two involutions on $X$: the hyperelliptic one denoted by $\iota$ and the another one coming from the double covering of a quadratic differential in $\cQ(1^2,-1^2)$ denoted by $\tau$. Set $\vartheta :=\iota\circ\tau$, since $\iota$ commutes with all the automorphisms of $X$, $\vartheta$ is also an involution.

Set $Y:= X/\langle \vartheta \rangle$. By definition, we have $\iota^*\omega=\tau^*\omega=-\omega$, thus $\omega \in \ker(\vartheta -{\rm Id})\subset \Omega(X)$ where $\Omega(X)$ is the space of holomorphic $1$-forms on $X$. Remark that $\dim_\bC\ker(\vartheta -{\rm Id})=\dim_\bC\ker(\tau+{\rm Id})=2$, therefore $Y$ is a surface of genus two and $\omega$ arises from a holomorphic $1$-form $\eta$ on $Y$.

Let $\pi : X \rightarrow Y$ be the double covering induced by $\vartheta$. Since $Y$ has genus two, the Riemann-Hurwitz formula then implies that $\pi$ is unramified. It follows that $\eta$ has a double zero on $Y$, that is $(Y,\eta)\in \cH(2)$. Therefore we have $\tilde{\cQ}(1^2,-1^2) \subset \covershyp$. Since $\covershyp$ is clearly a rank two invariant submanifold of $\cH^{\rm hyp}(2,2)$, from Theorem~\ref{thm:H22hypClass} we can conclude that $\covershyp=\tilde{\cQ}(1^2,-1^2)$.
\end{proof}

\subsection{$\cH^{\rm odd}(2,2)$}

Our goal is to show the following

\begin{theorem}
\label{thm:H22:odd:4cyl}
Let $\cM$ be a rank two affine manifold of $\Odd$.   Then either
$$
\left\{
\begin{array}{ccc}
 \dim_{\bC}\cM=5 & \text{ and } & \cM=\prymodd \\
  \dim_{\bC}\cM=4 & \text{ and } & \cM=\coversodd,
\end{array}
\right.
$$

\noindent where $\coversodd$ is the locus of unramified double covers of surfaces in $\cH(2)$ in $\Odd$.
\end{theorem}

By Corollary~\ref{cor:Hmn:4cyl:diag}, we know that $\cM$ contains a horizontally periodic surface with four cylinders.  The following lemma is a direct consequence of Proposition~\ref{prop:DG:complete}.

\begin{lemma}
\label{lm:4cyl:dm}
The $4$-cylinder diagrams in $\Odd$ satisfy one of the following five cases: 4.I.OA), 4.I.OB), 4.II.OA), 4.II.OB), 4.II.OC) as shown in Figures~\ref{H22OddCase4IFig} and \ref{H22OddCase4IIFig}.
\end{lemma}

\begin{figure}[htb]
\begin{minipage}[t]{0.4\linewidth}
\centering
\begin{tikzpicture}[scale=0.40]
\draw (0,0)--(0,4)--(2,4)--(2,2)--(4,2)--(4,4)--(6,4)--(6,0)--(4,0)--(4,-2)--(2,-2)--(2,0)--cycle;
\foreach \x in {(0,0),(0,4),(2,4),(4,4),(6,4),(6,0),(4,0),(2,0)} \filldraw[fill=black] \x circle (3pt);
\foreach \x in {(2,2),(4,2),(4,-2),(2,-2),(0,2),(6,2)} \filldraw[fill=white] \x circle (3pt);
\draw(1,4) node[above] {\tiny $1$};
\draw(3,2) node[above] {\tiny $2$};
\draw(5,4) node[above] {\tiny $3$};
\draw(1,0) node[below] {\tiny $1$};
\draw(3,-2) node[below] {\tiny $2$};
\draw(5,0) node[below] {\tiny $3$};
\draw(1,3) node {\tiny $C_1$};
\draw(3,-1) node {\tiny $C_2$};
\draw(5,3) node {\tiny $C_3$};
\draw(3,1) node {\tiny $C_4$};

\end{tikzpicture}
\end{minipage}
\begin{minipage}[t]{0.4\linewidth}
\centering
\begin{tikzpicture}[scale=0.40]
\draw (0,0)--(0,2)--(-2,2)--(-2,6)--(0,6)--(0,8)--(2,8)--(2,6)--(2,4)--(4,4)--(4,0)--cycle;
\foreach \x in {(0,0),(0,8),(2,8),(2,4),(4,4),(4,0),(-2,4),(2,0)} \filldraw[fill=black] \x circle (3pt);
\foreach \x in {(-2,6),(-2,2), (0,6),(0,2), (2,6), (4,2)} \filldraw[fill=white] \x circle (3pt);
\draw(-1,6) node[above] {\tiny $1$};
\draw(1,8) node[above] {\tiny $2$};
\draw(3,4) node[above] {\tiny $3$};
\draw(-1,2) node[below] {\tiny $1$};
\draw(1,0) node[below] {\tiny $2$};
\draw(3,0) node[below] {\tiny $3$};

\draw (-3,7) node {\tiny $C_1$} (-3,5) node {\tiny $C_2$} (-3,3) node {\tiny $C_3$} (-3,1) node {\tiny $C_4$};
\end{tikzpicture}
\end{minipage}
\caption{$\Odd$ Cases 4.I.OA) (left), and 4.I.OB) (right)}
\label{H22OddCase4IFig}
\end{figure}
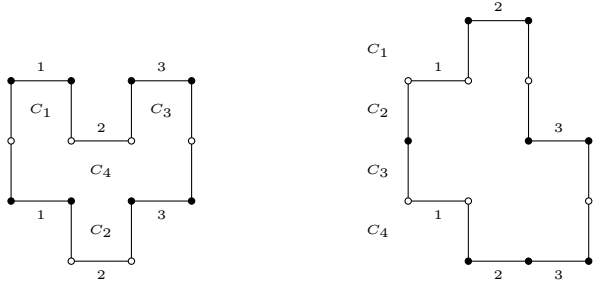

\begin{figure}[htb]
\centering
\begin{minipage}[t]{0.3\linewidth}
\centering
\begin{tikzpicture}[scale=0.40]
\draw (0,0)--(0,8)--(4,8)--(4,6)--(2,6)--(2,4)--(4,4)--(4,2)--(2,2)--(2,0)--cycle;

\foreach \x in {(0,0),(0,8),(2,8), (4,8),(4,2),(2,2),(2,0),(0,2)} \filldraw[fill=black] \x circle (3pt);
\foreach \x in {(4,6),(2,6),(2,4),(4,4),(0,4),(0,6)} \filldraw[fill=white] \x circle (3pt);

\draw(1,8) node[above] {\tiny 1};
\draw(3,4) node[above] {\tiny 3};
\draw(3,8) node[above] {\tiny 2};
\draw(1,0) node[below] {\tiny 1};
\draw(3,2) node[below] {\tiny 2};
\draw(3,6) node[below] {\tiny 3};

\draw(0,1) node[left] {\tiny $C_4$};
\draw(0,3) node[left] {\tiny $C_3$};
\draw(0,5) node[left] {\tiny $C_2$};
\draw(0,7) node[left] {\tiny $C_1$};

\end{tikzpicture}
\end{minipage}
\begin{minipage}[t]{0.3\linewidth}
\centering
\begin{tikzpicture}[scale=0.40]
\draw (0,0)--(0,8)--(4,8)--(4,6)--(2,6)--(2,4)--(4,4)--(4,2)--(2,2)--(2,0)--cycle;
\foreach \x in {(0,0),(0,8),(2,8), (4,8),(4,6),(2,6),(2,0), (0,6)} \filldraw[fill=black] \x circle (3pt);
\foreach \x in {(0,4),(0,2),(2,4),(2,2),(4,4),(4,2)} \filldraw[fill=white] \x circle (3pt);
\draw(1,8) node[above] {\tiny 1};
\draw(3,4) node[above] {\tiny 3};
\draw(3,8) node[above] {\tiny 2};
\draw(1,0) node[below] {\tiny 1};
\draw(3,2) node[below] {\tiny 3};
\draw(3,6) node[below] {\tiny 2};

\draw(0,1) node[left] {\tiny $C_4$};
\draw(0,3) node[left] {\tiny $C_3$};
\draw(0,5) node[left] {\tiny $C_2$};
\draw(0,7) node[left] {\tiny $C_1$};
\end{tikzpicture}
\end{minipage}
\begin{minipage}[t]{0.3\linewidth}
\centering
\begin{tikzpicture}[scale=0.40]
\draw (0,0)--(0,8)--(2,8)--(2,6)--(4,6)--(4,4)--(6,4)--(6,2)--(4,2)--(4,0)--cycle;
\foreach \x in {(0,0),(0,8),(2,8),(2,6),(4,6),(2,0),(0,6),(4,0)} \filldraw[fill=black] \x circle (3pt);

\foreach \x in {(4,4),(6,4),(6,2),(4,2),(0,2),(0,4)} \filldraw[fill=white] \x circle (3pt);

\draw(1,8) node[above] {\tiny 1};
\draw(3,6) node[above] {\tiny 2};
\draw(5,4) node[above] {\tiny 3};
\draw(1,0) node[below] {\tiny 1};
\draw(3,0) node[below] {\tiny 2};
\draw(5,2) node[below] {\tiny 3};

\draw(0,1) node[left] {\tiny $C_4$};
\draw(0,3) node[left] {\tiny $C_3$};
\draw(0,5) node[left] {\tiny $C_2$};
\draw(0,7) node[left] {\tiny $C_1$};

\end{tikzpicture}
\end{minipage}
\caption{$\Odd$ Cases 4.II.OA) (left), 4.II.OB) (center), 4.II.OC) (right)}
\label{H22OddCase4IIFig}
\end{figure}
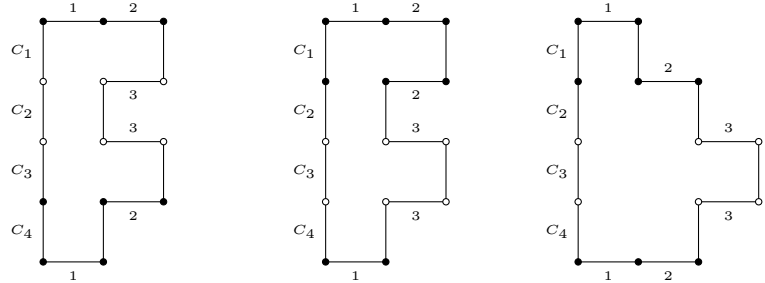

We will consider each of the diagrams listed in Lemma~\ref{lm:4cyl:dm}. Theorem~\ref{thm:H22:odd:4cyl} will follow from the Lemmas~\ref{lm:Odd:4IOA}, \ref{lm:Odd:4IOB}, \ref{lm:Odd:4IIOC}, \ref{lm:Odd:4IIOA}, \ref{lm:Odd:4IIOB}. Throughout, it is important to note that if $\cM \subset \cH(2,2)$ has rank two, then $\dim_{\mathbb{C}}(\cM) \leq 5$.

\bigskip

%
%
%

\subsubsection{Case 4.I.OA)}

\begin{lemma}
\label{lm:Odd:4IOA}
Let $M$ be a horizontally periodic translation surface in a rank two affine manifold $\cM$.  If $M$ satisfies Case 4.I.OA), then $\cM = \prymodd$.
\end{lemma}

\begin{proof}
Observe that
$$c_1+c_2+c_3=c_4.$$
By Theorem~\ref{thm:RankkImpkCyl}, the horizontal cylinders of $M$ fall into at least two equivalence classes. We claim that one of the simple cylinders is free.  If $C_1,C_2,C_3$ are $\cM$-parallel, then the relation above implies that $C_4$ also belongs to this equivalence class, and we have a contradiction to the assumption that there are at least two equivalence classes. If one of $C_1,C_2,C_3$ is $\cM$-parallel to $C_4$, then we also have a unique equivalence class by Lemma~\ref{lm:n:f:sim:cyl}. Therefore, we can conclude that $C_4$ and one of the simple cylinders are free.


Let $C_1$ be the free simple cylinder.  Collapse $C_1$ while keeping the other cylinders unchanged.  This yields a surface $M' \in \prym$ by Proposition~\ref{prop:RankkImpRankkBd} and Theorem \ref{NWANWThm}.

In particular, $M' \in \prym$, thus there exists an involution $\tau'$ on $M'$ that fixes $C_4$ and exchanges $C_2$ and $C_3$. Remark that $C_1$ degenerates to a horizontal saddle connection invariant by $\tau'$.  We claim that $\tau'$ can be extended to an involution $\tau$ on $X$ that fixes $C_1,C_4$ and exchanges $C_2$ and $C_3$.  To see this it suffices to show that there is an involution fixing $C_1$.  However, every simple cylinder can be realized as a parallelogram, which clearly admits an involution of order two given by rotation by $\pi$.  Thus, $\tau'$ acts on $M$ exactly as claimed, and we conclude that $M \in \prymodd$.

Since the boundary of $\cM$ contains $\prym$, by algebraicity, $\dim_{\bC}(\cM) > \dim_{\bC} \prym = 4$.  Hence, $\dim_{\bC}(\cM) = 5$.  For any $v\in T^\bR_M \cM$ and $\eps\in \bR$ small enough, the deformation $M_\eps:=M+\eps v$ of $M$ also has a cylinder decomposition in the horizontal direction with the same diagram. The arguments above then imply that $M_\eps \in \prymodd$, thus we have $T^\bR_M \cM \subset T^\bR_M \prymodd$. Since $\dim_\bC\cM=\dim_\bC\prymodd$, we have $T^\bR_M \cM = T^\bR_M \prymodd$ and $\cM=\prymodd$.
\end{proof}

\bigskip

\subsubsection{Case 4.I.OB)}

\begin{lemma}
\label{lm:Odd:4IOB}
Let $M$ be a horizontally periodic translation surface in a rank two affine manifold $\cM$.  If $M$ satisfies Case 4.I.OB), then $\cM = \prymodd$.
\end{lemma}

\begin{proof}
Again, by Theorem~\ref{thm:RankkImpkCyl}, we know that the horizontal cylinders of $M$ must fall into at least two equivalence classes. If all of the cylinders are free, then $p(T^\bR_M\cM)$ contains a Lagrangian subspace of dimension three.  Therefore, we have at most three equivalence classes.  We claim that there are exactly three equivalence classes of cylinders.  Observe that we have
$$c_1+c_3=c_2+c_4.$$

Therefore, if three of the cylinders belong to an equivalence class then all four cylinders are $\cM$-parallel. Note that we cannot have $\cC_1=\{C_1,C_3\}$ and $\cC_2=\{C_2,C_4\}$ because it would imply that there is one equivalence class by the homological relation among the core curves of cylinders.

Suppose that we have $\cC_1=\{C_1,C_2\}$ and $\cC_2=\{C_3,C_4\}$. Let us denote by $\wth(C_i)$ the circumference of $C_i$ and $\alpha_i$ the element of $(T^\bR_M\cM)^\ast$ defined by $C_i$. Recall that $C_i$ and $C_j$ are $\cM$-parallel precisely means $\alpha_i$ and $\alpha_j$ are proportional (collinear). Thus, there exist $\lambda, \mu \in \bR_{>0}$ such that $\alpha_1=\lambda \alpha_2$ and $\alpha_3=\mu \alpha_4$. Note that
$$
 \lambda=\cfrac{\wth(C_1)}{\wth(C_2)}, \, \, \mu=\cfrac{\wth(C_3)}{\wth(C_4)}
$$
Thus we have
 $$
  (1-\lambda)\alpha_2+(1-\mu)\alpha_4=0 \in (T^\bR_{M}\cM)^\ast.
 $$
But by assumption, $\alpha_2$ and $\alpha_4$ are not collinear, since $C_2$ and $C_4$ are not $\cM$-parallel.  Hence, we must have $\lambda = \mu = 1$.  However, it is clear that $\wth(C_2) > \wth(C_1)$, and $\wth(C_3)> \wth(C_4)$, thus we get a contradiction.  By the same argument, $\{C_1, C_4\}$ and $\{C_2, C_3\}$ are not equivalence classes. We can now conclude that the cylinders $C_1,\dots,C_4$ belong to three equivalence classes.

If $C_2$ is free, then we can collapse it to get a surface $N$ in $\cH(4)$ which must belong to $\prym$ (by Proposition~\ref{prop:RankkImpRankkBd} and Theorem~\ref{NWANWThm}). But a surface in $\prym$ cannot admit the same cylinder decomposition as $N$, therefore we have a contradiction, which means that $C_2$ is not free. The same arguments apply to $C_4$.

Since $M$ has three equivalence classes, and we know that neither of $C_2$ and  $C_4$ is free, it follows that  $C_2$ and $C_4$ must belong to the same equivalence class and both $C_1$ and $C_3$ are free.  We can then collapse $C_1$ to get a surface $N$ in $\cH(4)$. Proposition~\ref{prop:RankkImpRankkBd} and Theorem~\ref{NWANWThm} imply that this surface is contained in $\prym$.  The Prym involution of $N$ extends to an involution of $M$ that fixes $C_1,C_3$, and exchanges $C_2,C_4$. In particular, we have $C_2$ and $C_4$ are isometric, and $M \in \prymodd$.

Since the boundary of $\cM$ contains $\prym$, by algebraicity, $\dim_{\bC}(\cM) > \dim_{\bC} \prym = 4$.  Hence, $\dim_{\bC}(\cM) = 5$.  For any vector $v \in T^\bR_M\cM$, and $\eps \in \bR$ small enough, the deformation $M_\eps:=M+\eps v$ also admits a cylinder decomposition in the horizontal direction with the same diagram. The arguments above show that $M_\eps \in \prymodd$. Thus we have $T^\bR_M\cM \subset T^\bR_M\prymodd$.   But $\dim_\bR T^\bR_M\prymodd = 5$,  thus we have $ T^{\bR}_M\cM =T^\bR_M\prymodd=5$, which implies that $\cM=\prymodd$.
\end{proof}
%
%

\bigskip

\subsubsection{Case 4.II.OC)}

\begin{lemma}
\label{lm:Odd:4IIOC}
If $M$ is a horizontally periodic surface in $\cM$, then $M$ cannot admit a cylinder decomposition in Case 4.II.OC).
\end{lemma}
\begin{proof}
We assume that $M$ admits a decomposition in Case 4.II.OC) and will derive a contradiction. We first observe that $c_2$ and $c_4$ are homologous, hence $C_2$ and $C_4$ belong to the same equivalence class $\cC$.

We will now show that $C_3$ is free. Suppose that $C_3\in \cC$, then $C_1$ is free since there are at least two equivalence classes. There exists a simple cylinder $D$ in $C_3$ which can be supposed to be vertical. Let $\cD=\{D,D_1,\dots,D_k\}$ denote the equivalence class of $D$. Since $C_2$ is $\cM$-parallel to $C_3$, there must exist some vertical cylinders in $\cD$ that cross $C_2$. Since $D$ does not cross $C_1$, neither do $D_1,\dots,D_k$, it follows that $D_1,\dots,D_k$ do not fill $C_2$.   For each $i=1,\dots,k$, let $h'_i$ be the height of $D_i$, and let $n_i$ (resp. $m_i$) be the number of intersections of the core curve of $D_i$ with the core curve of $C_2$ (resp. the core curve of $C_3$). Clearly, $n_i \leq m_i$. Let $\ell_i$ denote the length of the horizontal saddle connection $i$, and let $\ell_4 = \ell_1 + \ell_2$.  Remark that the height of $D$ is $\ell_3$.  We have
$$
P(C_2,\cD)=\cfrac{(n_1h'_1+\dots+n_kh'_k)h_2}{\ell_4h_2}=\cfrac{n_1h'_1+\dots+n_kh'_k}{\ell_4}
$$
\noindent and
$$
P(C_3,\cD)=\cfrac{(m_1h'_1+\dots+m_kh'_k+\ell_3)h_3}{(\ell_4+\ell_3)h_3}=\cfrac{m_1h'_1+\dots+m_kh'_k+\ell_3}{\ell_4+\ell_3}.
$$

\noindent Since the cylinders in $\cD$ do not fill $C_2$, we have $P(C_2,\cD) < 1$. Thus
$$
P(C_2,\cD) < \cfrac{n_1h'_1+\dots + n_kh'_k + \ell_3}{\ell_4+\ell_3} \leq \cfrac{m_1h'_1+\dots+m_kh'_k+\ell_3}{\ell_4+\ell_3}=P(C_3,\cD).
$$
We then get a contradiction, from  which we conclude that $C_3$ cannot be $\cM$-parallel to $C_2$.

Next, let us assume that $C_3$ is $\cM$-parallel to $C_1$. Let $D$ be the vertical simple cylinder in $C_3$ described above. Since $C_1$ is $\cM$-parallel to $C_3$, there must exist a vertical cylinder $D'$ in the equivalence class of $D$ that crosses $C_1$.  However, any vertical cylinder crossing $C_1$ must cross $C_2$ and $C_4$.  Hence, $P(D',\cC) > 0$, but $P(D,\cC) = 0$.  The contradiction implies that $C_3$ is free.

Let us now assume that $C_1 \in \cC$.  Let $4$ be the saddle connection between $C_2$ and $C_3$, and let $5$ be the saddle connection between $C_3$ and $C_4$.  Cut off $C_3$ from $M$ and identify $4$ and $5$, we then get a surface $M_1 \in \cH(2)$ which is the union of $C_1,C_2,C_4$.  By construction,  $M_1$ is decomposed into two horizontal cylinders with distinguished simple closed geodesic $c$ in the larger cylinder which is the identification of $4$ and $5$ together with a marked point $x$ on this geodesic that corresponds to the singularity in the boundary of $C_3$.

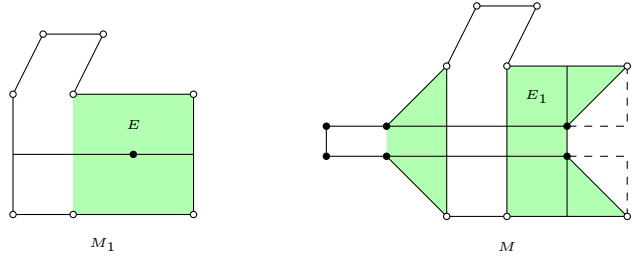
\begin{figure}[htb]
\centering
\begin{minipage}[t]{0.4\linewidth}
\centering
\begin{tikzpicture}[scale=0.4]
\fill[green!30] (2,4) -- (2,0) -- (6,0) -- (6,4) -- cycle;

\draw (0,4) -- (0,0) -- (6,0) -- (6,4) -- (2,4) -- (3,6) -- (1,6) -- cycle;
\draw (0,2) -- (6,2);

\foreach \x in {(0,4),(0,0),(1,6),(2,4), (2,0),(3,6),(6,4),(6,0)} \filldraw[fill=white] \x circle (3pt);
\filldraw[fill=black] (4,2) circle (3pt);

\draw (4,3) node {\tiny $E$} (3,-1) node {\tiny $M_1$};
\end{tikzpicture}
\end{minipage}
\begin{minipage}[t]{0.4\linewidth}
\centering
\begin{tikzpicture}[scale=0.4]
\fill[green!30] (-2,3) -- (-2,2) -- (0,0) -- (0,5) -- cycle;
\fill[green!30] (2,5) -- (2,0) -- (6,0) -- (4,2) -- (4,3) -- (6,5) -- cycle;

\draw (-4,3) -- (-4,2) -- (-2,2) -- (0,0) -- (6,0) -- (4,2) -- (4,3) -- (6,5) -- (2,5) -- (3,7) -- (1,7) -- (0,5) -- (-2,3) -- cycle;
\draw (-2,3) -- (4,3);
\draw (0,5) -- (0,0) (-2,2) -- (4,2) (2,5) -- (2,0) (4,5) -- (4,3) (4,2) -- (4,0) ;

\draw[dashed] (6,5) -- (6,3) -- (4,3);
\draw[dashed] (6,0) -- (6,2) -- (4,2);

\foreach \x in {(0,5),(0,0),(1,7),(2,5), (2,0),(3,7),(6,5),(6,0)} \filldraw[fill=white] \x circle (3pt);
\foreach \x in {(-4,3), (-4,2), (-2,3), (-2,2), (4,3), (4,2)} \filldraw[fill=black] \x circle (3pt);
\draw (3,4) node {\tiny $E_1$} (2,-1) node {\tiny $M$};
\end{tikzpicture}
\end{minipage}
 \caption{Decomposition in the vertical direction of $M$}
 \label{fig:4IIOC:vert:cyl}
\end{figure}

Using $U = \{\left( \begin{smallmatrix} 1 & t \\ 0 & 1 \end{smallmatrix} \right), \, t \in \bR\}$, we can assume that there is a vertical cylinder $E$ contained in the larger cylinder of $M_1$ crossing the saddle connection $2$. Since $C_3$ is free, we can freely twist it so that there is a vertical simple cylinder $D$ contained in $C_3$. Glue $C_3$ back to $M_1$, we see that either $E$ extends to a vertical cylinder in $M$ crossing $C_2,C_3,C_4$ (if the marked point $x\not\in E$), or $E$ splits into two vertical cylinders (if $x \in E$) (see Figure~\ref{fig:4IIOC:vert:cyl}). In both cases let $E_1$ be one of the cylinders arising from $E$. Note that $E_1$ crosses $C_2$ and $C_4$, but not $C_1$. Since we assume that $C_1$ is $\cM$-parallel to $C_2$, there must exist a vertical cylinder $E_2$ in the equivalence class of $E_1$ that crosses $C_1$.  Let $E_2$ cross $C_i$, $n_i$ times.
  Some observations are in order.  Since $C_2$ and $C_4$ are homologous, $n_2 = n_4$.  The assumption that the simple cylinder $D \subset C_3$ is vertical implies $n_3 = n_2$ as well.  This allows us to compute $P(E_1,\{C_3\}) = P(E_2,\{C_3\})$ to get
$$\frac{h_3}{h_2+h_3+h_4} = \cfrac{n_2 h_3}{n_1h_1+n_2h_2+n_2h_3+n_2h_4}.$$
Hence, $n_1 = 0$, which implies the contradictory conclusion that no cylinder $\cM$-parallel to $E_1$ passes through $C_1$.  We can then conclude that $C_1$ is free.

Recall that $h_i\zeta_i$ is the tangent vector (in $T^{\bR}_M \cM$) to the path which is the deformation of $M$ by twisting $C_i$ and fixing the rest of the surface.  In particular $\zeta_i$ evaluates to zero on the core curves of all the cylinders (as the lengths of the core curves are unchanged along this path).

Observe that the core curves of $C_1,C_2,C_3$  span a Lagrangian in $H_1(M,\bZ)$. Since $C_1, C_3$ are free, and  $C_2,C_4$ are $\cM$-parallel, Theorem~\ref{thm:Wright:Cyl:Def} implies that $\Twi_1, \Twi_3$, and $h_2\Twi_2+h_4\Twi_4$ belong to $T^\bR_M \cM$. Since $\Twi_i$ vanishes on the core curves of $C_1,\dots,C_4$, the cup product of $H^1(M,\bR)$ vanishes on the subspace $L \subset p(T^\bR_M\cM)$ spanned by $p(\Twi_1),p(\Twi_3),p(h_2\Twi_2+h_4\Twi_4)$. Since $\cM$ has rank two,   we have $\dim_\bR L \leq 2$.

Let $\Twi'_1=p(\Twi_1),\Twi'_2=p(h_2\Twi_2+h_4\Twi_4), \Twi'_3=p(\Twi_3)$. Let $s_i$ be an oriented saddle connection in $C_i$ joining the zero in the bottom border to the zero in the top border of $C_i$. Let $\gamma_1=s_1, \gamma_2=s_2\cup s_4$, and $\gamma_3=s_3$, then $\gamma_1,\gamma_2,\gamma_3$ represent elements of $H_1(M,\bZ)$. Remark  that we have
$$
\Twi'_i(\gamma_j)=0 \text{  if  } i \neq j \text{ and }  \Twi'_i(\gamma_i) \neq 0
$$

\noindent which implies that $\Twi'_1, \Twi'_2, \Twi'_3$ are independent, hence $\dim_\bR L=3$. We then have a contradiction, which means that Case 4.II.OC) cannot occur.
\end{proof}

\bigskip

\subsubsection{Case 4.II.OA)}

\begin{lemma}
\label{lm:Odd:4IIOA}
Let $M$ be a horizontally periodic translation surface in a rank two affine manifold $\cM \subset \cH^{\rm odd}(2,2)$.  If $M$ satisfies Case 4.II.OA), then there exists $M' \in \cM$ horizontally periodic satisfying Case 4.II.OB).
\end{lemma}
\begin{proof}
Clearly we have $c_1$ and $c_3$ are homologous, therefore $\alpha_1=\alpha_3$ ($\alpha_i$ is the element of $(T^\bR_M\cM)^\ast$ defined by $c_i$). Let $\cC$ denote the equivalence class of $C_1$. Since we have at least two equivalence classes, at least one of $C_2,C_4$ does not belong to $\cC$. Without loss of generality,  let us assume that $C_4 \not\in \cC$. Set $V:=T^\bR_M\cM \subset H^1(M,\Sigma,\bR)$.  If $C_2$ and $C_4$ are both free, then by similar arguments to the proof of Case 4.II.OC), we see that  the projection of the family $ \{h_1\Twi_1+h_3\Twi_3, \Twi_2,\Twi_4\} $ spans a Lagrangian subspace of dimension three in $p(V)$, which is impossible. Therefore, we only need to consider  two cases:
 \begin{itemize}
  \item[a)]  $C_2 \in \cC$ and $C_4$ is free. In this case we have $\ker\alpha_1=\ker\alpha_2=\ker\alpha_3 \neq \ker\alpha_4 \subset T^\bR_M\cM$. Let $v \in \ker\alpha_1\setminus \ker\alpha_4$ be a vector such that $\alpha_1(v)=0$ and $\alpha_4(v)=1$. Consider  the surface $M':=M+\imath\eps v$, with $\eps \in \bR$ small enough.  We will identify surfaces in a neighborhood of $M$ with elements of $H^1(M,\Sigma;\bR)$ via the period mapping.  With this identification, if $s$ is a saddle connection or a core curve of a cylinder on $M$, by $s(M)$ and $s(M')$ we will mean the periods of $s$ in $M$ and $M'$ respectively.   Note that the cylinders $C_i$ persist under all small deformations of $M$ in the sense that closed curves are sent to closed curves even if they do not remain parallel (or horizontal).
  We have
  $$
  c_j(M')=c_j(M) \in \bR, \, j=1,2,3, \text{ since } v(c_j)=0,
  $$
 \noindent and
 $$
 c_4(M')=c_4(M)+\imath\eps \not\in \bR,
 $$
which implies that $c_1,c_2,c_3$ are horizontal in $M'$, but $c_4$ is not.  It follows that the lower boundary of $C_3$ and the upper boundary of $C_1$ become broken lines in $M'$ (see Figure~\ref{fig:Case:4IIOA:deform1}). Let $C'_i$ denote the cylinder in $M'$ with core curve $c_i$. Note also that the cylinders  $C'_1,C'_2,C'_3,C'_4$ do not fill $M'$.  The complement of $C'_1\cup C'_2\cup C'_3$ is a slit torus that properly contains $C'_4$. Since $\cM$ is defined over $\bQ$, we can choose $M$ and $\eps$ so that $M'$ is a square-tiled surface, in particular $M'$ admits a cylinder decomposition in the horizontal direction, which means that  the slit torus is horizontally periodic. It is not difficult to check that the cylinder decomposition of $M'$ belongs to the Case 4.II.OC). Hence, by Lemma~\ref{lm:Odd:4IIOC} this case does not occur.
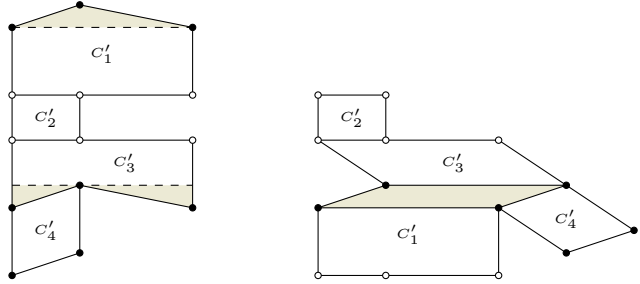
\begin{figure}[htb]
\begin{minipage}[t]{0.4\linewidth}
\centering
\begin{tikzpicture}[scale=0.3]
\fill[black!40!yellow!20] (0,11) -- (8,11) -- (3,12) -- cycle;
\fill[black!40!yellow!20] (0,4) -- (0,3) -- (3,4) -- cycle;
\fill[black!40!yellow!20] (3,4) -- (8,4) -- (8,3) -- cycle;
\draw (0,0) -- (3,1) -- (3,4) -- (8,3) -- (8,6) -- (3,6) -- (3,8) -- (8,8) -- (8,11) -- (3,12) -- (0,11) -- cycle;
\draw (0,8) -- (3,8) (0,6) -- (3,6) (0,3) -- (3,4);
\draw[dashed] (0,11 ) -- (8,11) (0,4) -- (8,4);
\foreach \x in {(0,11),(0,3),(0,0), (3,12),(3,4), (3,1), (8,11), (8,3)} \filldraw[fill=black] \x circle (4pt);
\foreach \x in {(0,8),(0,6),(3,8),(3,6),(8,8),(8,6)} \filldraw[fill=white] \x circle (4pt);

\draw (4,9) node[above] {\tiny $C'_1$} (1.5,7) node {\tiny $C'_2$} (5,5) node {\tiny $C'_3$} (1.5,2) node {\tiny $C'_4$};
\end{tikzpicture}
\end{minipage}
\begin{minipage}[t]{0.4\linewidth}
\centering
\begin{tikzpicture}[scale=0.3]
\fill[black!40!yellow!20] (0,3) -- (8,3) -- (11,4) -- (3,4) -- cycle;

\draw (0,0) -- (8,0) -- (8,3) -- (11,1) -- (14,2) -- (11,4) -- (8,6) -- (3,6) -- (3,8) -- (0,8) -- (0,6) -- (3,4) -- (0,3) -- cycle;
\draw (0,6) -- (3,6) (3,4) -- (11,4) (0,3) -- (8,3) -- (11,4);

\foreach \x in {(0,3),(3,4),(8,3),(11,4),(11,1),(14,2)} \filldraw[fill=black] \x circle (4pt);
\foreach \x in {(0,8),(0,6),(0,0),(3,8),(3,6),(3,0),(8,6),(8,0)} \filldraw[fill=white] \x circle (4pt);

\draw (4,1) node[above] {\tiny $C'_1$} (6,5) node {\tiny $C'_3$} (1.5,7) node {\tiny $C'_2$} (11,2.5) node {\tiny $C'_4$} ;
\end{tikzpicture}
\end{minipage}
\caption{Deformation of a surface in Case 4.II.OA) to Case 4.II.OC)}
\label{fig:Case:4IIOA:deform1}
\end{figure}

 \item[b)] $C_2$ and $C_4$ are $\cM$-parallel. By assumption, there exists $\mu \in \bR_{>0}$ such that $\alpha_2=\mu\alpha_4$.  Since $C_4 \not\in \cC$, there exists a vector $v \in V$ such that $\alpha_1(v)=0$ and $\alpha_4(v)=1$. Note that in this case we have $\alpha_2(v) =\mu\alpha_4(v)\neq 0$. Consider the deformation $M'=M+\imath\eps v$, with $\eps \in \bR$ small enough. Let $C'_i$ be the cylinders in $M'$ corresponding to $C_i$. Using the same argument as above, we see that $C'_1$ and $C'_3$ are horizontal, but $C'_2$ and $C'_4$ are not. Note that in this case $C'_1$ and $C'_3$ are simple cylinders, and the complement of $C'_1\cup C'_3$ is the disjoint union of two slit tori containing $C'_2$ and $C'_4$ (see Figure~\ref{fig:Case:4IIOA:deform2}). We can choose $M$ and $\eps$ so that $M'$ is a square-tiled surface.  Hence, $M'$ is horizontally periodic. It is not difficult to check that the diagram of the cylinder decomposition of $M'$ belongs to Case 4.II.OB).  The lem
 ma is then proved.

\begin{figure}[htb]
\begin{minipage}[t]{0.4\linewidth}
\centering
\begin{tikzpicture}[scale=0.35]
\fill[blue!30!] (0,10) -- (6,10) -- (3,11) -- cycle;
\fill[blue!30] (0,3) -- (0,2) -- (3,3) -- cycle;
\fill[blue!30] (3,3) -- (6,2) -- (6,3) -- cycle;

\fill[black!20!yellow!30] (0,5) -- (6,5) -- (3,6) -- cycle;
\fill[black!20!yellow!30] (0,8) -- (0,7) -- (3,8) -- cycle;
\fill[black!20!yellow!30] (3,8) -- (6,8)  -- (6,7) -- cycle;

\draw (0,0) -- (3,1) -- ( 3,3) -- ( 6,2) -- (6,5) -- (3,6) -- (3,8) -- (6,7) -- (6,10) -- (3,11) -- (0,10) -- cycle;
\foreach \x in {(0,10), (0,8), (0,5), (0,3)} \draw \x -- +(6,0);
\foreach \x in {(0,7), (0,5), (0,2)} \draw \x -- +(3,1);

\foreach \x in {(0,10),(0,2), (0,0), (3,11), (3,3), (3,1) ,(6,10), (6,2)} \filldraw[fill=black] \x circle (4pt);
\foreach \x in {(0,7),(0,5), (3,8), (3,6),(6,7), (6,5)} \filldraw[fill=white] \x circle (4pt);

\draw (4,9) node {\tiny $C'_1$} (1.5,6.5) node {\tiny $C'_2$} (4,4) node {\tiny $C'_3$} (1.5,1.5) node {\tiny $C'_4$};
\end{tikzpicture}
\end{minipage}
\begin{minipage}[t]{0.4\linewidth}
\centering
\begin{tikzpicture}[scale=0.35]
\fill[blue!30!] (0,0) -- (6,0) -- (9,1) -- (3,1) -- cycle;
\fill[black!20!yellow!30] (0,3) -- (3,4) -- (9,4) -- (6,3) -- cycle;

\draw (0,0) -- (6,0) -- ( 9,-1) -- (12,0) -- (9,1) -- (6,3) -- (9,2) -- (12,3) -- (9,4) -- (6,6) -- (0,6) -- (3,4) -- (0,3) -- (3,1) -- cycle;

\draw (3,4) -- (9,4) -- (6,3) -- (0,3) (3,1) -- (9,1) -- (6,0);

\foreach \x in {(0,6), (0,0), (3,1), (6,6), (6,0), (9,1) , (9,-1), (12,0)} \filldraw[fill=black] \x circle (4pt);
\foreach \x in {(0,3),(3,4), (6,3), (9,4), (9,2), (12,3)} \filldraw[fill=white] \x circle (4pt);

\draw (5,5) node {\tiny $C'_1$} (9,3) node {\tiny $C'_2$} (5,2) node {\tiny $C'_3$} (9,0) node {\tiny $C'_4$};

\end{tikzpicture}
\end{minipage}
\caption{Deformation of a surface in Case 4.II.OA) to Case 4.II.OB)}
\label{fig:Case:4IIOA:deform2}
\end{figure}
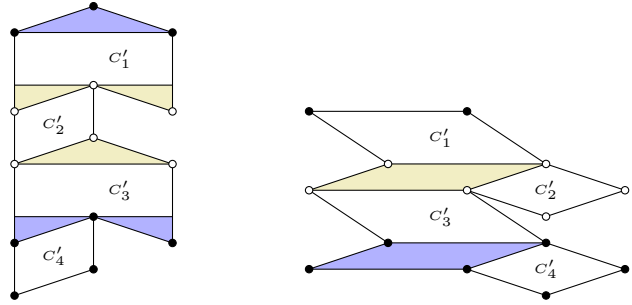
\end{itemize}
\end{proof}

%

\bigskip

\subsubsection{Case 4.II.OB)}

\begin{lemma}
\label{lm:Odd:4IIOB}
Let $M$ be a horizontally periodic translation surface in a rank two affine manifold $\cM \subset \Odd$.  If $M$ satisfies Case 4.II.OB), then $C_1$ and $C_3$ are isometric, and either
\begin{itemize}
 \item[a)] $\cM = \prymodd$, or

 \item[b)] $\cM = \coversodd \subset \prymodd$, in which case $C_2$ and $C_4$ are isometric.
\end{itemize}

In particular, $\coversodd$ is connected.
\end{lemma}

\begin{proof}

We first notice that $c_2$ and $c_4$ are homologous, therefore $C_2$ and $C_4$ are $\cM$-parallel, let us denote by $\cC$ their equivalence class.
\bigskip

\noindent \underline{Claim 1:} Neither $C_1$ nor $C_3$ belongs to $\cC$.
\begin{proof}[Proof of the claim]
Assume that $C_1\in \cC$, then $C_3$ must be free since we have at least two equivalence classes. There exists a transverse simple cylinder $D$ contained in $C_1$ which can be made vertical by using $\{\left( \begin{smallmatrix} 1 & t \\ 0 & 1 \end{smallmatrix}\right), \ t \in \bR\}$.  By assumption, there must exists another vertical cylinder $D'$ in the equivalence class of $D$ that crosses $C_2$. But any vertical cylinder crossing $C_2$ must cross $C_3$, hence $P(D',\{C_3\}) >0$ while $P(D,\{C_3\})=0$, and we have a contradiction. The same arguments apply for $C_3$.
\end{proof}

\medskip

\noindent \underline{Claim 2:} $C_1$ and $C_3$ are $\cM$-parallel.
\begin{proof}[Proof of the claim]
Indeed, if this is not the case then both $C_1$ and $C_3$ are free by Claim 1. It follows that the projection of $\{\Twi_1,\Twi_3,h_2\Twi_2+h_4\Twi_4\}$ in absolute cohomology spans a Lagrangian subspace of $p(T^\bR_M\cM)$ of dimension three, which is impossible.
\end{proof}

\medskip

\noindent \underline{Claim 3:} $C_1$ and $C_3$ are isometric.
\begin{proof}[Proof of the claim]
Observe that the closures of $C_1$ and $C_3$ are two slit tori, which will be denoted by $T_1$ and $T_3$, respectively.  Let $D_1$ be any simple cylinder in $T_1$ that does not meet the slit. Since $C_3$ is $\cM$-parallel to $C_1$ by Claim 2, there must exist a cylinder $D_3$ that is $\cM$-parallel to $D_1$ that crosses $C_3$. Remark that $P(D_1,\cC)=0$, hence $P(D_3,\cC)=0$, which implies that $D_3$ is a simple cylinder contained in $T_3$. Moreover, from Proposition~\ref{CylinderPropProp} we have
$$
P(D_1,\{C_1,C_3\})=P(D_3,\{C_1,C_3\}) \Rightarrow P(D_1,C_1) = P(D_3,C_3).
$$
The claim then follows from \cite[Lemma 8.1]{AulicinoNguyenWright}.
\end{proof}

\medskip

Since $C_1$ and $C_3$ are isometric, we see that $M:=(X,\omega)$ admits an involution $\tau$ that exchanges $C_1$ and $C_3$, fixes $C_2$, $C_4$, and satisfies $\tau^*\omega=-\omega$ (that is the derivative of $\tau$ in local charts defined by the flat metric structure is given by $-\Id$). Observe that $\tau$  exchanges the singularities of $M$ and has four fixed points, two in $C_2$ and two in $C_4$. It is now a routine to check that $M$ is a double cover of a quadratic differential in $\cQ(4,-1^4)$. We can now conclude that $M \in \prymodd$.

\medskip

\noindent \underline{Claim 4:} $\cM \subseteq \prymodd$.
\begin{proof}[Proof of the claim]
Let $v$ be a vector in $V:=T^\bR_M \cM$. For $\eps \in \bR$ small enough, the deformation $M':=M+\eps v$ of $M$ also admits a cylinder decomposition in the horizontal direction with the same diagram.  The argument above implies that $M' \in \prymodd$. Thus, we have $T^{\bR}_M \cM \subseteq T^\bR_M \prymodd$, and $\cM \subseteq \prymodd$.
\end{proof}

\medskip
Recall that $\Twi_i$ is the vector in $H^1(X,\Sigma,\bR)$ tangent to the deformations of $M$ by twisting $C_i$ alone. Remark that deformations of $M$ by twisting and stretching simultaneously $C_1$ and $C_3$, or $C_2$ and $C_4$ remain in $\cM$. Since  $C_1$ and $C_2$ are isometric, twisting simultaneously $C_1$ and $C_3$  gives deformations along the line defined by the vector $\Twi_1+\Twi_3$.  On the other hand, twisting simultaneously $C_2$ and $C_4$ gives deformations  along the line defined by the vector $h_2\Twi_2+h_4\Twi_4$.

\medskip

\noindent\underline{Claim 5:} If $\Twi_2$ or $\Twi_4$ belongs to $T^\bR_{M}\cM$, then $\cM=\prymodd$.
\begin{proof}[Proof of the claim]
Set $ V:= T^ \bR_M\cM \subset H^1(M,\Sigma; \bR)$. By the observation above, we already have $h_2\Twi_2+h_4\Twi_4 \in V$. If $\Twi_2$ or $\Twi_4$ belongs to $V$, then both of them belong to $V$. But if $\Twi_2$ belongs to $T_M^\bR\cM$ then we can freely stretch and shear $C_2$ while keeping the rest of the surface unchanged to obtain other surfaces in $\cM$ (see \cite{WrightCylDef}, Lemma 2.3). It follows that we can collapse $C_2$ to degenerate a saddle connection in $C_2$ to a point and obtain a surface $N \in \overline{\cM}\cap\cH(4)$. From Proposition~\ref{prop:RankkImpRankkBd} and Theorem~\ref{NWANWThm}, we know that $\overline{\cM}\cap\cH(4)$ contains $\prym$. Thus $\dim_\bC\cM \geq  \dim_\bC\prym +1 =5$. On the other hand, we know that $\cM \subset \prymodd$, and $\dim_\bC\prymodd =5$. Therefore, we can conclude that $\dim_\bC\cM= 5$, and $\cM=\prymodd$.
\end{proof}

\medskip

Recall that, we use the notations $h(C_i), \wth(C_i), \twist(C_i)$ to denote the height, circumference, and twist of $C_i$.  We can normalize so that $h(C_2)=\wth(C_2)=1$, and $\twist(C_2)=0$.  Note that we also have $\wth(C_4)=1$. Set $h=h(C_4)$ and $x:=\twist(C_4)$.

\medskip

\noindent \underline{Claim 6:} If $x \not \equiv  0 \mod \bZ$, then $\cM=\prymodd$.

\begin{proof}[Proof of the claim]
Without loss of generality, let $x \in (0,1)$.  Collapse the cylinders $\{C_2, C_4\}$ to get a translation surface in $\cH(4)$.  By Proposition \ref{prop:RankkImpRankkBd}, the resulting translation surface lies in a rank two affine manifold, which must be $\prym$.  Hence, $\prym$ lies in the boundary of $\cM$.  By algebraicity, this implies that $\dim_{\bC}(\cM) = 5$, so Claim 4 implies that $\cM$ is a full dimensional subset of $\prymodd$. But it is well-known that $\prymodd$ is connected (see \cite[Theorem 1.2]{LanneauComponents}). Therefore, we must have $\cM=\prymodd$.
\end{proof}

\bigskip

\noindent \underline{Claim 7:} If $\cM \neq \prymodd$, then $C_2$ and $C_4$ are isometric, and $\cM \subset \coversodd$. 
\begin{proof}[Proof of the claim]

By Claim 6 we must have $\twist(C_4)=0$.  Let $\hat{C}_2$ and $\hat{C}_4$ denote the cylinders obtained by applying the matrix $\left( \begin{smallmatrix} 1 & 1 \\ 0 & 1  \end{smallmatrix} \right)$ to $C_2$ and $C_4$, respectively. We have $\twist(\hat{C}_2)=\twist(C_2)=0$, and $\twist(\hat{C}_4)=h$. By Claim 6, we must have $h \equiv 0 \mod \bZ$, which implies that $h \in \bN$, or equivalently $h(C_4)/h(C_1) \in \bN$. Since the roles of $C_2$ and $C_4$ can be exchanged, we can conclude that $h(C_1)=h(C_4)$, and it follows that $C_2$ and $C_4$ are isometric.

Observe that we then have an automorphism $f : M \rightarrow M$ of order two, that exchanges $C_1$ with $C_3$,  $C_2$ with $C_4$, and satisfies the following condition: the derivative of $f$ in local charts defined by the flat metric structure is given by $\Id$. Identify $M$ with a pair $(X,\omega)$, we have $f^*\omega=\omega$. Observe that $f$ has no fixed points on $X$.

Let $Y:=X/\langle f \rangle$, and $\pi: X\rightarrow Y$ be the natural  projection. Since $f$ has order two and no fixed points, $\pi$ is an unramified double cover, which implies that $Y$ is a surface of genus two. Since $f^*\omega=\omega$, there exists a holomorphic $1$-form $\eta$ on $Y$ such that $\omega =\pi^*\eta$. By definition $\omega$ has two zeros of order two, thus $\eta$ must have a single zero of order two, which means that $(Y,\eta)\in \cH(2)$ hence $M=(X,\omega) \in \coversodd$.

Now, for any vector $v \in T^\bR_M\cM$, and $\eps \in \bR$ small enough, the surface $M_\eps:=M+\eps v$ also belongs to $\cM$ and admits a cylinder decomposition in the horizontal direction with the same diagram as $M$. The previous claims imply that $M_\eps\in \coversodd$. Therefore we have $T^{\bR}_M\cM \subset T^\bR_M\coversodd$, and $\cM \subset \coversodd$.
\end{proof}

\bigskip

\noindent \underline{Claim 8:} $\coversodd$ is connected. If $\cM \varsubsetneq \prymodd$, then  $\cM=\coversodd$.
\begin{proof}[Proof of the claim]
By definition, each component of $\coversodd$ is a proper rank two affine submanifold of dimension four in $\prymodd$. From the  Lemmas~\ref{lm:Odd:4IOA}, \ref{lm:Odd:4IOB}, \ref{lm:Odd:4IIOC}, \ref{lm:Odd:4IIOA}, we know that this affine submanifold contains a surface horizontally periodic satisfying  Case 4.II.OB). Claim 3, and Claim 7 then imply that $C_1$ and $C_2$ are isometric to $C_3$ and $C_4$, respectively. Observe that if $C_1$ is isometric to $C_3$ and $C_2$ is isometric to $C_4$, then $M\in \coversodd$. Clearly the set of horizontally periodic surfaces satisfying Case 4.II.OB) with this additional condition is connected. Thus we can conclude that $\coversodd$ is connected.

Suppose that $\cM \varsubsetneq \prymodd$ is a rank two affine submanifold. Since $\prymodd$ is connected, we must have $\dim_\bC\cM < \dim_\bC\prymodd =5$. Since $\cM$ has rank two, we must have $\dim_\bC \cM \geq 4$, from which we conclude that $\dim_\bC \cM= 4$.

From Claim 7, we know that $\cM \subset \coversodd$. Since $\dim_\bC\cM=\dim_\bC\coversodd=4$, $\cM$ must be a component of $\coversodd$. But $\coversodd$ is connected, therefore we have $\cM=\coversodd$.
The claim is then proved
\end{proof}
The proof of the lemma is now complete.
\end{proof}

\bigskip


\appendix

\section{Dual Graphs and $4$-Cylinder Diagrams in $\cH(m,n)$, $m+n=4$}\label{sec:list:4cyl:diag}
In this section we give the complete list of $4$-cylinder diagrams for surfaces in genus three having two singularities.  To obtain this list our approach  is to use the dual graphs. In this situation, the dual graphs have exactly four vertices and six edges. We classify them by the valencies at their vertices (the total valency is $12$). Given an integral vector $(n_1,\dots,n_4)$ such that $n_1\leq \dots \leq n_4$, and $n_1+\dots +n_4=12$, we look for undirected graphs satisfying this condition on the valencies. For each of the graphs, we then look for  orientations of the edges such that no forbidden configuration occurs (see Section~\ref{sec:intro:DG}). Finally, we choose for each vertex a cyclic ordering on the set of incoming edges, and a cyclic ordering on the set of outgoing edges. It turns out this procedure can be carried out ``by hand'' for the case $\cH(m,n),\, m+n=4$, as all except one admissible  undirected graph provide us with a unique corresponding cylinder di
 agram (that is, there is only at most one way to chose the orientations of the edges such that there is no forbidden configuration). As a result, we found $11$ $4$-cylinder diagrams. The exercise to determine which component the diagram belongs to is left to the reader.

\begin{proposition}
 \label{prop:DG:complete}
 Let $M$ be a horizontally periodic surface with four horizontal cylinders in a stratum $\cH(m,n), \, m+n=4$. Then the cylinder diagram of $M$ is given in Figures~\ref{fig:DG:2226}, \ref{fig:DG:2235}, \ref{fig:DG:2244}, \ref{fig:DG:2334}. In all of these figures, the cylinder $C_i$ corresponds to the vertex $v_i$ of the dual graph for $i=1,\dots,4$.
\end{proposition}

\begin{figure}[htb]
\centering
\begin{minipage}[t]{0.3\linewidth}
\centering
\begin{tikzpicture}[scale=0.35, inner sep=0.2mm, vertex/.style={circle, draw=black, fill=blue!30, minimum size=1mm},>= stealth]
\node (center) at (0,0) [vertex] {\tiny $v_4$};
\node (above) at (0,3) [vertex] {\tiny $v_1$};
\node (left) at (-2.5,-1) [vertex] {\tiny $v_2$};
\node (right) at (2.5,-1) [vertex] {\tiny $v_3$};

\draw[->] (above) to[out=240, in=120] (center) ;
\draw[->] (center) to [out=60, in=-60] (above);
\draw[->] (center) to [out=180, in= 60] (left);
\draw[->] (left) to [out=0, in=240] (center);
\draw[->] (center) to [out=0, in=120] (right);
\draw[->] (right) to [out=180, in=-60] (center);

\draw (0,-3) node {\small DG.4.I};

\end{tikzpicture}
\end{minipage}
\begin{minipage}[t]{0.3\linewidth}
\centering
\begin{tikzpicture}[scale=0.4]
\draw (0,0)--(0,4)--(2,4)--(2,2)--(4,2)--(4,4)--(6,4)--(6,0)--(4,0)--(4,-2)--(2,-2)--(2,0)--cycle;
\foreach \x in {(0,0),(0,4),(2,4),(4,4),(6,4),(6,0),(4,0),(2,0)} \filldraw[fill=black] \x circle (3pt);
\foreach \x in {(0,2),(2,-2),(2,2),(4,-2),(4,2),(6,2)} \filldraw[fill=white] \x circle (3pt);

\draw(1,4) node[above] {\tiny $3$};
\draw(3,2) node[above] {\tiny $2$};
\draw(5,4) node[above] {\tiny $1$};
\draw(1,0) node[below] {\tiny $1$};
\draw(3,-2) node[below] {\tiny $2$};
\draw(5,0) node[below] {\tiny $3$};

\draw (5,3) node {\tiny $C_1$} (3,-1) node {\tiny $C_2$} (1,3) node {\tiny $C_3$} (3,1) node {\tiny $C_4$};
\end{tikzpicture}
\end{minipage}
\begin{minipage}[t]{0.3\linewidth}
\centering
\begin{tikzpicture}[scale=0.4]
\draw (0,0)--(0,4)--(2,4)--(2,2)--(4,2)--(4,4)--(6,4)--(6,0)--(4,0)--(4,-2)--(2,-2)--(2,0)--cycle;
\foreach \x in {(0,0),(0,4),(2,4),(4,4),(6,4),(6,0),(4,0),(2,0)} \filldraw[fill=black] \x circle (3pt);
\foreach \x in {(2,2),(4,2),(4,-2),(2,-2),(0,2),(6,2)} \filldraw[fill=white] \x circle (3pt);
\draw(1,4) node[above] {\tiny $1$};
\draw(3,2) node[above] {\tiny $2$};
\draw(5,4) node[above] {\tiny $3$};
\draw(1,0) node[below] {\tiny $1$};
\draw(3,-2) node[below] {\tiny $2$};
\draw(5,0) node[below] {\tiny $3$};

\draw (5,3) node {\tiny $C_3$} (3,-1) node {\tiny $C_2$} (1,3) node {\tiny $C_1$} (3,1) node {\tiny $C_4$};
\end{tikzpicture}
\end{minipage}
\caption{Admissible dual graphs for the valency vector $(2,2,2,6)$.}
\label{fig:DG:2226}
\end{figure}
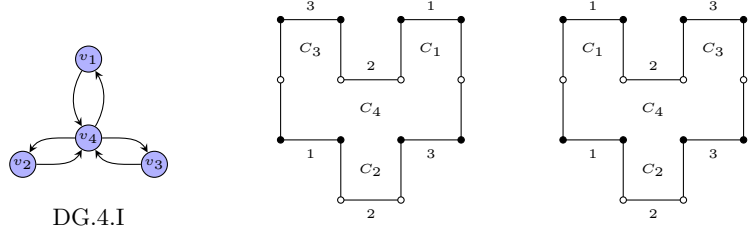

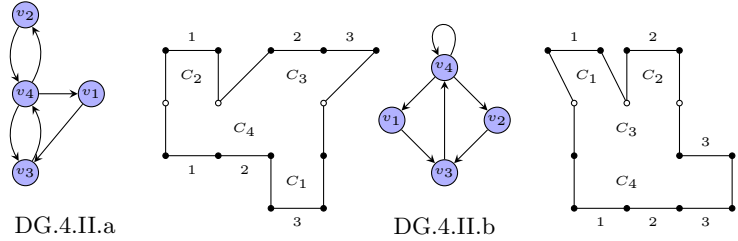
\begin{figure}[htb]
 \centering
 \begin{minipage}[t]{0.2\linewidth}
  \centering
  \begin{tikzpicture}[scale=0.35, inner sep=0.2mm, vertex/.style={circle, draw=black, fill=blue!30, minimum size=1mm},>= stealth]
   \node (center) at (0,0) [vertex] {\tiny $v_4$};
   \node (above) at (0,3) [vertex] {\tiny $v_2$};
   \node (below) at (0,-3) [vertex] {\tiny $v_3$};
   \node (right) at (2.5,0) [vertex] {\tiny $v_1$};

   \draw[->] (above) to[out=240, in=120] (center) ;
   \draw[->] (center) to[out=60, in=-60] (above) ;
   \draw[->] (center) to[out=240, in=120] (below) ;
   \draw[->] (below) to[out=60, in=-60] (center) ;
   \draw[->] (center) to (right);
   \draw[->] (right) to (below);

   \draw (1.5,-5) node {\small DG.4.II.a};
  \end{tikzpicture}
\end{minipage}
\begin{minipage}[t]{0.2\linewidth}
 \centering
 \begin{tikzpicture}[scale=0.35]
 \draw (0,6) -- (0,2) -- (4,2) -- (4,0) -- (6,0) -- (6,4) -- (8,6) -- (4,6) -- (2,4) -- (2,6) -- cycle;
 \foreach \x in {(0,6), (0,2), (2,6),(2,2),(4,6),(4,2),(4,0), (6,6), (6,2),(6,0), (8,6)} \filldraw[fill=black] \x circle (3pt);
 \foreach \x in {(0,4),(2,4),(6,4)} \filldraw[fill=white] \x circle (3pt);

 \draw (1,6) node[above] {\tiny $1$} (1,2) node[below] {\tiny $1$} (5,6) node[above] {\tiny $2$} (3,2) node[below] {\tiny $2$} (7,6) node[above] {\tiny $3$} (5,0) node[below] {\tiny $3$};

 \draw (1,5) node {\tiny $C_2$} (5,5) node {\tiny $C_3$} (3,3) node {\tiny $C_4$} (5,1) node {\tiny $C_1$};

 \end{tikzpicture}
 \end{minipage}
\begin{minipage}[t]{0.2\linewidth}
  \centering
  \begin{tikzpicture}[scale=0.35, inner sep=0.2mm, vertex/.style={circle, draw=black, fill=blue!30, minimum size=1mm},>= stealth]
   \node (above) at (0,2) [vertex] {\tiny $v_4$};
   \node (below) at (0,-2) [vertex] {\tiny $v_3$};
   \node (right) at (2,0) [vertex] {\tiny $v_2$};
   \node (left) at (-2,0) [vertex] {\tiny $v_1$};

   \draw[->] (above) ..controls +(1,2) and +(-1,2).. (above) ;
   \draw[->] (above) to (left);
   \draw[->] (above) to (right) ;
   \draw[->] (left) to (below) ;
   \draw[->] (right) to (below);
   \draw[->] (below) to (above);

   \draw (0,-4) node {\small DG.4.II.b};
  \end{tikzpicture}
\end{minipage}
\begin{minipage}[t]{0.2\linewidth}
\centering
\begin{tikzpicture}[scale=0.35]
\draw (-1,6) -- (0,4) -- (0,0) -- (6,0) -- (6,2) -- (4,2) -- (4,6) -- (2,6) -- (2,4) -- (1,6)  -- cycle;
 \foreach \x in {(-1,6), (0,2), (0,0),(1,6), (2,6), (2,0), (4,6),(4,2),(4,0), (6,2), (6,0)} \filldraw[fill=black] \x circle (3pt);
  \foreach \x in {(0,4),(2,4),(4,4)} \filldraw[fill=white] \x circle (3pt);

  \draw (0,6) node[above] {\tiny $1$} (1,0) node[below] {\tiny $1$} (3,6) node[above] {\tiny $2$} (3,0) node[below] {\tiny $2$} (5,2) node[above] {\tiny $3$} (5,0) node[below] {\tiny $3$};

  \draw (0.5,5) node {\tiny $C_1$} (3,5) node {\tiny $C_2$} (2,3) node {\tiny $C_3$} (2,1) node {\tiny $C_4$};

\end{tikzpicture}
\end{minipage}
\caption{Admissible dual graphs for the valency vector $(2,2,3,5)$.}
\label{fig:DG:2235}
\end{figure}

\begin{figure}[htb]
 \centering
 \begin{minipage}[t]{0.2\linewidth}
  \centering
  \begin{tikzpicture}[scale=0.35, inner sep=0.2mm, vertex/.style={circle, draw=black, fill=blue!30, minimum size=1mm},>= stealth]
   \node (above) at (0,3) [vertex] {\tiny $v_1$};
   \node (below) at (0,-3) [vertex] {\tiny $v_2$};
   \node (center1) at (0,1) [vertex] {\tiny $v_3$};
   \node (center2) at (0,-1) [vertex] {\tiny $v_4$};

   \draw[->] (above) to[out=240, in=120] (center1) ;
   \draw[->] (center1) to[out=60, in=-60] (above) ;

   \draw[->] (center1) to[out=240, in=120] (center2) ;
   \draw[->] (center2) to[out=60, in=-60] (center1) ;

   \draw[->] (center2) to[out=240, in=120] (below) ;
   \draw[->] (below) to[out=60, in=-60] (center2) ;

   \draw (0,-5) node {\small DG.4.III.a};
   \end{tikzpicture}
\end{minipage}
\begin{minipage}[t]{0.25\linewidth}
\centering
\begin{tikzpicture}[scale=0.35]
\draw  (0,8) -- (0,4) -- (2,4) -- (2,2) -- (4,2) -- (4,0) -- (6,0) -- (6,4) -- (4,4) -- (4,6) -- (2,6) -- (2,8) -- cycle;
\foreach \x in {(0,8),(0,4),(2,8),(2,4),(4,4),(4,0),(6,4),(6,0)} \filldraw[fill=black] \x circle (3pt);
\foreach \x in {(0,6),(2,6),(2,2),(4,6),(4,2),(6,2)} \filldraw[fill=white] \x circle (3pt);
\draw (1,8) node[above] {\tiny $3$} (1,4) node[below] {\tiny $3$} (3,6) node[above] {\tiny $2$} (3,2) node[below] {\tiny $2$} (5,4) node[above] {\tiny $1$} (5,0) node[below] {\tiny $1$};

\draw (1,7) node {\tiny $C_1$} (5,1) node {\tiny $C_2$} (3,5) node {\tiny $C_3$} (3,3) node {\tiny $C_4$};
\end{tikzpicture}
\end{minipage}
\begin{minipage}[t]{0.2\linewidth}
  \centering
  \begin{tikzpicture}[scale=0.35, inner sep=0.2mm, vertex/.style={circle, draw=black, fill=blue!30, minimum size=1mm},>= stealth]
   \node (above) at (0,3) [vertex] {\tiny $v_1$};
   \node (below) at (0,-3) [vertex] {\tiny $v_3$};
   \node (center1) at (0,1) [vertex] {\tiny $v_4$};
   \node (center2) at (0,-1) [vertex] {\tiny $v_2$};

   \draw[->] (above) to[out=240, in=120] (center1) ;
   \draw[->] (center1) to[out=60, in=-60] (above) ;

   \draw[->]  (center2) to (center1);
   \draw[->] (below) to (center2)  ;

   \draw[->] (below) ..controls +(-1,-2) and +(1,-2) .. (below) ;
   \draw[->] (center1) to[out=-30, in=30] (below) ;

   \draw (0,-6) node {\small DG.4.III.b};
   \end{tikzpicture}
\end{minipage}
\begin{minipage}[t]{0.25\linewidth}
\centering
 \begin{tikzpicture}[scale=0.35]
 \draw (-1,6) -- (0,4) -- (0,2) -- (2,2) -- (2,0) -- (6,0) -- (6,2) -- (4,2) -- (4,6) -- (2,6) -- (2,4) -- (1,6) -- cycle;
 \foreach \x in {(-1,6), (0,2), (1,6),(2,6), (2,2), (2,0), (4,6),(4,2),(4,0), (6,2),(6,0)} \filldraw[fill=black] \x circle (3pt);
  \foreach \x in {(0,4),(2,4),(4,4)} \filldraw[fill=white] \x circle (3pt);

  \draw (0,6) node[above] {\tiny $1$} (1,2) node[below] {\tiny $1$} (3,6) node[above] {\tiny $2$} (3,0) node[below] {\tiny $2$} (5,2) node[above] {\tiny $3$} (5,0) node[below] {\tiny $3$};

  \draw (0.5,5) node {\tiny $C_1$} (3,5) node {\tiny $C_2$} (2,3) node {\tiny $C_4$} (3,1) node {\tiny $C_3$};

\end{tikzpicture}
\end{minipage}
\begin{minipage}[t]{0.2\linewidth}
  \centering
  \begin{tikzpicture}[scale=0.35, inner sep=0.2mm, vertex/.style={circle, draw=black, fill=blue!30, minimum size=1mm},>= stealth]
   \node (above) at (0,2) [vertex] {\tiny $v_3$};
   \node (below) at (0,-2) [vertex] {\tiny $v_4$};
   \node (left) at (-2,0) [vertex] {\tiny $v_1$};
   \node (right) at (2,0) [vertex] {\tiny $v_2$};

   \draw[->] (above) to[out=240, in=120] (below) ;
   \draw[->] (below) to[out=60, in=-60] (above) ;

   \draw[->] (above) to (left) ;
   \draw[->] (right) to (above) ;

   \draw[->] (left) to (below) ;
   \draw[->] (below) to (right) ;

   \draw (0,-4) node {\small DG.4.III.c};
   \end{tikzpicture}
\end{minipage}
\begin{minipage}[t]{0.25\linewidth}
\centering
\begin{tikzpicture}[scale=0.35]
\draw (0,0)--(0,8)--(4,8)--(4,6)--(2,6)--(2,4)--(4,4)--(4,2)--(2,2)--(2,0)--cycle;

\foreach \x in {(0,0),(0,8),(2,8), (4,8),(4,2),(2,2),(2,0),(0,2)} \filldraw[fill=black] \x circle (3pt);
\foreach \x in {(4,6),(2,6),(2,4),(4,4),(0,4),(0,6)} \filldraw[fill=white] \x circle (3pt);

\draw(1,8) node[above] {\tiny 1};
\draw(3,4) node[above] {\tiny 3};
\draw(3,8) node[above] {\tiny 2};
\draw(1,0) node[below] {\tiny 1};
\draw(3,2) node[below] {\tiny 2};
\draw(3,6) node[below] {\tiny 3};

\draw (1,7) node {\tiny $C_3$} (1,5) node {\tiny $C_1$} (1,3) node {\tiny $C_4$} (1,1) node {\tiny $C_2$};

\end{tikzpicture}
\end{minipage}
\begin{minipage}[t]{0.2\linewidth}
  \centering
  \begin{tikzpicture}[scale=0.35, inner sep=0.2mm, vertex/.style={circle, draw=black, fill=blue!30, minimum size=1mm},>= stealth]
   \node (above) at (0,2) [vertex] {\tiny $v_3$};
   \node (below) at (0,-2) [vertex] {\tiny $v_4$};
   \node (left) at (-1,0) [vertex] {\tiny $v_1$};
   \node (right) at (1,0) [vertex] {\tiny $v_2$};

   \draw[->] (above) ..controls +(1,2) and +(-1,2) .. (above) ;
   \draw[->] (below) .. controls +(-1,-2) and +(1,-2).. (below) ;

   \draw[->] (above) to (left) ;
   \draw[->] (right) to (above) ;

   \draw[->] (left) to (below) ;
   \draw[->] (below) to (right) ;

   \draw (0,-5) node {\small DG.4.III.d};
   \end{tikzpicture}
\end{minipage}
\begin{minipage}[t]{0.25\linewidth}
\centering
\begin{tikzpicture}[scale=0.35]
\draw (0,0)--(0,8)--(4,8)--(4,6)--(2,6)--(2,4)--(4,4)--(4,2)--(2,2)--(2,0)--cycle;
\foreach \x in {(0,0),(0,8),(2,8), (4,8),(4,6),(2,6),(2,0), (0,6)} \filldraw[fill=black] \x circle (3pt);
\foreach \x in {(0,4),(0,2),(2,4),(2,2),(4,4),(4,2)} \filldraw[fill=white] \x circle (3pt);
\draw(1,8) node[above] {\tiny 1};
\draw(3,4) node[above] {\tiny 3};
\draw(3,8) node[above] {\tiny 2};
\draw(1,0) node[below] {\tiny 1};
\draw(3,2) node[below] {\tiny 3};
\draw(3,6) node[below] {\tiny 2};

\draw (1,7) node {\tiny $C_3$} (1,5) node {\tiny $C_1$} (1,3) node {\tiny $C_4$} (1,1) node {\tiny $C_2$};

\end{tikzpicture}
\end{minipage}

\caption{Admissible dual graphs for the valency vector $(2,2,4,4)$.}
\label{fig:DG:2244}
\end{figure}
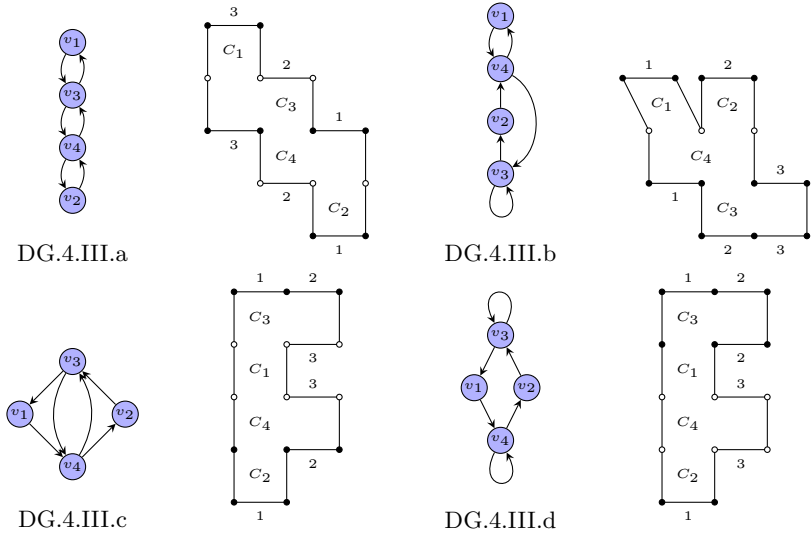

\medskip

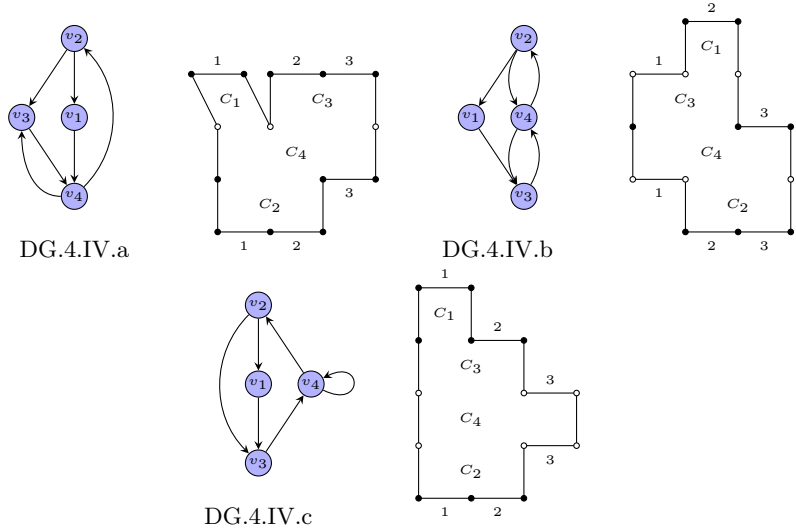
\begin{figure}[htb]
 \centering
 \begin{minipage}[t]{0.2\linewidth}
  \centering
  \begin{tikzpicture}[scale=0.35, inner sep=0.2mm, vertex/.style={circle, draw=black, fill=blue!30, minimum size=1mm},>= stealth]
   \node (above) at (0,3) [vertex] {\tiny $v_2$};
   \node (below) at (0,-3) [vertex] {\tiny $v_4$};
   \node (center) at (0,0) [vertex] {\tiny $v_1$};
   \node (left) at (-2,0) [vertex] {\tiny $v_3$};

   \draw[->] (above) to (center) ;
   \draw[->] (above) to (left);
   \draw[->] (center) to (below) ;
   \draw[->] (below) to[out=45, in=-45] (above) ;

   \draw[->] (left) to (below) ;
   \draw[->] (below) to[out=180,in=-90] (left);

   \draw (0,-5) node {\small DG.4.IV.a};
   \end{tikzpicture}
\end{minipage}
\begin{minipage}[t]{0.25\linewidth}
\centering
\begin{tikzpicture}[scale=0.35]
\draw (-1,6) -- (0,4) -- (0,0) -- (4,0) -- (4,2) -- (6,2) -- (6,6) -- (2,6) -- (2,4) -- (1,6)  -- cycle;
 \foreach \x in {(-1,6), (0,2), (0,0),(1,6), (2,6), (2,0), (4,6),(4,2),(4,0), (6,6), (6,2)} \filldraw[fill=black] \x circle (3pt);
  \foreach \x in {(0,4),(2,4),(6,4)} \filldraw[fill=white] \x circle (3pt);

  \draw (0,6) node[above] {\tiny $1$} (1,0) node[below] {\tiny $1$} (3,6) node[above] {\tiny $2$} (3,0) node[below] {\tiny $2$} (5,6) node[above] {\tiny $3$} (5,2) node[below] {\tiny $3$};

  \draw (0.5,5) node {\tiny $C_1$} (4,5) node {\tiny $C_3$} (3,3) node {\tiny $C_4$} (2,1) node {\tiny $C_2$};

\end{tikzpicture}
\end{minipage}
\begin{minipage}[t]{0.2\linewidth}
  \centering
  \begin{tikzpicture}[scale=0.35, inner sep=0.2mm, vertex/.style={circle, draw=black, fill=blue!30, minimum size=1mm},>= stealth]
   \node (above) at (0,3) [vertex] {\tiny $v_2$};
   \node (below) at (0,-3) [vertex] {\tiny $v_3$};
   \node (center) at (0,0) [vertex] {\tiny $v_4$};
   \node (left) at (-2,0) [vertex] {\tiny $v_1$ };

   \draw[->] (above) to[out=240, in=120] (center) ;
   \draw[->] (center) to[out=60,in=-60] (above);

   \draw[->] (center) to[out=240, in=120] (below) ;
   \draw[->] (below) to[out=60, in=-60] (center) ;

   \draw[->] (above) to (left);
   \draw[->] (left) to (below) ;

   \draw (-1,-5) node {\small DG.4.IV.b};
   \end{tikzpicture}
\end{minipage}
\begin{minipage}[t]{0.25\linewidth}
\centering
\begin{tikzpicture}[scale=0.35]
\draw (0,0)--(0,2)--(-2,2)--(-2,6)--(0,6)--(0,8)--(2,8)--(2,6)--(2,4)--(4,4)--(4,0)--cycle;
\foreach \x in {(0,0),(0,8),(2,8),(2,4),(4,4),(4,0),(-2,4),(2,0)} \filldraw[fill=black] \x circle (3pt);
\foreach \x in {(-2,6),(-2,2), (0,6),(0,2), (2,6), (4,2)} \filldraw[fill=white] \x circle (3pt);
\draw(-1,6) node[above] {\tiny $1$};
\draw(1,8) node[above] {\tiny $2$};
\draw(3,4) node[above] {\tiny $3$};
\draw(-1,2) node[below] {\tiny $1$};
\draw(1,0) node[below] {\tiny $2$};
\draw(3,0) node[below] {\tiny $3$};

\draw (1,7) node {\tiny $C_1$} (0,5) node {\tiny $C_3$} (1,3) node {\tiny $C_4$} (2,1) node {\tiny $C_2$};
\end{tikzpicture}
\end{minipage}
\begin{minipage}[t]{0.2\linewidth}
  \centering
  \begin{tikzpicture}[scale=0.35, inner sep=0.2mm, vertex/.style={circle, draw=black, fill=blue!30, minimum size=1mm},>= stealth]
   \node (above) at (0,3) [vertex] {\tiny $v_2$};
   \node (below) at (0,-3) [vertex] {\tiny $v_3$};
   \node (center) at (0,0) [vertex] {\tiny $v_1$};
   \node (right) at (2,0) [vertex] {\tiny $v_4$};

   \draw[->] (above) to (center) ;
   \draw[->] (above) to[out=225, in=135] (below) ;
   \draw[->] (right) to (above);
   \draw[->] (center) to (below);
   \draw[->] (below) to  (right) ;
   \draw[->] (right) ..controls +(2,-1) and +(2,1) .. (right) ;

   \draw (0,-5) node {\small DG.4.IV.c};
   \end{tikzpicture}
\end{minipage}
\begin{minipage}[t]{0.25\linewidth}
\centering
\begin{tikzpicture}[scale=0.35]
\draw (0,0)--(0,8)--(2,8)--(2,6)--(4,6)--(4,4)--(6,4)--(6,2)--(4,2)--(4,0)--cycle;
\foreach \x in {(0,0),(0,8),(2,8),(2,6),(4,6),(2,0),(0,6),(4,0)} \filldraw[fill=black] \x circle (3pt);

\foreach \x in {(4,4),(6,4),(6,2),(4,2),(0,2),(0,4)} \filldraw[fill=white] \x circle (3pt);

\draw(1,8) node[above] {\tiny 1};
\draw(3,6) node[above] {\tiny 2};
\draw(5,4) node[above] {\tiny 3};
\draw(1,0) node[below] {\tiny 1};
\draw(3,0) node[below] {\tiny 2};
\draw(5,2) node[below] {\tiny 3};

\draw (1,7) node {\tiny $C_1$} (2,5) node {\tiny $C_3$} (2,3) node {\tiny $C_4$} (2,1) node {\tiny $C_2$};

\end{tikzpicture}
\end{minipage}

\caption{Admissible dual graphs for the valency vector $(2,3,3,4)$.}
\label{fig:DG:2334}
\end{figure}

\medskip

\begin{proof}
 Let $C_1,\dots, C_4$ be the horizontal cylinders of $M$. Let $\cD$ denote the dual graph of the cylinder decomposition of $M$, the vertices of $\cD$ are $v_1,\dots,v_4$, where $v_i$  corresponds to $C_i$. Let $n_i$ be the valency of $v_i$. We always choose the numbering so that $n_1\leq n_2 \leq n_3 \leq n_4$. Recall that  we must have $n_1+n_2+n_3+n_4=12$, therefore $n_1 \leq 3 \leq n_4$.

 We first claim that $n_1=2$. Clearly $n_1 \not= 1$.  If $n_1=3$, then $n_1 = n_2 = n_3 = n_4 = 3$, which means that all of the cylinders $C_1,\dots,C_4$ are semi-simple. For any $i$, let $w_i$ be the  circumference of $C_i$. Since $C_i$ is semi-simple, without loss of generality assume that its bottom side consists of a single saddle connection $\s_i$. Since $\s_i$ must be contained in the top of a cylinder $C_j$, and clearly the top of $C_j$ must contain another saddle connection, we derive $w_i < w_j$. It follows that for all $i \in \{1,\dots,4\}$, there exists $j\in \{1,\dots,4\}$ such that $w_i <w_j$, and we get a contradiction.

 From the previous claim, we have
$$\vec{n}:=(n_1,\dots,n_4) \in \{(2,2,2,6), (2,2,3,5), (2,2,4,4), (2,3,3,4)\}.$$

\noindent {\bf Case $\vec{n}=(2,2,2,6)$}.  Since two simple cylinders cannot be adjacent, the three vertices with valency equal to two must be attached to $v_4$.  Thus, we have the dual graph DG.4.I, and the cylinder diagram of $M$ is given by one of the two diagrams in Figure~\ref{fig:DG:2226}.

\medskip

\noindent {\bf Case $\vec{n}=(2,2,3,5)$}. First suppose that each of $v_1$ and $v_2$  is attached to a single vertex. If both $v_1$ and $v_2$ are attached to the same vertex, which must be $v_4$, then we would have a forbidden configuration at $v_3$. If $v_1$ or $v_2$ is attached to $v_3$, then we also have a forbidden configuration. Thus, at least one of $v_1,v_2$, say $v_1$, is connected to both $v_3$ and $v_4$.

If $v_2$ is attached to a single vertex, then this vertex must be $v_4$, and in this case $v_4$ and $v_3$ must be connected by two edges, we then get the dual graph DG.4.II.a. If $v_2$ is also connected to both $v_3$ and $v_4$, then there must be an edge joining $v_3$ to $v_4$ and a loop at $v_4$, thus we get the dual graph DG.4.II.b.

\medskip

\noindent{\bf Case $\vec{n}=(2,2,4,4)$}. In the case that each of $v_1,v_2$ is attached (by two edges) to a single vertex, we can assume that $v_1$ is attached to $v_3$, and $v_2$ to $v_4$. It follows that $v_3$ and $v_4$ must be connected by two edges, and we get the dual graph DG.4.III.a.

If $v_1$ is connected to both $v_3$ and $v_4$, and $v_2$ is attached to only one vertex, say $v_4$, then there must be an edge between $v_3$ and $v_4$, and a loop at $v_3$, thus we get the dual graph DG.4.III.b.

If each of  $v_1,v_2$ are connected to both $v_3$ and $v_4$, then we have two cases: either there are two edges between $v_3$ and $v_4$, in which case we get the dual graph DG.4.III.c, or there is a loop at $v_3$ and a loop at $v_4$, which gives the dual graph DG.4.III.d.

\medskip

\noindent {\bf Case $\vec{n}=(2,3,3,4)$}. If $v_1$ is attached to a single vertex, this vertex must be $v_4$. If $v_4$ is connected to only one of $v_2,v_3$, then we would have the forbidden configuration at $v_2$ or $v_3$, thus there must be an edge between $v_4$ and $v_2$, and an edge between $v_4$ and $v_3$. It follows that $v_2$ and $v_3$ are connected by two edges, and there is a unique dual graph satisfying this case. But by case-by-case inspection, it turns out that for any choice of orientation for the edges, we always have a forbidden configuration, therefore this case is excluded.

Now assume that $v_1$ is connected to two other vertices.

\begin{itemize}
\item[$\bullet$] If $v_1$ is connected to $v_2$ and $v_3$, then we have two cases: either $v_2$ and $v_3$ are not connected, which means that they are each connected to $v_4$ by two edges, and we get the dual graph DG.4.IV.b.  Otherwise, $v_2$ and $v_3$ are connected by an edge, which implies that $v_4$ is connected to $v_2$ and $v_3$ by one edge, and there is a loop at $v_4$, which yields the dual graph DG.4.IV.c.

 \item[$\bullet$] If $v_1$ is connected to $v_4$ and a vertex with valency three, says $v_2$, then we also have two cases: either $v_2$ is connected to $v_4$ by one edge, which implies that $v_3$ is connected to $v_2$ by one edge and to $v_4$ by two edges, and we get the dual graph DG.4.IV.a, or $v_2$ is not connected to $v_4$, which implies that $v_3$ is connected to $v_2$ by two edges, to $v_4$ by one edge, and there is a loop at $v_4$. In the latter case, an inspection of the orientations of the edges shows that we always have a forbidden configuration.  Hence, there is no corresponding cylinder diagram for this case.
\end{itemize}

It is not difficult to get the complete list of cylinder diagrams from the  list of admissible undirected dual graphs. The details are left for the reader.
\end{proof}

\bibliography{fullbibliotex}{}

\end{document}